\documentclass[article,onefignum,onetabnum]{siamart190516}

\usepackage{lipsum}
\usepackage{amsfonts}
\usepackage{graphicx}
\usepackage{epstopdf}
\usepackage{algorithmic}
\ifpdf
  \DeclareGraphicsExtensions{.eps,.pdf,.png,.jpg}
\else
  \DeclareGraphicsExtensions{.eps}
\fi

\newsiamremark{remark}{Remark}
\newsiamremark{hypothesis}{Hypothesis}
\crefname{hypothesis}{Hypothesis}{Hypotheses}
\newsiamthm{claim}{Claim}

\headers{Semi-Lagrangian MGRIT coarse-grid operators}{H. De Sterck, R. D. Falgout, O. A. Krzysik}

\title{Fast multigrid reduction-in-time for advection via modified semi-Lagrangian coarse-grid operators\thanks{Submitted to the editors March 24, 2022.
\funding{This work was performed under the auspices of the U.S. Department of Energy by Lawrence Livermore National Laboratory under Contract DE-AC52-07NA27344 (LLNL-JRNL-833088).
This work was supported in part by the U.S. Department of Energy, Office of Science, Office of Advanced Scientific Computing Research, Applied Mathematics program, and by NSERC of Canada.
}}}

\author{H. De Sterck\thanks{Department of Applied Mathematics, University of Waterloo, Waterloo, Ontario, Canada} 
  (\email{hans.desterck@uwaterloo.ca}).
\and R. D. Falgout\thanks{Center for Applied Scientific Computing, Lawrence Livermore National Laboratory, Livermore, California, USA 
  (\email{falgout2@llnl.gov}).}
\and O. A. Krzysik\thanks{Department of Applied Mathematics, University of Waterloo, Waterloo, Ontario, Canada
  (\email{okrzysik@uwaterloo.ca}, \url{https://orcid.org/0000-0001-7880-6512}). The work of this author was partially supported by an Australian Government Research Training Program Scholarship.}
  }

\usepackage{amsopn}
\DeclareMathOperator{\diag}{diag} 

\usepackage{xcolor}
\usepackage{bm} 

\usepackage{enumitem} 

\usepackage{comment} 

\usepackage{mathtools} 
\mathtoolsset{centercolon} 

\usepackage{amssymb} 
\usepackage{makecell} 
\usepackage{multirow} 
\usepackage{makecell} 
\usepackage{pifont}      

\usepackage{boldline}  

\usepackage{grffile} 

\Crefname{subsection}{Section}{Sections}

\usepackage{tikz}
\usetikzlibrary{calc}
\usepackage{graphbox}

\renewcommand{\d}[0]{\ensuremath{\operatorname{d}\!}} 

\definecolor{darkgreen}{rgb}{0, 0.5, 0}

\usepackage[us,12hr]{datetime} 

\begin{document}




\maketitle

\begin{abstract}
	Many iterative parallel-in-time algorithms have been shown to be highly efficient for diffusion-dominated partial differential equations (PDEs), but are inefficient or even divergent when applied to advection-dominated PDEs. 
	We consider the application of the multigrid reduction-in-time (MGRIT) algorithm to linear advection PDEs. 
	The key to efficient time integration with this method is using a coarse-grid operator that provides a sufficiently accurate approximation to the the so-called ideal coarse-grid operator.
	For certain classes of semi-Lagrangian discretizations, we present a novel semi-Lagrangian-based coarse-grid operator that leads to fast and scalable multilevel time integration of linear advection PDEs. 
	The coarse-grid operator is composed of a semi-Lagrangian discretization followed by a correction term, with the correction designed so that the leading-order truncation error of the composite operator is approximately equal to that of the ideal coarse-grid operator.
	Parallel results show substantial speed-ups over sequential time integration for variable-wave-speed advection problems in one and two spatial dimensions, and using high-order discretizations up to order five.
	The proposed approach establishes the first practical method that provides small and scalable MGRIT iteration counts for advection problems.
\end{abstract}



\section{Introduction}
\label{sec:introduction}

Traditionally, the solutions of initial-value partial differential equation (PDE) problems are approximated numerically via the sequential process of time-stepping, an approach motivated by the temporally causal nature of the solution itself.
However, there also exist many time-parallel methods for simulating these PDEs, and these can often yield reductions in wall-clock time relative to time-stepping.
A history and broad survey of the field of parallel-in-time methods can be found in the review \cite{Gander_2015}, with the later review \cite{Ong_Schroder_2020} providing a summary of more recent developments. 

Presently, parallel-in-time methods are not widely used for large-scale PDE simulations, with the traditional technique of time-stepping remaining the standard.
As supercomputer architectures continue to use more and more cores, parallel-in-time methods are likely to become essential for circumventing sequential time-stepping bottlenecks. 
However, one issue likely to limit their practicability is a lack of robustness for hyperbolic PDEs, or more broadly for advection-dominated problems.
In particular, this is true for the multigrid reduction-in-time (MGRIT) algorithm \cite{Falgout_etal_2014}, which is the algorithm that we focus on in this work.

MGRIT enables time-parallelism by applying local time-stepping in parallel across a sequence of temporally coarsened grids combined with a global coarse-grid correction on a coarsest grid.
The well-known time-parallel Parareal algorithm \cite{Lions_etal_2001} can be interpreted as a special case of MGRIT that uses only two levels and a specific choice for the relaxation scheme \cite{Gander_etal_2018}, and thus our work also applies to Parareal.
MGRIT is iterative, and its convergence properties hinge on how well the coarse-grid problem approximates the fine-grid problem.
The default technique for developing the coarse-grid problem is to \textit{rediscretize} the fine-grid PDE on the coarsened mesh, and this tends to work excellently for diffusion-dominated problems
\cite{Lions_etal_2001, 
Bal_Maday_2002, 
Falgout_etal_2014, 
Falgout_etal_2017_nonlin,
Falgout_etal_2019}. 
Conversely, it is well-documented that convergence tends to be substantially worse for advection-dominated PDEs when using rediscretization, or closely related techniques \cite{Chen_etal_2014, Dai_Maday_2013, Dobrev_etal_2017, Gander_Vandewalle_2007, Gander_2008, Hessenthaler_etal_2018, Howse_etal_2019, Howse2017_thesis, Nielsen_etal_2018, Ruprecht_2018, Schmitt_etal_2018, Steiner_etal_2015, DeSterck_etal_2019, DeSterck_etal_2021, KrzysikThesis2021}.

Unsurprisingly, it has also been observed that MGRIT convergence may deteriorate substantially when transitioning from the diffusion-dominated regime of an advection-diffusion problem to the advection-dominated regime \cite{Steiner_etal_2015,Howse2017_thesis,Schmitt_etal_2018,Ruprecht_2018}. In fact, this observation served as the motivation for an idea that was explored in \cite{Schroder_2018}, with some preliminary numerical experiments showing that MGRIT convergence of certain advection discretizations can be improved substantially by adding judiciously chosen amounts of numerical dissipation on the fine grid.

While we focus on multigrid-in-time methods in this paper, we note that some other parallel-in-time methods have shown potential for solving advection-dominated or wave-related problems.
Some examples include using reduction-based algebraic multigrid on the space-time system \cite{Sivas_etal_2021}, block-circulant preconditioning techniques applied to the space-time system \cite{Liu_Wu_2020}, and direct solution techniques based on diagonalizing in the time direction \cite{Gander_etal_2019}.

In this paper, we develop a novel coarse-grid operator leading to the fast MGRIT solution of linear advection problems, building on our earlier work \cite{DeSterck_etal_2021} of optimizing coarse-grid operators for constant-wave-speed advection problems.
In that work, we showed that it is possible to achieve fast MGRIT convergence on advection problems when using a carefully chosen coarse-grid operator. 
The coarse-grid operator proposed here is based on a semi-Lagrangian discretization, and is applied to certain classes of semi-Lagrangian discretizations on the fine grid; unlike the coarse-grid operators in \cite{DeSterck_etal_2021}, the coarse-grid operator is practically computable, and not limited only to constant-wave-speed problems.
Employing a semi-Lagrangian coarse-grid operator was tested in \cite{Schmitt_etal_2018}; however, solver convergence was not robust with respect to the strength of advection in the advection-diffusion test problem considered.
In fact, we will show even for the case of constant-wave-speed advection that simply employing a coarse-grid semi-Lagrangian operator does not result in effective convergence.

Our semi-Lagrangian coarse-grid operator is designed to provide a more accurate approximation to the ideal coarse-grid operator than basic rediscretization, using a truncation error approach to more faithfully match the coarse operator to the ideal coarse operator. Our motivation stems in part from the realization that MGRIT convergence issues for advection-dominated problems are in some sense analogous to the well-known issues that plague the multigrid solution of steady state advection-dominated problems \cite{Brandt_1981,Brandt_Yavneh_1993,Yavneh_1998,
Trottenberg_etal_2001,Wan_Chan_2003,Bank_etal_2006}. 
In particular, the idea of increasing the order of accuracy of the coarse-grid operator for steady state advection relative to the fine-grid operator, as first proposed in \cite{Yavneh_1998} and pointed out to us by the author, is a key insight and inspiration for our work. 
Developing such a coarse-grid operator in the MGRIT context involves additional stability challenges due to the coarsening being in only one coordinate direction rather than in all directions uniformly. The link between slow multigrid convergence in the case of steady state advection in multi-dimensional space, as analyzed in \cite{Yavneh_1998}, and slow multigrid convergence in the case of space-time hyperbolic problems with MGRIT, is investigated in detail in \cite[Chap. 3]{KrzysikThesis2021}.

This manuscript focuses on the MGRIT solution of semi-Lagrangian discretizations of advection problems; however, we stress that the principles developed here can be extended to the MGRIT solution of classical method-of-lines discretizations of hyperbolic problems, which is a subject of our on-going research.
Finally, we remark that this manuscript is based on the PhD thesis \cite[Chap. 4]{KrzysikThesis2021}, and as such, we refer the interested reader there for further details on this work.

The remainder of this manuscript is organized as follows. 
\Cref{sec:prelims} describes the semi-Lagrangian discretizations we consider, followed by an algorithmic description of MGRIT and motivating examples.
\Cref{sec:var} presents the proposed coarse-grid operator.
\Cref{sec:2D} generalizes the coarse-grid operator from one to two spatial dimensions. 
\Cref{sec:parallel} presents parallel results.
Concluding remarks are given in \Cref{sec:conclusions}.

\section{Preliminaries}
\label{sec:prelims}

\subsection{Semi-Lagrangian discretization of the one-dimensional advection equation}
\label{sec:discretization}

We consider semi-Lagrangian discretizations of one-dimensional, variable-wave-speed advection problems of the form
\begin{align} \label{eq:ad}
\frac{\partial u}{\partial t} + \alpha(x, t)\frac{\partial u}{\partial x} = 0,
\quad (x, t) \in \Omega \times (0, T],
\quad u(x, 0) = u_0(x),
\end{align}
for spatial domain $\Omega \subset \mathbb{R}$, and solution $u$ subject to periodic boundary conditions on $\partial \Omega$. 
Specifically, our numerical tests for this one-dimensional problem will use the initial condition $u_0(x) = \sin^4 (\pi x)$, and the spatial domain $\Omega = (-1, 1)$.
The MGRIT coarse-grid operators we propose for the multigrid solution of these problems rely on a detailed understanding of semi-Lagrangian methods, which we now describe.
We consider semi-Lagrangian discretizations of \eqref{eq:ad} that are based on finite differences, assume a sufficient degree of smoothness of the solution, and assume a readily computable wave-speed $\alpha(x,t)$ for any $x$ and $t$ is available. 
The reader is directed to \cite{Falcone_Ferretti_2014} and \cite[Sec. 7]{Durran_2010} for more detailed descriptions of semi-Lagrangian methods.

Consider discretizing the spatial domain $\Omega$ with a set of $n_x$ equidistant nodes $\bm{x} = (x_1, \ldots, x_{n_x})^\top$, $x_{i+1} = x_{i} + h$.
Furthermore, discretize the time interval $t \in [0, T]$ with an equidistant mesh of $n_t+1$ nodes $0 = t_0 < \cdots < t_{n_t} = T$, $t_{n+1} = t_n + \delta t$.
Given the vector $\bm{u}_n \approx u(\bm{x}, t_n)$, which represents the approximate solution of \eqref{eq:ad} at time $t_n$ at spatial mesh points $\bm{x}$, the semi-Lagrangian method advances this to a new approximation $\bm{u}_{n+1} \approx u(\bm{x}, t_{n+1})$ at $t_{n+1}$ as we now describe. 
The Lagrangian formulation of the PDE \eqref{eq:ad} is 
\begin{align}
\label{eq:LagForm}
\frac{\d }{\d t} \xi(t) = \alpha(\xi(t),t), 
\quad
\frac{\d}{\d t}u(\xi(t),t) = 0, 
\end{align}
in which the curves $(x,t) = (\xi(t), t)$ are the \textit{characteristics} of the PDE.
On some characteristic $\xi_i(t)$, the evolution equation---the latter equation in \eqref{eq:LagForm}---states that the PDE solution is constant.
Since the solution at the mesh point $(x,t) = (x_i, t_{n+1})$ is desired, one forces the characteristic $\xi_i(t)$ to pass through this point, such that the solution at this point is simply the solution at the foot of the characteristic at $t = t_n$ (see \Cref{fig:departure_diagram}).
To this end, define the local characteristic $\xi^{(t_n, \delta t)}_i(t)$ to be that which passes through the \textit{arrival point} $(x, t) = (x_i, t_{n+1})$. Then, the associated \textit{departure point} $(x, t) = \big( \xi^{(t_n, \delta t)}_i(t_n), t_n \big)$ is given by the solution of the final-value problem
\begin{align} \label{eq:depart_ODE}
\frac{\d}{\d t} \xi^{(t_n, \delta t)}_i(t) = \alpha \big( \xi^{(t_n, \delta t)}_i(t), t \big), \quad t \in [t_n, t_{n+1}), \quad \xi^{(t_n, \delta t)}_i(t_{n+1}) = x_i,
\end{align}
at $t = t_n$.
For general wave-speed functions $\alpha$, the initial-time solution of \eqref{eq:depart_ODE} cannot be found exactly, and thus needs to be approximated in some way.
In this paper, we approximate this departure point on the fine grid by integrating backwards using a single explicit Runge-Kutta (ERK) step of size $\delta t$.\footnote{On coarser levels in our multigrid-in-time hierarchy, we will explore several alternative possibilities rather than simply redeploying the ERK method with a coarse time step, which is not sufficiently accurate.}  
Suppose that the ERK method has a global accuracy of order $r$, such that each departure point is located with an accuracy of ${\cal O}(\delta t^{r+1})$ (e.g., forward Euler has $r = 1$).
In particular, we will consider ERK schemes with $r = 1,3,$ and $5$, with the Butcher tableaux for the specific schemes available from \cite[App. A.1]{KrzysikThesis2021}.

\begin{figure}[tb!]
\centerline{
\includegraphics[scale=0.75]{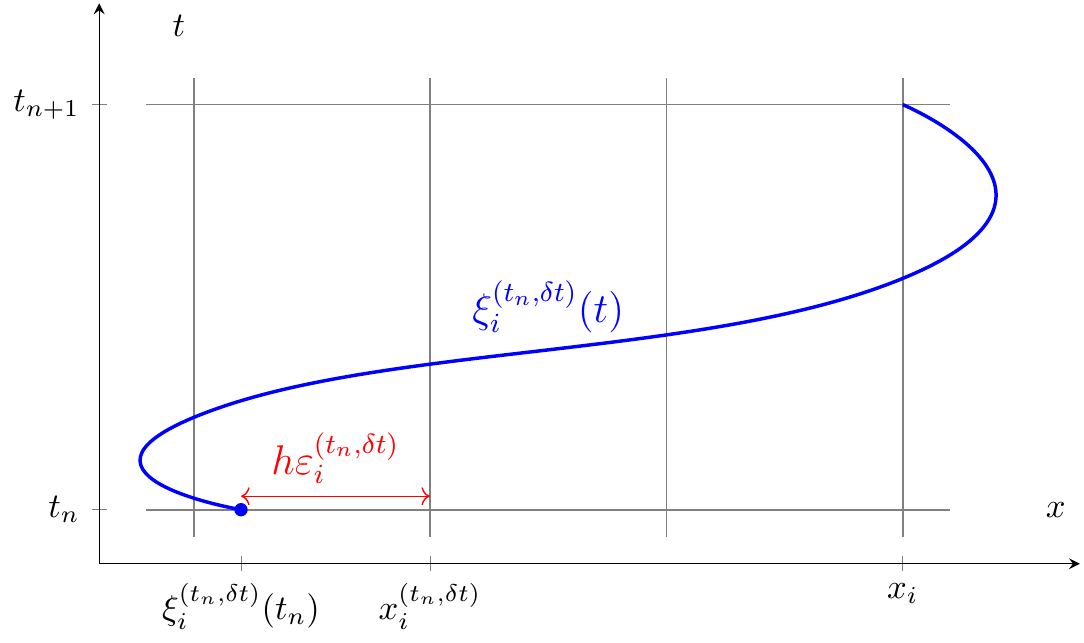}
}
\caption{
A local characteristic $\xi_{i}^{(t_n, \delta t)}(t)$ for $t \in [t_n, t_{n+1}]$ of the variable-wave-speed problem \eqref{eq:ad}.
By definition, the characteristic passes through the arrival point $(x,t)=(x_i, t_{n+1})$. 
The departure point (or foot) of the characteristic is its location at time $t = t_n$.
The departure point is decomposed into the sum of its east-neighboring mesh point and its distance from this point (see \eqref{eq:depart_decompose}).
\label{fig:departure_diagram}
}
\end{figure}

Upon (approximately) locating the departure point $\xi^{(t_n, \delta t)}_i(t)$, it will not, in general, coincide with a mesh point, yet the PDE approximation $\bm{u}_n$ is only available at mesh points.
To resolve this, an interpolating polynomial of at most degree $p \in \mathbb{N}$ is fit through the entries of $\bm{u}_n$ at the $p+1$ mesh nodes nearest to the departure point. 
Specifically, the departure point is decomposed as
\begin{align} \label{eq:depart_decompose}
\xi_i^{(t_n, \delta t)}(t_n)
\equiv
x_i^{(t_n, \delta t)} - h \varepsilon_i^{(t_n, \delta t)}, 
\quad  
\varepsilon_i^{(t_n, \delta t)} \in [0, 1),
\end{align}
in which $x_i^{(t_n, \delta t)}$ is the mesh node immediately east of the departure point $\xi_i^{(t_n, \delta t)}(t_n)$, and $\varepsilon_i^{(t_n, \delta t)}$ is its (mesh-normalized) distance from this point; see \Cref{fig:departure_diagram} for a schematic example.
The $p+1$ interpolation nodes are thus $\big\{ x_i^{(t_n, \delta t)} + h j \big\}_{j = -\ell(p)}^{r(p)}$, with $\ell(p)$ and $r(p)$---the west and east extents of the stencil, respectively---chosen so that the set of interpolation nodes represent the $p+1$ nearest neighbors of the departure point.\footnote{For interpolating polynomials of degree $p \geq 1$, both the west and east neighboring mesh points appear in the interpolation stencil of the departure point. Therefore, the choice made in \eqref{eq:depart_decompose} to write the departure point in terms of its east neighbor $x_i^{(t_n,\delta t)}$ is arbitrary in the sense that it could also have been written in terms of its west neighbor.} 
When $p$ is odd, this results in $\ell(p) = \tfrac{p+1}{2}$ and $r(p) = \tfrac{p-1}{2} = \ell(p) - 1$.
When $p$ is even, the stencil has a one-point bias, such that $\ell(p)$ and $r(p)$ depend on whether $\varepsilon^{(t_n,\delta t)}_i$ is larger than one half (we ignore this dependence in our notation).

Locating departure points for all arrival points $(x,t)=(\bm{x}, t_{n+1})$, then carrying out this piecewise polynomial interpolation constitutes a single time-step of the semi-Lagrangian discretization.
We denote the linear time-stepping operator corresponding to this semi-Lagrangian method as ${\Phi^{(t_n, \delta t)} = {\cal S}_{p,r}^{(t_n, \delta t)} \in \mathbb{R}^{n_x \times n_x}}$.

By tracking characteristics (with a sufficient level of accuracy), this discretization ensures the physical domain of dependence lies within the numerical domain of dependence.
Generally speaking, this is why semi-Lagrangian discretizations are typically free of a CFL constraint. 
However, ensuring that characteristics are tracked with sufficient accuracy can in certain circumstances lead to the imposition of a CFL-like constraint, although it is typically looser than that imposed by Eulerian schemes \cite{Huang_etal_2016,Smolarkiewicz_Pudykiewicz_1992}.
For sufficiently smooth solutions of \eqref{eq:ad}, ${\cal S}_{p,r}^{(t_n, \delta t)}$ has a global convergence rate of the form ${{\cal O} \big( \delta t^r + \frac{h^{p+1}}{\delta t} \big)}$, with the first term arising from approximately locating departure points, and the second from the polynomial interpolation at them \cite[Sec. 6.1.2]{Falcone_Ferretti_2014}.
Thus, while stability can be maintained with large time steps, they likely lead to a reduction in accuracy of the method. 
Nonetheless, the time-step size can be chosen based on accuracy requirements rather than stability requirements, unlike CFL-constrained Eulerian schemes for which stability is often the key factor in determining the time-step size. 
For example, this is why semi-Lagrangian methods see frequent use in numerical weather prediction \cite{Mengaldo_etal2018,Staniforth_Cote1991,Williamson2007}.
In any event, we attempt to balance temporal and spatial errors by using $\delta t \approx h$ and $r = p$, which is common in the literature when developing semi-Lagrangian discretizations \cite{Qiu_Shu_2011,Huang_etal_2016,Cai_etal_2017,Cai_etal_2021}.
%

\subsection{Algorithmic description of MGRIT}
\label{sec:MGRIT}

We now provide a brief algorithmic description of MGRIT as it applies to linear problems.
Consider the following fully discrete, one-step problem
\begin{align} \label{eq:one-step}
\bm{u}_{n+1} = \Phi^{(t_n, \delta t)} \bm{u}_n + \bm{g}_{n+1}, \quad n = 0, 1, \ldots, n_t - 1,
\end{align}
in which $\bm{u}_n$ denotes the approximate spatial solution of some time-dependent PDE at time $t = t_n$, and the initial condition $\bm{u}_0$ is given.
The vector $\bm{g}_n$ contains solution-independent information, and $\Phi^{(t_n, \delta t)}$ is the (linear) time-stepping operator.
Of course, \eqref{eq:one-step} may be solved straightforwardly with sequential time-stepping---solving for $\bm{u}_1$ given $\bm{u}_0$, then solving for $\bm{u}_2$ given $\bm{u}_1$, and so on, but this process is inherently sequential.
In contrast, MGRIT solves \eqref{eq:one-step} for all unknowns $\bm{u}_n, n = 1, \ldots n_t$, at once in parallel using multigrid reduction techniques, as we now detail.

Suppose that the underlying time grid in \eqref{eq:one-step} is equispaced with $t_{n+1} = t_{n} + \delta t$. 
Then, let a coarsening factor $m \in \mathbb{N}$ induce a coarse grid defined by taking every $m$th time point from the fine grid.
Define F-points as those appearing exclusively on the fine grid, and all other points as C-points.
Given an approximate solution of \eqref{eq:one-step}, define a C-relaxation of \eqref{eq:one-step} as updating C-point values to have zero residuals---this is achieved by time-stepping to each C-point from the F-point immediately before it.
Further, define an F-relaxation of \eqref{eq:one-step} as updating F-point values to have zero residuals---this is achieved by sequentially time-stepping from each C-point across the $m-1$ F-points immediately after it.

Given an approximate solution of \eqref{eq:one-step}, a two-level MGRIT iteration proceeds by pre-relaxation, coarse-grid correction, and post-relaxation.
The typical pre-relaxation scheme is FCF: An F-relaxation followed by a C-relaxation followed by a further F-relaxation. However, a single F-relaxation can also be used, as is often the case in the Parareal literature \cite{Gander_etal_2018}.
In this work, we exclusively use FCF-relaxation, since we find it often results in more robust convergence for the advection problems we consider---in some cases, it can even mean the difference between a quickly converging solver and a divergent solver.
Post-relaxation is simply an F-relaxation.
Note that the pre- and post-relaxation sweeps are highly parallelizable.
The coarse-grid problem is derived from the C-point Schur complement of the residual equation of \eqref{eq:one-step}.
More specifically, one computes an approximate error at C-points by solving the following problem containing fewer time points:
\begin{align} \label{eq:one-step-coarse}
\bm{e}_{m (n+1)} = \Phi^{(t_{m n}, m \delta t)} \bm{e}_{m n} + \bm{r}_{m(n+1)}, \quad n = 0, 1, \ldots, (n_t-1)/m,
\end{align}
in which $\bm{r}_{m(n+1)}$ is the algebraic residual of \eqref{eq:one-step} at the C-point $t = t_{m(n+1)}$.
Here $\Phi^{(t_{m n}, m \delta t)}$ is the \textit{coarse-grid time-stepping operator}, and it should approximate the \textit{ideal} coarse-grid time-stepping operator defined by stepping across the coarse time interval $t \in [t_{mn}, t_{mn} + m \delta t]$ using the fine-grid operators: $\Phi^{(t_{m n}, m \delta t)} \approx \prod \limits_{k = 0}^{m-1} \Phi^{(t_{m n + k}, \delta t)} =: \Phi^{(t_{m n}, m \delta t)}_{\rm ideal}$.
Upon solving \eqref{eq:one-step-coarse}, the approximate coarse-grid error is interpolated to the fine grid via injection---i.e., added to existing C-point values.
Equation \eqref{eq:one-step-coarse} may be solved by sequential time-stepping, resulting in a two-level method. Alternatively, since \eqref{eq:one-step-coarse} has the same structure as \eqref{eq:one-step}, its solution may be approximated in parallel by applying the MGRIT algorithm recursively, resulting in a multilevel method.

In the event that $\Phi^{(t_{m n}, m \delta t)} = \Phi^{(t_{m n}, m \delta t)}_{\rm ideal}$, \eqref{eq:one-step-coarse} is exactly the C-point Schur complement of the residual equation associated with \eqref{eq:one-step} and MGRIT converges in a single iteration. 
However, stepping across the coarse time interval $t \in [t_{mn}, t_{mn} + m \delta t]$ with $\Phi^{(t_{m n}, m \delta t)}_{\rm ideal}$ is just as expensive as stepping across this interval on the fine grid, and thus no parallel speed-up can be achieved.
Therefore, stepping with $\Phi^{(t_{m n}, m \delta t)}$ should be less expensive than the ideal coarse-grid operator. 
Crucially, however, fast convergence of the method generally requires that the coarse-grid operator accurately approximates (in a certain sense) its ideal counterpart \cite{Dobrev_etal_2017, Southworth_2019, DeSterck_etal_2021}. 

Our numerical tests will use the open-source MGRIT implementation provided by XBraid \cite{xbraid}.
In our tests, the initial MGRIT iterate will be taken as a vector with entries uniformly random between zero and one.
The metric that we use to report convergence is the number of MGRIT iterations required to reduce the space-time residual by at least 10 orders of magnitude in the two-norm from its initial value.

\subsection{Motivating examples}
\label{sec:motivating}

As mentioned previously, a rediscretized coarse-grid operator often yields fast convergence for diffusion-dominated problems, but tends to be a poor choice for advection-dominated problems \cite{KrzysikThesis2021, DeSterck_etal_2021}.
We now demonstrate that a rediscretized coarse-grid semi-Lagrangian operator is indeed also a poor choice for our model problem.
To do so, we consider PDE \eqref{eq:ad} with constant wave-speed $\alpha = 1$, and examine the convergence factor of two-level MGRIT.
Specifically, for fixed coarsening factor $m$, we consider the function
\begin{align}
\label{eq:rho_MGRIT_asym}
\rho(c) = \max \limits_{\omega \in [-\pi, \pi)} |\lambda(\omega; c)|^{m} \frac{|\lambda^m(\omega; c) - \mu(\omega; c)|}{1 - |\mu(\omega; c)|},
\end{align}
in which $c = \delta t / h$ is the (fine-grid) CFL number.
For fixed $c$, the quantities $\lambda(\omega; c)$ and $\mu(\omega; c)$ in \eqref{eq:rho_MGRIT_asym} are the Fourier symbols of the fine- and coarse-grid time stepping operators, respectively, as functions of spatial Fourier frequency $\omega$.
For fixed $c$, \eqref{eq:rho_MGRIT_asym} is the Fourier analysis estimate of the asymptotic convergence factor of two-level MGRIT as $n_t \to \infty$ \cite[Chap. 3]{KrzysikThesis2021}.
Provided the number of iterations is not so large that the initial condition has been sequentially propagated across much of the time-domain via the fine-grid relaxation scheme---the practically relevant case for MGRIT---, \eqref{eq:rho_MGRIT_asym} provides a useful estimate of the MGRIT convergence factor for finite $n_t$.

For semi-Lagrangian orders $p = 1,3$, \eqref{eq:rho_MGRIT_asym} is plotted in \Cref{fig:conv_factor_rediscretization} for $c \in [0.5, 1]$.\footnote{Specifically for the constant-wave-speed problem, by symmetry of the fine- and coarse-grid operators used here when $p$ is odd, the corresponding plots for CFL numbers $c \in [0, 0.5]$ are a reflection of those for $c \in [0.5, 1]$ in \Cref{fig:conv_factor_rediscretization} about $c = 0.5$. Similarly, the corresponding plots for $c \in [k,k+1]$ with $k \in \mathbb{N}$ are identical to those for $c \in [0, 1]$.}
For a given $c$, to numerically evaluate the maximum in \eqref{eq:rho_MGRIT_asym} over $\omega$ we use 512 equispaced points in $[-\pi, \pi)$.
In \Cref{fig:conv_factor_rediscretization} for each $m$, there is a significant interval of CFL numbers for which $\rho > 1$, indicating that the residual norm grows from one iteration to the next (that is, at least in iterations for which \eqref{eq:rho_MGRIT_asym} is a valid estimate).
Moreover, the size of interval for which $\rho > 1$ appears to grow with $m$.
We conclude that simply rediscretizing the fine-grid semi-Lagrangian scheme on the coarse grid does not result in a robust MGRIT solver for our model problem. 
We do not provide analogous plots for even polynomial degrees $p = 2,4$ because the convergence factor \eqref{eq:rho_MGRIT_asym} is larger than one for almost all $c$ in these cases. 
%

\begin{figure}[t!]
\centerline{
\includegraphics[scale=0.36]{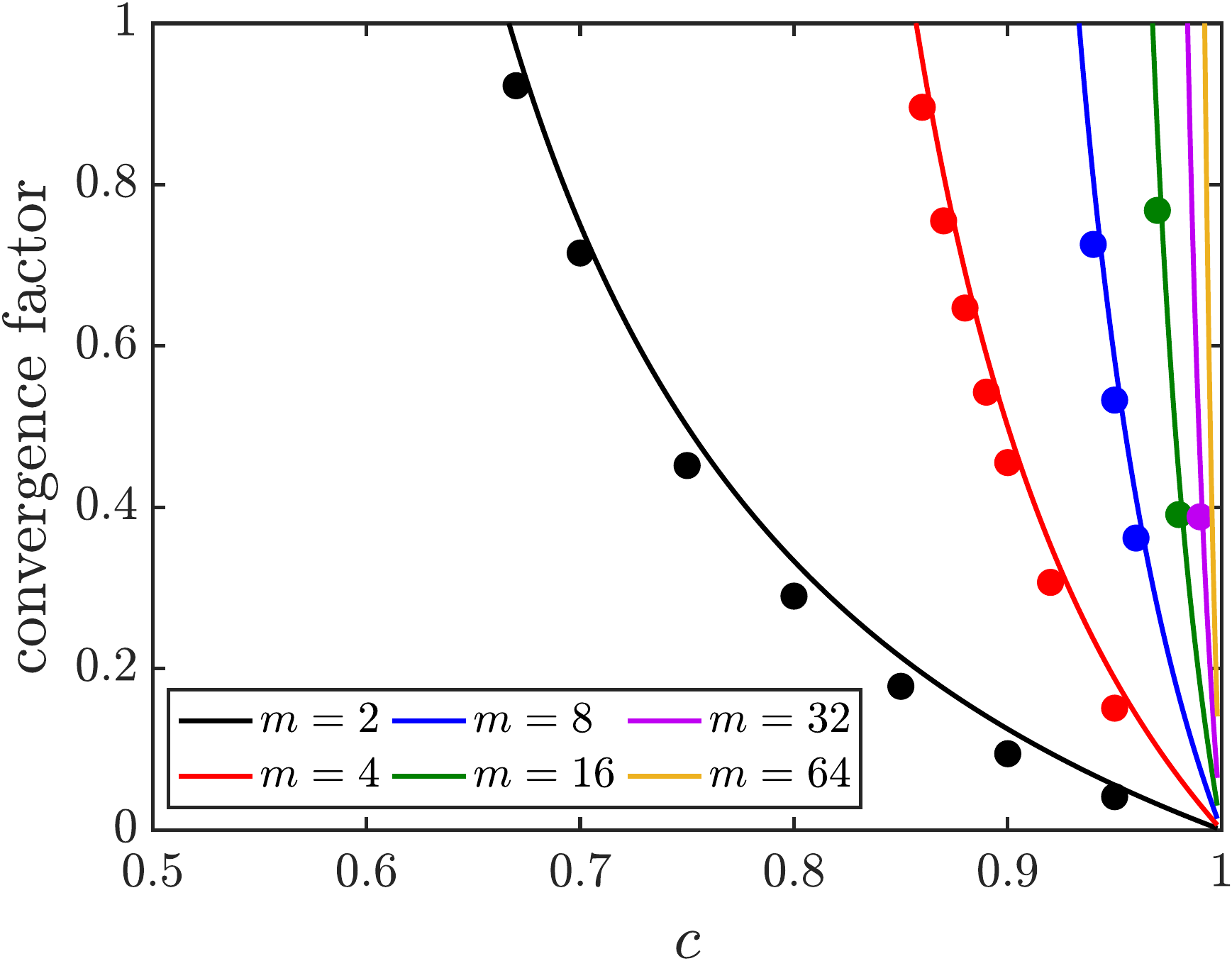}
\quad
\includegraphics[scale=0.36]{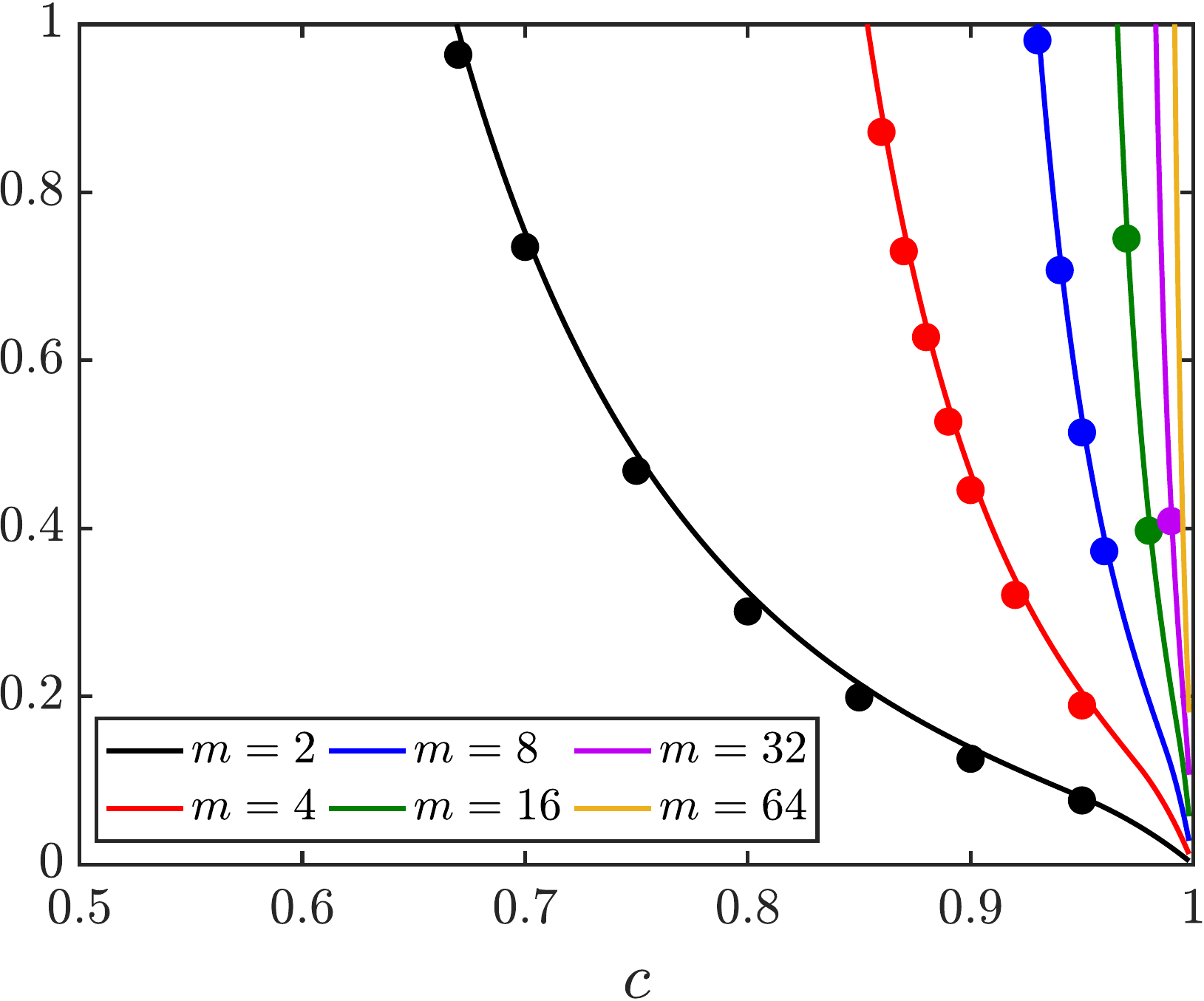}
}
\caption{Convergence factor for two-level MGRIT to solve semi-Lagrangian discretizations of the constant-wave-speed advection problem \eqref{eq:ad} using a rediscretized semi-Lagrangian coarse-grid operator. Left: semi-Lagrangian discretization order $p = 1$; right: $p = 3$.
For a fixed coarsening factor $m$, solid lines show the Fourier analysis estimate of the asymptotic convergence factor \eqref{eq:rho_MGRIT_asym}.
Solid markers show convergence factors measured on the final iteration of numerical experiments.
\label{fig:conv_factor_rediscretization}
}
\end{figure}

To demonstrate that \eqref{eq:rho_MGRIT_asym} is an accurate estimate of the true convergence factor, overlaid on the plots in \Cref{fig:conv_factor_rediscretization} are experimentally measured convergence factors taken from the final MGRIT iteration in numerical experiments.
The tests use a space-time domain with $n_x \times n_t = 2^7 \times 2^{15}$ points.
Note also that analogous numerical experiments for variable-wave-speed problems indicate that rediscretization is a poor choice for those problems too. 
%

\section{Coarse-grid operator in one spatial dimension}
\label{sec:var}

From \Cref{fig:conv_factor_rediscretization}, it is clear that, in general, rediscretizing the semi-Lagrangian operator on the coarse grid does not yield robust MGRIT convergence.
Despite this, the idea of using a semi-Lagrangian discretization on time-coarsened grids remains appealing because the discretization is stable for all time-step sizes.
Furthermore, in \cite{DeSterck_etal_2021}, we argued both for stability, and in terms of approximating the ideal coarse-grid operator, that a coarse-grid operator for an advection-dominated problem should be semi-Lagrangian-like in nature, in the sense that its numerical domain of dependence should track the characteristic curves.
For these reasons, we seek a coarse-grid operator that is based on a semi-Lagrangian discretization, but that provides a better approximation to the ideal coarse-grid operator than the semi-Lagrangian discretization does. 
There are many possible metrics one could use to characterize the difference between a coarse-grid operator and its ideal counterpart. Here, we use the concept of local truncation error, which is defined as the amount by which the exact PDE solution fails to satisfy the discrete scheme after one time step. 
Accordingly, we begin in the following section with estimates for the truncation error of the fine-grid semi-Lagrangian discretization, and the associated ideal coarse-grid operator.

Throughout the rest of the paper, the matrix ${\cal D}_{p+1} \in \mathbb{R}^{n_x \times n_x}$ is defined such that $h^{-(p+1)}{\cal D}_{p+1}$ represents a finite-difference rule for approximating the $p+1$st derivative of periodic grid functions.
Let $\bm{v} = \big(v(x_1), \ldots, v(x_{n_x}) \big)^\top \in \mathbb{R}^{n_x}$ denote a vector of a periodic function $v(x)$ evaluated on the spatial mesh. 
Then, if the finite-difference rule is of order $s \in \mathbb{N}$ and $v$ is at least $p+1+s$ times continuously differentiable
\begin{align} \label{eq:Dp+1_def}
\left(
\frac{{\cal D}_{p+1}}{h^{p+1}} \bm{v} 
\right)_i = \left. \frac{\d^{p+1} v}{\d x^{p+1}} \right|_{x_i} + {\cal O}(h^s), 
\quad i \in \{1, \ldots, n_x \}.
\end{align}
The order $s$ is not particularly important for our purposes so we do not specify it, but note that numerical tests in later sections use finite-difference approximations with $s = 2$.
Since the mesh points are equispaced, the matrix ${\cal D}_{p+1}$ is circulant.
Finally, note that while the entries of ${\cal D}_{p+1}$ are independent of $h$, its action is not: ${\cal D}_{p+1} \bm{v} = {\cal O}(h^{p+1})$ if $\bm{v}$ is independent of $h$.

\subsection{Truncation error estimates for exact departure points}
\label{sec:truncation_estimates}

For simplicity, consider an idealized semi-Lagrangian discretization ${\cal S}_{p, r}^{(t_n, \delta t)}$ that locates departure points exactly, which we will denote by writing $r = \infty$. 
Furthermore, the following error estimates (\Cref{lem:SL_trunc}, \Cref{cor:SL_ideal_trunc}, and \Cref{lem:SL_ideal_pert}) assume that the wave-speed function may depend on time but is independent of space.
Discussion on spatially varying wave-speed functions is given after \Cref{lem:SL_ideal_pert}.
Note that \cite[p. 170]{Falcone_Ferretti_2014} considers a simpler but related type of error estimate to \Cref{lem:SL_trunc} when $\alpha$ is constant.

\begin{lemma}[Semi-Lagrangian truncation error for $r = \infty$]
\label{lem:SL_trunc}
Let ${\cal D}_{p+1}$ be as in \eqref{eq:Dp+1_def}. 
Suppose that the solution $u(x,t)$ of \eqref{eq:ad} is at least $p+1+s$ times continuously differentiable with respect to $x$, and that the wave-speed in \eqref{eq:ad} is independent of space, $\alpha(x, t) \equiv \alpha(t)$.
Define $\bm{u}(t) \in \mathbb{R}^{n_x}$ as the vector composed of the PDE solution sampled in space at the mesh points $\bm{x}$ and at time $t$, and let ${\cal S}_{p, \infty}^{(t_n, \delta t)}$ be the semi-Lagrangian discretization of \eqref{eq:ad} that locates departure points exactly at time $t_{n}$. 
Then, the local truncation error of this discretization can be expressed as
\begin{align} \label{eq:SL_trunc}
\bm{u}(t_{n+1}) - {\cal S}_{p, \infty}^{(t_{n}, \delta t)} \bm{u}(t_{n})
=
(-h)^{p+1} 
\diag \Big( f_{p+1} \big(\bm{\varepsilon}^{(t_{n}, \delta t)} \big) \Big) 
\frac{{\cal D}_{p+1}}{h^{p+1}}
\bm{u}(t_{n+1})
+ {\cal O}(h^{p+2}),
\end{align}
in which $\bm{\varepsilon}^{(t_{n}, \delta t)} := \big( \varepsilon_1^{(t_{n}, \delta t)}, \ldots, \varepsilon_{n_x}^{(t_{n}, \delta t)} \big)^\top \in \mathbb{R}^{n_x}$, and $\varepsilon_i^{(t_{n}, \delta t)}$ is defined in \eqref{eq:depart_decompose}.
The function $f_{p+1}$ in \eqref{eq:SL_trunc} is the following degree $p+1$ polynomial 
\begin{align} \label{eq:fpoly_def}
f_{p+1}(z) := \frac{1}{(p+1)!} \prod \limits_{q = - \ell(p)}^{r(p)} (q+z).
\end{align}
Note that since $\bm{u}(t_{n+1})$ is independent of $h$, the vector $\frac{{\cal D}_{p+1}}{h^{p+1}}
\bm{u}(t_{n+1})$ in \eqref{eq:SL_trunc} is independent of $h$ to leading order (see \eqref{eq:Dp+1_def}).
\end{lemma}
\begin{proof}
Since ${\cal S}_{p,\infty}^{(t_n, \delta t)}$ exactly locates departure points of \eqref{eq:ad}, the only truncation error resulting from applying the $i$th row of ${\cal S}_{p,\infty}^{(t_n, \delta t)}$ to $u(\bm{x},t_n)$ is the error from the polynomial interpolation at the departure point $\xi_i^{(t_n,\delta t)}(t_n) = x_i^{(t_n,\delta t)} - h \varepsilon_i^{(t_n,\delta t)}$.
Since $u$ is at least $p+1$ times continuously differentiable with respect to $x$, the standard error estimate from polynomial interpolation theory can be applied (see, e.g., \cite[Thm. 3.1.1]{Davis_1975}). 
Since the interpolation nodes are separated by distance $h$, applying this estimate at the $i$th departure point at time $t_{n}$ yields
\begin{align} 
\begin{split}
&u\big(\xi_i^{(t_{n}, \delta t)}(t_{n}), t_{n}\big) 
-  
\left( {\cal S}_{p,\infty}^{(t_{n}, \delta t)} \bm{u}(t_{n}) \right)_i 
\\
&= 
\frac{1}{(p+1)!} \prod \limits_{q = - \ell(p)}^{r(p)} 
\left[ 
\big(x_i^{(t_{n},\delta t)} - h \varepsilon_i^{(t_{n},\delta t)} \big)
- \big( x_i^{(t_{n},\delta t)} + h q\big) \right] 
\left.\frac{\partial^{p+1} u}{\partial x^{p+1}}\right|_{\big(\zeta_i^{(t_n, \delta t)}, t_{n}\big)}
\end{split}
\\
\label{eq:poly_interp_error}
&=
(-h)^{p+1} f_{p+1}\big( \varepsilon_i^{(t_{n},\delta t)} \big) \left.\frac{\partial^{p+1} u}{\partial x^{p+1}}\right|_{\big(\zeta_i^{(t_n, \delta t)}, t_n\big)},
\end{align}
for some unknown point $\zeta_i^{(t_n, \delta t)} \in {\big(x_i^{(t_{n}, \delta t)} - h \ell(p), x_i^{(t_{n}, \delta t)} + h r(p)\big)}$.
Since $\zeta_i^{(t_n, \delta t)}$ and $\xi_i^{(t_n,\delta t)}(t_n)$ are a distance of ${\cal O}(h)$ apart, write $\zeta_i^{(t_n, \delta t)} = \xi_i^{(t_n,\delta t)}(t_n) + h \widehat{\zeta}_i^{(t_n, \delta t)}$ for some other unknown quantity $\widehat{\zeta}_i^{(t_n, \delta t)}$. 
Then, by Taylor expansion, 
\begin{align}
\label{eq:Taylor_expansion}
\left.\frac{\partial^{p+1} u}{\partial x^{p+1}}\right|_{\big(\zeta_i^{(t_n, \delta t)}, t_n\big)}
=
\left.\frac{\partial^{p+1} u}{\partial x^{p+1}}\right|_{\big(\xi_i^{(t_n, \delta t)}(t_n) + h \widehat{\zeta}_i^{(t_n, \delta t)}, t_{n} \big)}
= 
\left.\frac{\partial^{p+1} u}{\partial x^{p+1}}\right|_{\big(\xi_i^{(t_n, \delta t)}(t_n), t_{n}\big)}
+ 
{\cal O}(h).
\end{align}
Next, observe that differentiating the advection problem \eqref{eq:ad} $p+1$ times with respect to $x$ yields when $\alpha(x,t) \equiv \alpha(t)$
\begin{align}
\frac{\partial }{\partial t} \frac{\partial^{p+1} u}{\partial x^{p+1}} 
+
\alpha(t)
\frac{\partial }{\partial x} \frac{\partial^{p+1} u}{\partial x^{p+1}} 
=
0.
\end{align}
Thus, $\frac{\partial^{p+1} u}{\partial x^{p+1}}$ is constant along characteristics of the advection problem \eqref{eq:ad}, from which it follows that $\left.\frac{\partial^{p+1} u}{\partial x^{p+1}}\right|_{\big(\xi_i^{(t_n, \delta t)}(t_n), t_{n}\big)} = \left.\frac{\partial^{p+1} u}{\partial x^{p+1}}\right|_{(x_i, t_{n+1})}$.
Applying this in \eqref{eq:Taylor_expansion} and substituting the result along with ${u\big(\xi_i^{(t_n,\delta t)}(t_n), t_n \big) = u(x_i, t_{n+1})}$ into  \eqref{eq:poly_interp_error} gives
\begin{align}
u(x_i, t_{n+1})
-
\left( {\cal S}_{p,\infty}^{(t_{n}, \delta t)} \bm{u}(t_{n}) \right)_i 
=
(-h)^{p+1} f_{p+1}\big( \varepsilon_i^{(t_{n},\delta t)} \big) \left.\frac{\partial^{p+1} u}{\partial x^{p+1}}\right|_{(x_i, t_{n+1})}
+ 
{\cal O}(h^{p+2}).
\end{align}
Finally, substituting into this equation $\left.\frac{\partial^{p+1} u}{\partial x^{p+1}}\right|_{(x_i, t_{n+1})} = \left( \tfrac{{\cal D}_{p+1}}{h^{p+1}} 
\bm{u}(t_{n+1}) \right)_{i} + {\cal O}(h^s)$ (see \eqref{eq:Dp+1_def}) gives the $i$th row of the claimed result \eqref{eq:SL_trunc}.
\end{proof}
\begin{corollary}[Ideal coarse-grid semi-Lagrangian truncation error for $r = \infty$]
\label{cor:SL_ideal_trunc}
Suppose the assumptions of \Cref{lem:SL_trunc} hold.
Then, the ideal coarse-grid operator defined by time-stepping across $t \in [t_n, t_n + m \delta t]$ with the $m$ fine-grid operators ${\cal S}_{p, \infty}^{(t_{n} + k \delta t, \delta t)}$, $k = 0, 1, \ldots, m-1$, has a local truncation error given by
\begin{align} \label{eq:SL_ideal_trunc} 
\begin{split}
&\bm{u}(t_{n+m})- \Bigg[ \prod_{k = 0}^{m-1} {\cal S}_{p, \infty}^{(t_n+k \delta t, \delta t)} \Bigg]
\bm{u}(t_n)  
\\
&\quad=
(-h)^{p+1}  
\diag \Bigg( \sum \limits_{k = 0}^{m-1} f_{p+1} \big(\bm{\varepsilon}^{(t_n+k \delta  t, \delta t)} \big) \Bigg)
\frac{{\cal D}_{p+1}}{h^{p+1}}
\bm{u}(t_{n+m})  
+ {\cal O}(h^{p+2}).
\end{split}
\end{align}
\end{corollary}
\begin{proof}
Applying ${\cal S}_{p,\infty}^{(t_{n+1}, \delta t)}$ to both sides of \eqref{eq:SL_trunc} gives
\begin{align} 
\begin{split}
&{\cal S}_{p,\infty}^{(t_{n+1}, \delta t)} \bm{u}(t_{n+1}) 
- 
\Big[  {\cal S}_{p,\infty}^{(t_{n+1}, \delta t)} {\cal S}_{p,\infty}^{(t_n, \delta t)} \Big] \bm{u}(t_n)\\ 
&\quad =
(-h)^{p+1} 
{\cal S}_{p,\infty}^{(t_{n+1}, \delta t)} \diag \Big( f_{p+1} \big(\bm{\varepsilon}^{(t_n, \delta t)} \big) \Big) 
\frac{{\cal D}_{p+1}}{h^{p+1}}
\bm{u}(t_{n+1}) 
+ {\cal O}(h^{p+2}),
\end{split}
\\
\label{eq:SL_ideal_approx_step}
& \quad=
(-h)^{p+1} 
\diag \Big( f_{p+1} \big(\bm{\varepsilon}^{(t_n, \delta t)} \big) \Big) 
\frac{{\cal D}_{p+1}}{h^{p+1}}
\left( 
{\cal S}_{p,\infty}^{(t_{n+1}, \delta t)} 
\bm{u}(t_{n+1}) 
\right)
+ {\cal O}(h^{p+2}).
\end{align}
To arrive at \eqref{eq:SL_ideal_approx_step}, we have used the fact that when the wave-speed is independent of space ${\cal S}_{p,\infty}^{(t_{n+1}, \delta t)}$ commutes with $\diag \Big( f_{p+1} \big(\bm{\varepsilon}^{(t_n, \delta t)} \big) \Big)$ and ${\cal D}_{p+1}$.
When the wave-speed is independent of space, characteristics are parallel for all time $t$, and thus departure points of neighboring local characteristics are  equally separated by a distance $h$, from which it follows that $\bm{\varepsilon}^{(t_n, \delta t)}$ is constant.
Therefore $\diag \Big( f_{p+1} \big(\bm{\varepsilon}^{(t_n, \delta t)} \big) \Big)$ is a constant diagonal matrix and thus commutes with all other matrices.
Furthermore, since neighboring departure points are equispaced, the semi-Lagrangian matrix ${\cal S}_{p,\infty}^{(t_{n+1}, \delta t)}$ is circulant, and therefore commutes with the circulant matrix ${\cal D}_{p+1}$.

Moving on, from \eqref{eq:SL_trunc} one has
\begin{align}
\begin{split}
&{\cal S}_{p, \infty}^{(t_{n+1}, \delta t)} \bm{u}(t_{n+1})
\\ 
&\quad=
\bm{u}(t_{n+2})
-
(-h)^{p+1} 
\diag \Big( f_{p+1} \big(\bm{\varepsilon}^{(t_{n+1}, \delta t)} \big) \Big) 
\frac{{\cal D}_{p+1}}{h^{p+1}}
\bm{u}(t_{n+2}) 
+ {\cal O}(h^{p+2}).
\end{split}
\end{align}
Substituting this into \eqref{eq:SL_ideal_approx_step} and keeping only terms up to size ${\cal O}(h^{p+1})$ gives
\begin{align} \label{eq:SL_ideal_trunc2} 
\begin{split}
&u(\bm{x}, t_{n+2}) - \Big[  {\cal S}_{p, \infty}^{(t_{n+1}, \delta t)} {\cal S}_{p, \infty}^{(t_n, \delta t)} \Big]
\bm{u}(t_n) = {\cal O}(h^{p+2}) \, + 
\\
&\quad
(-h)^{p+1}  
\diag \Big( f_{p+1} \big(\bm{\varepsilon}^{(t_n, \delta t)} \big) + f_{p+1} \big(\bm{\varepsilon}^{(t_{n+1}, \delta t)} \big) \Big)
\frac{{\cal D}_{p+1}}{h^{p+1}}
\bm{u}( t_{n+2}).
\end{split}
\end{align}
Inductively repeating the above process on \eqref{eq:SL_ideal_trunc2} with the remaining $m-2$ fine-grid operators ${\cal S}_{p, \infty}^{(t_{n} + k \delta t, \delta t)}$, $k =2, \ldots, m-1$, one arrives at the result \eqref{eq:SL_ideal_trunc}. 
\end{proof}

Having developed an asymptotic expansion for the ideal coarse-grid operator, we now relate this to the coarse-grid semi-Lagrangian operator.
\begin{lemma}[Perturbed coarse-grid semi-Lagrangian operators]
\label{lem:SL_ideal_pert}
Suppose the assumptions of \Cref{lem:SL_trunc} hold.
Let ${\cal S}_{p,\infty}^{(t_n, m \delta t)}$ be the coarse-grid semi-Lagrangian discretization of \eqref{eq:ad} that locates departure points exactly.
Then, this operator can be expressed as a perturbation of the ideal coarse-grid operator $ \prod_{k = 0}^{m-1} {\cal S}_{p, \infty}^{(t_n+k \delta t, \delta t)}$ in the following three ways:
\begin{align} \label{eq:SL_ideal_pert1}
\begin{split}
&\Bigg[ \prod_{k = 0}^{m-1} {\cal S}_{p, \infty}^{(t_n+k \delta t, \delta t)} \Bigg]
\bm{u}(t_n)\\
&\quad =
{\cal S}_{p, \infty}^{(t_n, m \delta t)} 
\bm{u}(t_n) 
+ 
\diag \Big( \bm{\varphi}_{p+1}^{(t_n, m \delta t)} \Big)
{\cal D}_{p+1}
\bm{u}(t_{n+m})
+ {\cal O}(h^{p+2}), 
\end{split}
\\
\label{eq:SL_ideal_pert2}
&\quad =
\left[I
+ 
\diag \Big( \bm{\varphi}_{p+1}^{(t_n, m \delta t)} \Big)
{\cal D}_{p+1}
\right]
{\cal S}_{p, \infty}^{(t_n, m \delta t)} 
\bm{u}(t_{n})
+ {\cal O}(h^{p+2}), \\
\label{eq:SL_ideal_pert3}
&\quad=
\left[I
-
\diag \Big( \bm{\varphi}_{p+1}^{(t_n, m \delta t)} \Big)
{\cal D}_{p+1}
\right]^{-1}
{\cal S}_{p, \infty}^{(t_n, m \delta t)} 
\bm{u}(t_{n})
+ {\cal O}(h^{p+2}),
\end{align}
with the vector $\bm{\varphi}_{p+1}^{(t_n, m \delta t)} \in \mathbb{R}^{n_x}$ defined by
\begin{align}
\label{eq:varphi_p_def}
\bm{\varphi}_{p+1}^{(t_n, m \delta t)} := (-1)^{p+1}
\left(  f_{p+1} \big( \bm{\varepsilon}^{(t_n, m \delta t)} \big) - \sum \limits_{k = 0}^{m-1} f_{p+1} \big( \bm{\varepsilon}^{(t_{n} + k \delta t, \delta t)} \big) \right).
\end{align}
\end{lemma}
\begin{proof}
From the truncation error of the fine-grid semi-Lagrangian discretization ${\cal S}_{p, \infty}^{(t_n, \delta t)}$ given in \eqref{eq:SL_trunc}, it can immediately be seen that the truncation of the coarse-grid semi-Lagrangian discretization ${\cal S}_{p, \infty}^{(t_n,m \delta t)}$ is
\begin{align} \label{eq:SL_coarse_trunc}
\begin{split}
&
\bm{u}(t_{n+m}) - {\cal S}_{p, \infty}^{(t_n,m \delta t)} u(\bm{x}, t_n) 
\\ 
&\quad=
(-h)^{p+1} \diag \Big( f_{p+1} \big( \bm{\varepsilon}^{(t_n,m \delta t)} \big) \Big) 
\frac{{\cal D}_{p+1}}{h^{p+1}}
\bm{u}(t_{n+m}) 
+ {\cal O}(h^{p+2}).
\end{split}
\end{align}
Equation \eqref{eq:SL_ideal_pert1} follows by subtracting the truncation error \eqref{eq:SL_coarse_trunc} from that of the ideal coarse-grid operator's in \eqref{eq:SL_ideal_trunc}, and then rearranging the resulting equation for the ideal coarse-grid operator.
Equation \eqref{eq:SL_ideal_pert2} follows by substituting $\bm{u}(t_{n+m}) = {\cal S}_{p, \infty}^{(t_n,m \delta t)} u(\bm{x}, t_n) + {\cal O}(h^{p+1})$, as is given by \eqref{eq:SL_coarse_trunc}, into \eqref{eq:SL_ideal_pert1}.
Finally, \eqref{eq:SL_ideal_pert3} follows from the geometric expansion 
$\big[I
-
\diag \big( \bm{\varphi}_{p+1}^{(t_n, m \delta t)} \big)
{\cal D}_{p+1}
\big]^{-1} 
= 
I 
+ 
\diag \big( \bm{\varphi}_{p+1}^{(t_n, m \delta t)} \big)
{\cal D}_{p+1} 
\\
+ 
\big( \diag \big( \bm{\varphi}_{p+1}^{(t_n, m \delta t)} \big)
{\cal D}_{p+1} \big)^2 
+ 
\ldots$,
and, so, 
$\big[I
-
\diag \big( \bm{\varphi}_{p+1}^{(t_n, m \delta t)} \big)
{\cal D}_{p+1}
\big]^{-1} \bm{v} 
\\ 
= 
\big[ I + \diag \big( \bm{\varphi}_{p+1}^{(t_n, m \delta t)} \big) \big] \bm{v} 
+ 
{\cal O}(h^{2(p+1)})$
for sufficiently smooth $\bm{v}$.
\end{proof}

\begin{remark}[Estimates for spatially varying wave-speed functions]
\label{rem:space_vary_estimates}
The estimates in \Cref{lem:SL_trunc}, \Cref{cor:SL_ideal_trunc}, and \Cref{lem:SL_ideal_pert} were derived for PDE \eqref{eq:ad} with $\alpha(x,t) \equiv \alpha(t)$.
In particular, the proof of \Cref{lem:SL_trunc} does not generalize to the spatially variable case because then $\frac{\partial^{p+1} u}{\partial x^{p+1}}$ is no longer constant along characteristics of \eqref{eq:ad}.
Furthermore, the commutativity of ${\cal S}_{p,\infty}^{(t_{n+1}, \delta t)}$  with $\diag \big( f_{p+1} \big(\bm{\varepsilon}^{(t_n, \delta t)} \big) \big)$ and ${\cal D}_{p+1}$ used in the proof of \Cref{cor:SL_ideal_trunc} no longer holds.
However, in what follows below, we find numerically that the MGRIT method that is derived from the estimates in \Cref{lem:SL_ideal_pert} also works effectively for the case that $\alpha$ depends on $x$ and $p$ is odd.
For odd $p$, we conjecture based on numerical evidence reported in Supplementary Materials \Cref{SM:sec:SLtrunc_num_evidence}, that the estimates for the spatially independent wave-speed case hold for the spatially variable case up to terms of size ${\cal O}(h^{p+1} \delta t)$, while for even $p$, we do not believe that all of the estimates extend analogously.
\end{remark}

\subsection{Coarse-grid operator for exact departure points}
\label{sec:exact}

The significance of \Cref{lem:SL_ideal_pert} is that it provides several asymptotic relationships between the ideal coarse-grid operator and the coarse-grid semi-Lagrangian operator. 
Specifically, considering \eqref{eq:SL_ideal_pert1}, when the wave-speed function does not vary in space, the coarse-grid semi-Lagrangian operator serves as an ${\cal O}(h^{p+1})$ approximation to the ideal coarse-grid operator, recalling that ${\cal D}_{p+1} \bm{u}(t_{n+m}) = {\cal O}(h^{p+1})$ (see \eqref{eq:Dp+1_def}). 
However, the operators appearing in \eqref{eq:SL_ideal_pert2} and \eqref{eq:SL_ideal_pert3} serve as ${\cal O}(h^{p+2})$ approximations to the ideal coarse-grid operator when the wave-speed does not vary in space.

Based on \eqref{eq:SL_ideal_pert2}, we propose the following \textit{explicit} coarse-grid operator
\begin{align} \label{eq:Fp_def}
\Phi^{(t_n,m \delta t)} 
=
{\cal F}_{p+1}^{(t_n, m \delta t)} {\cal S}_{p,\infty}^{(t_n, m \delta t)},
\quad 
{\cal F}_{p+1}^{(t_n, m \delta t)} 
:=  
I + \diag \Big( \bm{\varphi}_{p+1}^{(t_n, m \delta t)} \Big) {\cal D}_{p+1}.
\end{align}
Based on \eqref{eq:SL_ideal_pert3} we propose the following \textit{implicit-explicit} coarse-grid operator
\begin{align} \label{eq:Bp_def}
\Phi^{(t_n,m \delta t)} 
=
{\cal B}_{p+1}^{(t_n, m \delta t)} {\cal S}_{p,\infty}^{(t_n, m \delta t)},
\quad 
{\cal B}_{p+1}^{(t_n, m \delta t)} 
:=  
\Big[I - \diag \Big( \bm{\varphi}_{p+1}^{(t_n, m \delta t)} \Big) {\cal D}_{p+1}\Big]^{-1}.
\end{align}
Based on \Cref{rem:space_vary_estimates}, we explore using these coarse-grid operators also for problems in which the wave-speed varies in space, even though this introduces further error terms in \eqref{eq:SL_ideal_pert2} and \eqref{eq:SL_ideal_pert3} that we conjecture are of size ${\cal O}(h^{p+1} \delta t)$ (at least when $p$ is odd).
The ${\cal F}_{p+1}^{(t_n, m \delta t)}$ and ${\cal B}_{p+1}^{(t_n, m \delta t)}$ notation is used to represent forward and backward Euler steps, respectively. 
Herein, we will typically refer to \eqref{eq:Fp_def} and \eqref{eq:Bp_def} as the `forward Euler' and `backward Euler' coarse-grid operators, respectively.
In our tests, ${\cal D}_{p+1}$ will be taken as a periodic, finite-difference approximation with second-order accuracy. 

The motivation for this nomenclature is that the coarse-grid operators can be interpreted as particular coarse-grid discretizations of a certain PDE.
For simplicity, suppose that the wave-speed is constant $\alpha(x, t) \equiv \alpha$, then the vector $\bm{\varphi}_{p+1}^{(t_n, m \delta t)} \in \mathbb{R}^{n_x}$ given by  \eqref{eq:varphi_p_def} will be constant and independent of $t_n$, so, let us denote its entries by $\varphi_{p+1}^{(m \delta t)} \in \mathbb{R}$. In this case, the PDE that \eqref{eq:Fp_def} and \eqref{eq:Bp_def} discretize on the coarse grid is
\begin{align} \label{eq:ad_cg_augmented}
\frac{\partial u}{\partial t} 
+ 
\alpha \frac{\partial u}{\partial x} 
= 
\varphi_{p+1}^{(m \delta t)} 
\frac{h^{p+1}}{m \delta t} 
\frac{\partial u^{p+1}}{\partial x^{p+1}}.
\end{align}
More specifically, they use a mixed discretization of \eqref{eq:ad_cg_augmented}, in which the coarse-grid semi-Lagrangian method ${\cal S}_{p,\infty}^{(m \delta t)}$ deals with the advection term, and then the method of lines is applied to solve the rest of the equation. In doing so, the right-hand side of \eqref{eq:ad_cg_augmented} is discretized in space using the matrix $ {\cal D}_{p+1}$, and the time derivative on the left-hand side is discretized using forward and backward Euler steps in \eqref{eq:Fp_def} and \eqref{eq:Bp_def}, respectively.
See \cite[Sec. 4.2.3]{KrzysikThesis2021} for further details.

An insightful numerical test case for the proposed coarse-grid operators \eqref{eq:Fp_def} and \eqref{eq:Bp_def} is when $\alpha$ is constant, since then departure points can be located exactly---an assumption made in deriving the truncation error estimates.
We now provide some general commentary on the results of our numerical tests for constant $\alpha$ to help motivate the direction of the remainder of this paper.

\underline{Case 1: Odd $p$, forward Euler operator \eqref{eq:Fp_def}.} For sufficiently small $m$ (e.g., $m = 2$ or $m = 4$), we often obtain a quickly converging MGRIT solver, while for larger $m$ the solver often diverges. 
In fact, we are able to rigorously prove for the case of constant $\alpha$ that the operator ${\cal F}_{p+1}^{(t_n, m \delta t)}$ in \eqref{eq:Fp_def} is unstable, in the sense that $\big\Vert {\cal F}_{p+1}^{(t_n, m \delta t)}  \big\Vert_2 > 1$, for sufficiently large $m$ (details on this can be found in \cite[Sec. 4.2.4]{KrzysikThesis2021}).
Note that stability of the coarse-grid operator is a necessary but not sufficient condition for MGRIT convergence.
It is likely that the instability of this operator is correlated with the poor performance we observe in our numerical tests for larger values of $m$. 
Since this instability arises even for moderate values of $m$, we do not believe this operator can be useful in practice.

\underline{Case 2: Odd $p$, backward Euler operator \eqref{eq:Bp_def}.} Generally speaking, we find this coarse-grid operator yields robust and fast MGRIT convergence. 
We are able to rigorously prove for the case of constant $\alpha$ that the operator ${\cal B}_{p+1}^{(t_n, m \delta t)}$ in \eqref{eq:Bp_def} is unconditionally stable, in the sense that $\big\Vert {\cal B}_{p+1}^{(t_n, m \delta t)} \big\Vert_2 \leq 1$ for all problem parameters (details can be found in \cite[Sec. 4.2.4]{KrzysikThesis2021}).

When $p$ is odd, the stability properties of ${\cal F}_{p+1}^{(t_n, m \delta t)}$ and ${\cal B}_{p+1}^{(t_n, m \delta t)}$ can be understood intuitively from the perspective that these operators resemble forward and backward Euler time discretizations, respectively, of a dissipative PDE $\frac{\partial v}{\partial t} = \kappa_{p+1}^{(m \delta t)} \frac{\partial^{p+1} v}{\partial x^{p+1}}$, with $\kappa_{p+1}^{(m \delta t)}$ a constant.
When $p = 1$ this dissipative PDE is the heat equation, for which it is well-known that forward Euler has poor stability properties, while backward Euler has excellent stability properties.

\underline{Case 3: Even $p$.} Our tests indicate that the coarse-grid operators do not yield robust MGRIT convergence. 
We do not yet have a full understanding of why this is the case, and addressing this remains ongoing work.
We speculate that this failure is at least in part due to MGRIT being able to more easily correct dissipative errors compared with dispersive errors (see the analysis of \cite{Ruprecht_2018}).
In addition, for even $p$, the nodes in the interpolation stencil change as a function of the mesh-normalized distance $\varepsilon$ to the east neighbor of the departure point (as explained in \Cref{sec:discretization}), and our proposed coarse-grid operator does not account for this.

Given the above discussion, throughout the remainder of this manuscript we focus on the backward Euler coarse-grid operator \eqref{eq:Bp_def} and consider only semi-Lagrangian discretizations using interpolating polynomials of odd degrees $p$.
As an initial demonstration of the improvement that the modified coarse-grid operator \eqref{eq:Bp_def} offers over the standard semi-Lagrangian coarse-grid operator, in \Cref{fig:conv_factor_dissipative} we recreate the convergence factor plots from \Cref{fig:conv_factor_rediscretization} for the constant-wave-speed advection problem.
%

\begin{figure}[b!]
\centerline{
\includegraphics[scale=0.36]{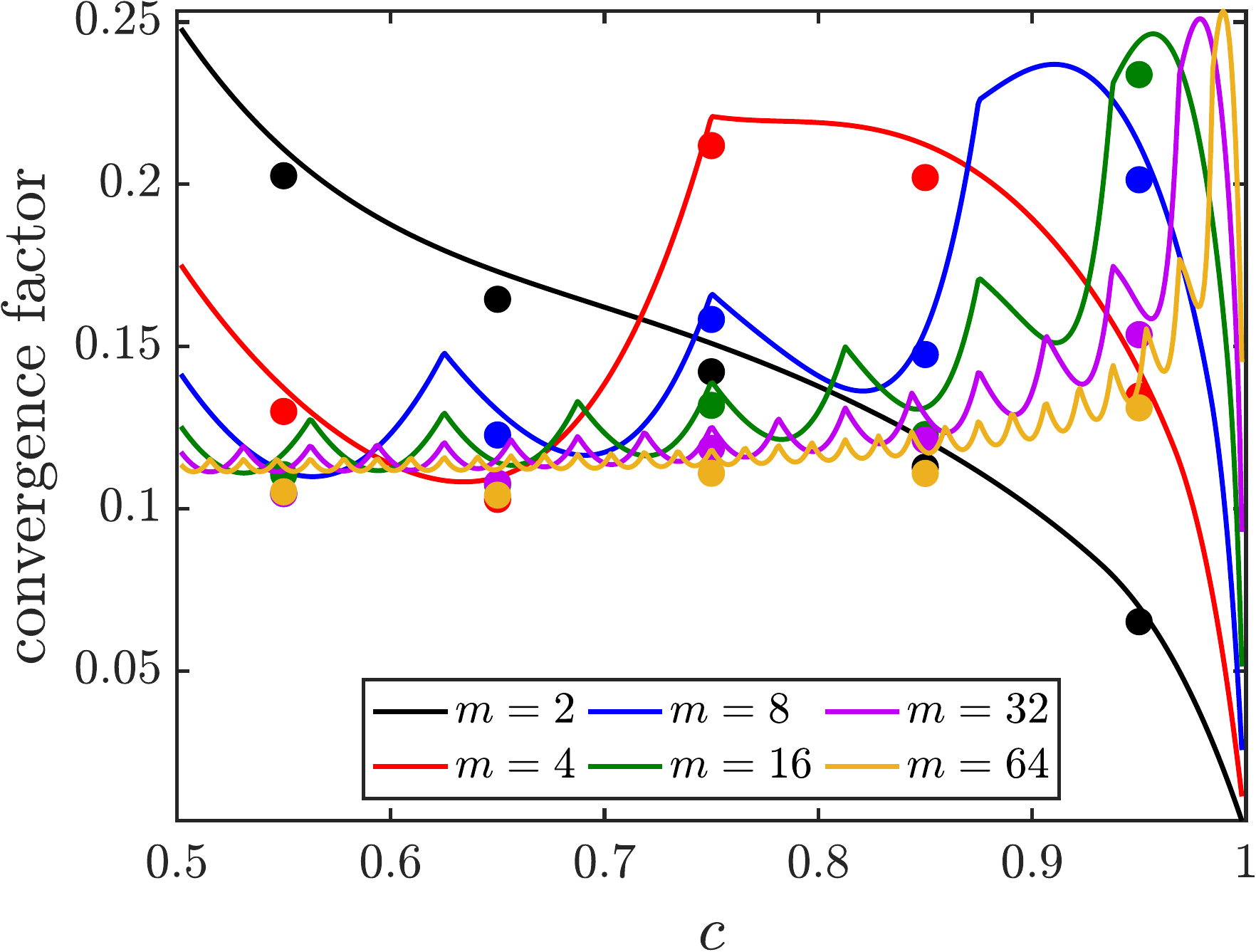}
\quad
\includegraphics[scale=0.36]{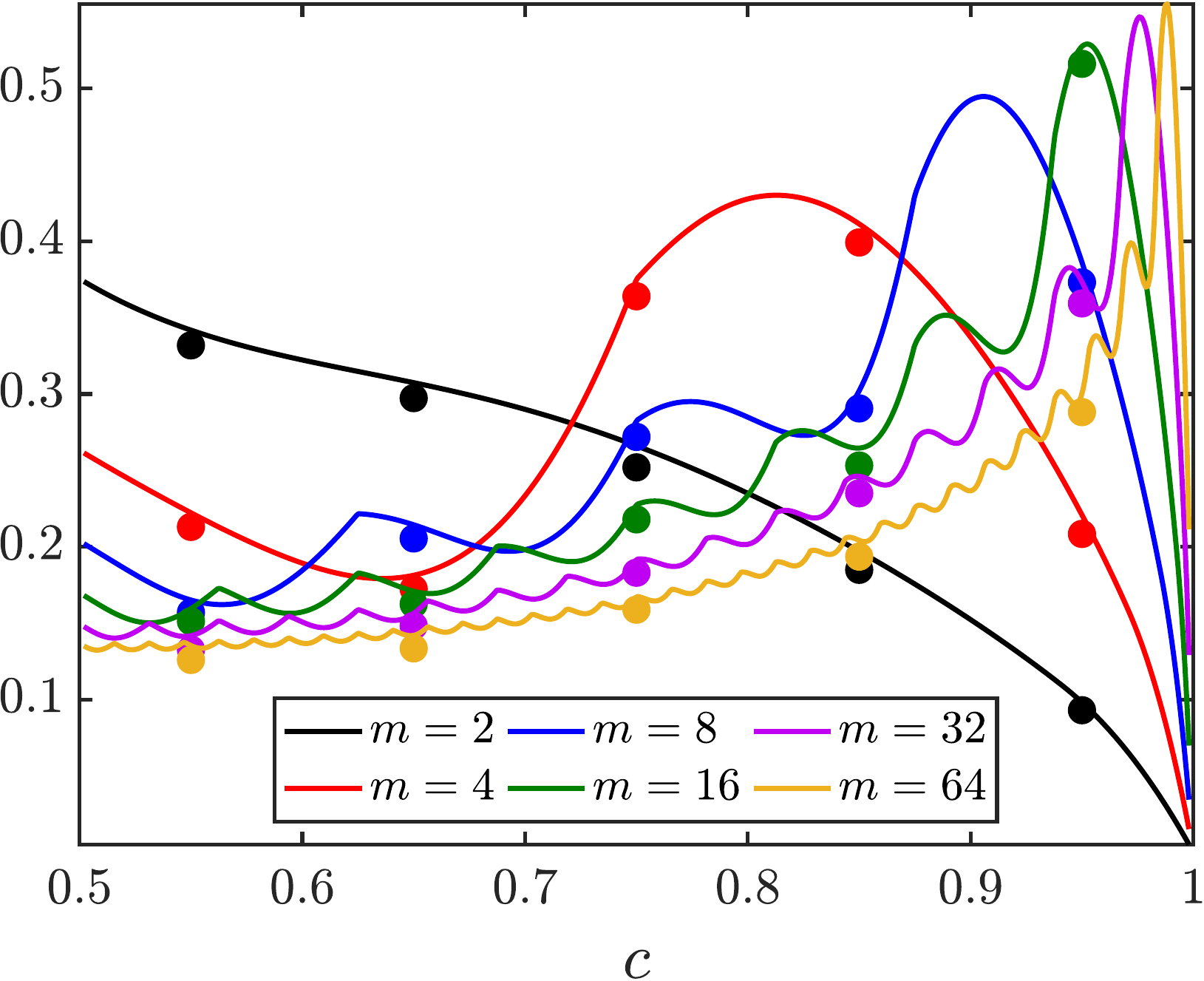}
}
\caption{Convergence factor for two-level MGRIT to solve semi-Lagrangian discretizations of the constant-wave-speed advection problem \eqref{eq:ad} with the proposed backward Euler coarse-grid operator \eqref{eq:Bp_def}. Left: semi-Lagrangian discretization order is $p = 1$; right: $p = 3$.
For a fixed coarsening factor $m$, solid lines show the Fourier analysis estimate of the asymptotic convergence factor \eqref{eq:rho_MGRIT_asym} as a function of the CFL number $c = \delta t/h$.
Solid markers show experimentally measured convergence factors.
\label{fig:conv_factor_dissipative}
}
\end{figure}

Remarkably, the plots in \Cref{fig:conv_factor_dissipative} show that the proposed coarse-grid operator leads to a convergent MGRIT solver for all CFL numbers, at least  when $\alpha$ is constant.
Moreover, convergence is fast for many CFL numbers; however, it does show a somewhat strong dependence on the CFL number and coarsening factor.
For a given $m$, the convergence rate shown in \Cref{fig:conv_factor_dissipative} tends to deteriorate going from $p = 1$ to $p = 3$, as is particularly obvious by contrasting the peaks of the $p = 1$ and $p = 3$ curves (note the two different vertical scales).
Interestingly, in our prior work \cite{DeSterck_etal_2021} on optimizing coarse-grid operators for explicit Eulerian discretizations of constant-wave-speed advection, convergence improved with increasing discretization order.
The reversal of this trend here and the non-uniformity of convergence with respect to CFL number perhaps hint that there exist \textit{better} coarse-grid operators than \eqref{eq:Bp_def}, but we leave this to future research.
In any event, for constant-wave-speed problems, the proposed backward Euler coarse-grid operator results in a robustly converging MGRIT solver for odd polynomial degrees that is fast for the majority of CFL numbers. 

A potential concern with the backward Euler operator \eqref{eq:Bp_def} is that applying it requires performing a linear solve, a task which is considerably more computationally expensive than applying an explicit semi-Lagrangian update. 
Our numerical experiments in the following section will demonstrate, however, that this linear solve may be carried out approximately with a small number of GMRES iterations.
In fact, the numerical tests used to generate the data points overlaid in \Cref{fig:conv_factor_dissipative} used only 10 GMRES iterations to approximately solve these linear systems.

\subsection{Coarse-grid operator for inexact departure points}
\label{sec:inexact}

We now move to the more practical case of developing a coarse-grid operator for when the fine- and coarse-grid semi-Lagrangian methods do not exactly locate departure points. 
Recall from \Cref{sec:discretization} that on the fine grid, the semi-Lagrangian method  estimates departure points by a single step of an ERK method.
The question is now how they should be located on the coarse grid. To answer this, we consider a heuristic strategy, since incorporating inexact departure point locations into the truncation estimates from \Cref{lem:SL_trunc} is not straightforward.

The immediately obvious way to compute coarse-grid departure points is to redeploy the ERK scheme that was used on the fine grid but with the coarse-grid time-step, i.e., rediscretize the ERK scheme. 
However, for larger values of $m$, this can be expected to produce highly inaccurate departure points, at least for variable wave-speeds.
Numerical tests for variable-wave-speed problems (not shown here for brevity) confirm that this strategy does not lead to robust MGRIT convergence.
A second option (see \cite[Sec 4.3.2]{KrzysikThesis2021} for details) is to increase the accuracy of this coarse-grid ERK integration by taking many small steps, such as $m$ steps using the fine-grid time-step $\delta t$, for example.
The obvious downside of this strategy is that it is expensive, since it uses fine-grid resolution to determine coarse-grid quantities. 

Based on the above discussion, for the fine-grid semi-Lagrangian discretization ${\cal S}_{p,r}^{(t_{n}, \delta t)}$ we propose the following backward Euler coarse-grid operator
\begin{align} \label{eq:Psi_Bp_def2}
\Phi^{(t_n,m \delta t)} 
=
{\cal B}_{p+1}^{(t_n, m \delta t)} {\cal S}_{p,r_*}^{(t_n, m \delta t)}.
\end{align}
Here $r_*$ signifies that the coarse-grid semi-Lagrangian operator ${\cal S}_{p,r_*}^{(t_n, m \delta t)}$ should locate departure points with an accuracy that is in some sense comparable to that of the fine-grid operator ${\cal S}_{p,r}^{(t_{n}, \delta t)}$, see \Cref{sec:backtracking}.
The backward Euler matrix ${\cal B}_{p+1}^{(t_n, m \delta t)}$ is still defined as it was in \eqref{eq:Bp_def}.
In \Cref{sec:backtracking}, a scalable strategy is presented for estimating coarse-grid departure points by reusing the departure point calculations from the fine grid. 
However, since this discussion is detailed, we first present numerical experiments that use this strategy for the coarse-grid operator \eqref{eq:Psi_Bp_def2}.

In the numerical tests, the fine-grid time-step is chosen as $\delta t = 0.85 h$, and the following wave-speed functions are considered
\begin{align}
\label{eq:alpha1}
\alpha(x,t) &= 1, \\
\label{eq:alpha2}
\alpha(x,t) &= \cos (2 \pi t), \\
\label{eq:alpha4}
\alpha(x,t) &= \cos (2 \pi t) \cos (2 \pi x).
\end{align}
Plots of the functions \eqref{eq:alpha2} and \eqref{eq:alpha4}, and the corresponding PDE solutions are given in Supplementary \Cref{SM:fig:wave_and_sol_1D_alpha2,SM:fig:wave_and_sol_1D_alpha4}, respectively.
The number of MGRIT iterations required to reach convergence on these problems is given in \Cref{tab:two_level_iters}.
The solution of coarse-grid linear systems involving the matrix ${\cal B}_{p+1}^{(t_n, m \delta t)}$ is approximated using 10 GMRES iterations with a zero initial guess.
Fewer than 10 GMRES iterations can be used without impacting the results, but the focus of these particular tests is to determine the MGRIT convergence rate, independent of the cost of solving the coarse-grid linear systems.
Note also for reasons relating to using GMRES inside MGRIT, we use a linear version of MGRIT, which we have implemented in XBraid (by default, XBraid uses the FAS framework); see \Cref{SM:sec:GMRES_FAS} for details.

Generally speaking, the convergence rates in \Cref{tab:two_level_iters} are fast for the two-level solution of hyperbolic problems.
A general trend among these results is the convergence rate deteriorating with increasing discretization order, consistent with \Cref{fig:conv_factor_dissipative}.
For the constant-wave-speed case, the iteration counts in \Cref{tab:two_level_iters} are essentially constant as the space-time mesh is refined. 
For the variable-wave-speed cases, there is some growth in iteration counts for the two high-order discretizations; at the same time, the iterations for these variable-wave-speed problems are typically smaller than those for the constant-wave-speed problem.
%

%
%
\renewcommand{\arraystretch}{1.1}
\begin{table}[t!]
  \centering
  \begin{tabular}{| c | c | ccc | ccc | ccc |}  
  \cline{3-11}
     \multicolumn{2}{ c|}{} & \multicolumn{3}{ c|}{$\alpha = \eqref{eq:alpha1}$} & \multicolumn{3}{ c|}{$\alpha = \eqref{eq:alpha2}$} & \multicolumn{3}{ c|}{$\alpha = \eqref{eq:alpha4}$}\\\cline{3-11}
   \multicolumn{2}{ c|}{} & \multicolumn{3}{ c|}{$m$} & \multicolumn{3}{ c|}{$m$} & \multicolumn{3}{ c|}{$m$} \\\hline
$p,r$ & $n_x \times n_t$ & 4 & 8 & 16 & 4 & 8 & 16 & 4 & 8 & 16 \\\hline
\Xhline{2\arrayrulewidth}
{\multirow{3}{*}{$1,1$}}
 & $2^{8} \times 2^{10}$    & 14  & 12 & 11        & 12 & 10 & 11    & 11 & 11 & 12 \\\cline{2-11}
 & $2^{10} \times 2^{12}$  & 14  & 12  & 11       & 12 & 11  & 11    & 12 & 12 & 12 \\\cline{2-11}
 & $2^{12} \times 2^{14}$  & 14  & 12  & 11       & 14 & 11 & 11    & 13 & 13 & 13   \\\hline
\Xhline{2\arrayrulewidth}
{\multirow{3}{*}{$3,3$}}
 & $2^{8} \times 2^{10}$    & 22 & 17 & 15      & 15 & 13 & 12     & 13 & 15 & 15  \\\cline{2-11}
 & $2^{10} \times 2^{12}$  & 22 & 17 & 15      & 16 & 14 & 13     & 16 & 16 & 16 \\\cline{2-11}
 & $2^{12} \times 2^{14}$  & 22 & 17 & 15      & 21 & 15 & 14     & 19 & 19 & 19  \\\hline
\Xhline{2\arrayrulewidth}
{\multirow{3}{*}{$5,5$}}
 & $2^{8} \times 2^{10}$  & 30 & 22 & 18      & 18 & 15 & 14      & 15 & 18 & 16   \\\cline{2-11}
 & $2^{10} \times 2^{12}$ & 31 & 23 & 20      & 18 & 16 & 15     & 18 & 19 & 20  \\\cline{2-11}
 & $2^{12} \times 2^{14}$ & 31 & 23 & 20      & 27 & 18 & 15     & 23 & 24 & 24 \\\hline
 \end{tabular}
  \caption{
	Number of two-level MGRIT iterations to reach convergence.
	The wave-speed $\alpha(x,t)$ of the advection problem \eqref{eq:ad} is indicated in the top row of the table.
	The fine-grid semi-Lagrangian discretization is ${\cal S}_{p,r}^{(t_n, \delta t)}$, and the coarse-grid operator is the dissipatively corrected semi-Lagrangian operator \eqref{eq:Psi_Bp_def2}.
	The coarse-grid semi-Lagrangian operators ${\cal S}_{p,r_*}^{(t_n, m \delta t)}$ estimate departure points using the linear interpolation and backtracking strategy from \Cref{sec:backtracking}.
  \label{tab:two_level_iters}
}
\end{table}

\subsubsection{Scalable strategy for estimating coarse-grid departure points}
\label{sec:backtracking}

We now describe the strategy for estimating coarse-grid departure points that was used to generate the results in \Cref{tab:two_level_iters}.
Recall that in applying the coarse-grid semi-Lagrangian operator ${\cal S}_{p, r_*}^{(t_n, m \delta t)}$, we need to compute the values at time $t_n$ of the local coarse-grid characteristics $\xi_i^{(t_n, m \delta t)}(t)$ that arrive at $(x,t) = (x_i, t_n + m \delta t)$.
When time-stepping across the interval $t \in [t_n, t_n + m \delta t]$ on the fine grid with ${\cal S}_{p, r}^{(t_n + k \delta t, \delta t)}$, ${k = 0, 1, \ldots, m-1}$, we map out the trajectories of the fine-grid characteristics $\xi_i^{(t_n + k \delta t, \delta t)}(t)$ over the fine-grid subintervals $t \in [t_n + k \delta t, t_n + (k+1) \delta t]$ (see the gold lines in \Cref{fig:char_interp}, with the circle markers representing the fine-grid departure points).
In computing these fine-grid characteristics using an ERK method of order $r$, we in effect map out the vector field that dictates the flow of any local coarse-grid characteristic across the coarse space-time slab $(x,t) \in \Omega \times [t_n, t_n + m \delta t]$ with fine-grid-scale accuracy.
The idea we propose now is to approximately propagate a coarse-grid characteristic through this space-time slab by recycling the fine-grid characteristic directions to guide its path in an interpolating manner, following the schematic shown in \Cref{fig:char_interp}.
%

\begin{figure}[t!]
\centerline{
\includegraphics[width=0.55\textwidth]{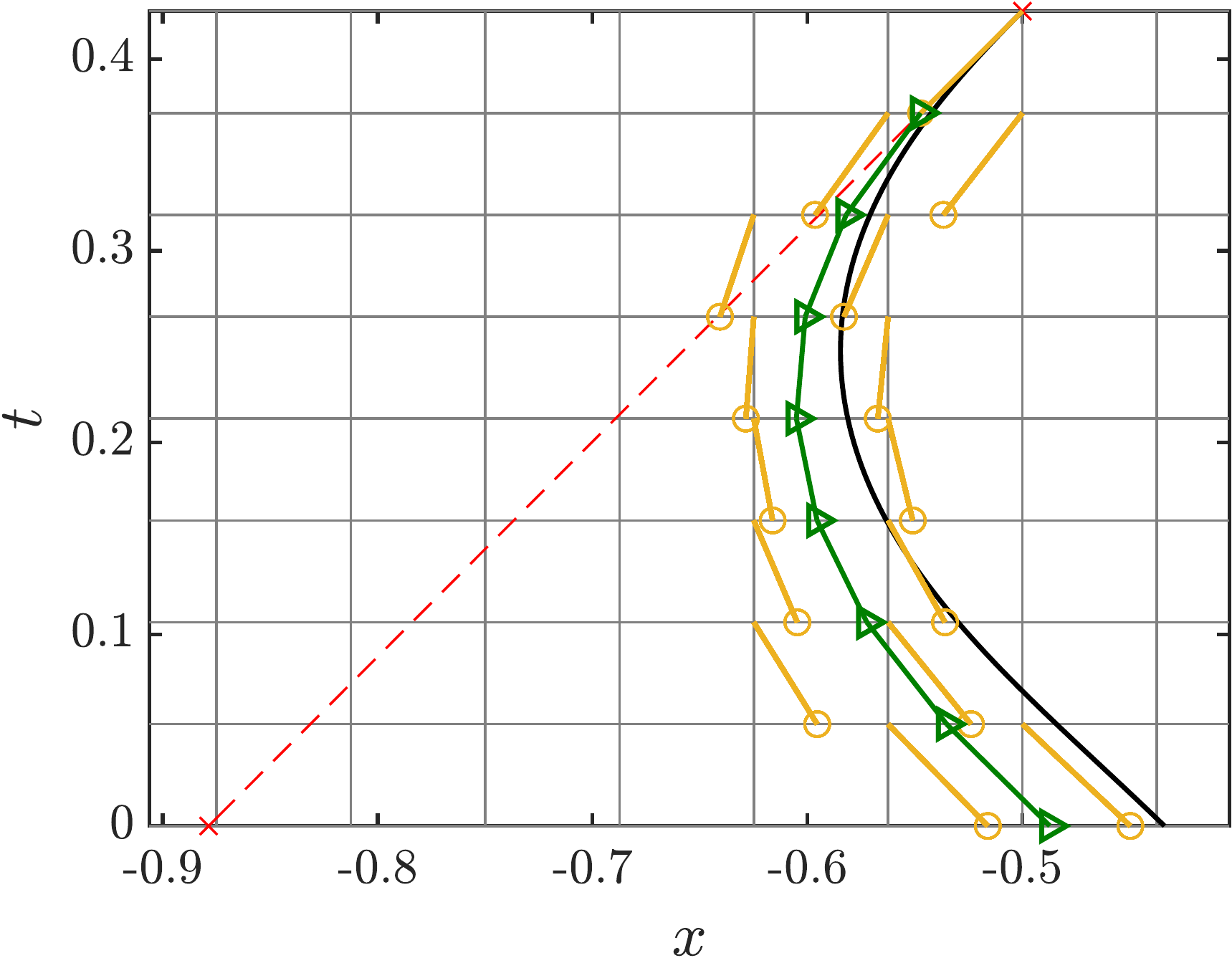}
}
\caption{Evolution of a local coarse-grid characteristic of advection problem \eqref{eq:ad} with wave-speed \eqref{eq:alpha4} using piecewise linear interpolation and backtracking of fine-grid characteristics obtained with an ERK method of accuracy $r = 1$.
Note that only a subset of the spatial domain $x \in \Omega = (-1,1)$ is shown to better highlight the detail of the characteristic.
The black curve is the exact coarse-grid characteristic.
The green curve is the coarse-grid characteristic approximated with the interpolation strategy (i.e., the triangle marker at time $t= k \delta t$ is $c_i^{(k)}$ from \eqref{eq:coarse_depart_interp}).
The gold lines are the fine-grid characteristics that are the nearest neighbors of the approximate coarse-grid characteristic (i.e., the left and right circle markers at time $t= k \delta t$ are, respectively, $f_{\textrm{W}}^{(k)}$ and $f_{\textrm{E}}^{(k)}$ from \eqref{eq:coarse_depart_interp}). 
These fine-grid characteristics were determined by a single ERK step of size $\delta t$.
The red dashed line is the coarse-grid characteristic approximated by a single ERK step of size $m \delta t = 8 \delta t$.
\label{fig:char_interp}
}
\end{figure}

For simplicity of notation, let us only consider the first coarse time interval ${t \in [0, m \delta t]}$.
Let $c^{(k)}_i$, $k \in \{0, 1, \ldots, m-1\}$, denote our approximation to the local coarse-grid characteristic $\xi_i^{(0, m \delta t)}(t)$ at time $t = k \delta t$, ${c^{(k)}_i \approx \xi_i^{(0, m \delta t)}( k \delta t)}$.
For shorthand, denote departure points of the local fine-grid characteristics on this interval by
$f^{(k)}_i \equiv \xi_i^{(k \delta t, \delta t)}(k \delta t)$, ${k \in \{0, 1, \ldots, m-1\}}$.
Over the last fine-grid subinterval $t \in [(m-1) \delta t, m \delta t]$, the approximate coarse-grid characteristic is the same as the fine-grid characteristic, since they both arrive at $(x,t) = (x_i, m \delta t)$, and therefore they intersect the $x$ axis at the same location,
\begin{align}
c^{(m-1)}_i = f^{(m-1)}_i.
\end{align}
Using this as a final-time condition, the remaining intersection points of the coarse-grid characteristic can be estimated by carrying out the following interpolating update in sequence 
\begin{align} \label{eq:coarse_depart_interp}
c^{(k)}_i
=
\frac{f^{(k)}_{\textrm{E}} - f^{(k)}_{\textrm{W}}}{h} 
\left( 
c^{(k+1)}_i
-
x_{\textrm{E}}
\right)
+
f^{(k)}_{\textrm{E}}
,
\quad
\textrm{for }
k = m-2, \ldots, 1, 0,
\end{align}
where the integers $\textrm{E}$ and $\textrm{W}$ are such that $x_{\textrm{E}}$ and $x_{\textrm{W}}$ are the east- and west-neighboring mesh points of $c^{(k+1)}_i$, respectively.
A schematic of this procedure is shown in \Cref{fig:char_interp}.
Upon completing the iteration \eqref{eq:coarse_depart_interp}, we have an  approximation for the $i$th coarse-grid departure point, ${c^{(0)}_i \approx \xi_{i}^{(0, m \delta t)}(0)}$.
The update formula \eqref{eq:coarse_depart_interp} is based on piecewise linear nearest neighbor interpolation to estimate $c^{(k)}_i$. 
The example in \Cref{fig:char_interp} shows how this strategy has the potential to approximate coarse-grid departure points much more accurately than a single $m \delta t$-sized step of the ERK scheme used for fine-grid characteristics.\footnote{In fact, if the wave-speed is spatially independent, then the strategy shown in \Cref{fig:char_interp} yields the same estimate for coarse-grid departure points as taking $m$ steps of size $\delta t$ with the fine-grid ERK scheme.}

Recall that the motivation for the proposed strategy was to estimate coarse-grid departure points in a way that is less expensive than taking $m$ ERK steps of size $\delta t$.
However, since this linear interpolation and backtracking strategy requires taking $m-1$ steps, it cannot, on the first coarse level, be significantly cheaper than using $m$ steps of an ERK scheme.\footnote{Whether it is cheaper or not depends on the number of stages of the ERK scheme and the cost of evaluating the wave-speed. Recall that an $s$-stage ERK scheme requires $s$ evaluations of the wave-speed per time-step. The linear interpolation strategy requires no evaluations of the wave-speed. In any event, the number of FLOPs for either strategy scales as ${\cal O}(m)$.}
However, supposing that the linear interpolation strategy yields sufficiently accurate departure points, in the sense that its use does not lead to a strong deterioration of MGRIT convergence, then it has the significant advantage over stepping at the fine-grid-scale with an ERK scheme that it becomes cheaper on coarser levels in a way that makes the cost scalable to multiple levels.
That is, say, for example, we have a three-level method in which we coarsen by $m$ on each level. The linear interpolation strategy takes ${\cal O}(m)$ work per coarse time-step to estimate departure points on the first coarse level, but if it is then applied recursively on the coarsest level, it requires only ${\cal O}(m)$ work there to estimate a departure point.
In general, if the strategy is applied recursively throughout a multilevel solver, it requires only ${\cal O}(m)$ work to estimate a departure point, independent of the level it occurs on.
In contrast, if an ERK method is to be used to estimate departure points on coarse levels, by our previous arguments regarding the inaccuracy of taking large time-steps, it must do so by taking many small time-steps, with size of order the fine-grid time-step $\delta t$. That is, using an ERK method to estimate a coarse-grid departure point on a coarse level $\ell \in \mathbb{N}$ requires ${\cal O}(m^{\ell})$ work. 

Finally, we remark that the linear interpolation and backtracking strategy proposed here is more expensive from a memory perspective, since estimating departure points on a coarse level requires storing all departure points on the level above it. 
%

\subsection{Multilevel setting}
\label{sec:multilevel}

In this section, we generalize the two-level, backward Euler coarse-grid operator \eqref{eq:Psi_Bp_def2} from the previous section so that it can be applied within a multilevel MGRIT algorithm. 
Let $\ell \in \mathbb{N}_0$ be the level index, and assume that the time-step size on level $\ell$ is $m^{\ell} \delta t$.
To begin, we introduce the shorthand for the following function on level $\ell \in \mathbb{N}$,
\begin{align}
\label{eq:varphi_p_multilevel_def}
\bm{\varphi}_{p+1}^{(t_n, m^{\ell} \delta t)} := (-1)^{p+1}
\left(  f_{p+1} \big( \bm{\varepsilon}^{(t_n, m^{\ell} \delta t)} \big) - \sum \limits_{k = 0}^{m-1} f_{p+1} \big( \bm{\varepsilon}^{(t_{n} + k m^{\ell-1} \delta t, m^{\ell-1} \delta t)} \big) \right),
\end{align}
which generalizes the function $\bm{\varphi}_{p+1}^{(t_n, m \delta t)}$ on level $\ell = 1$ defined in \eqref{eq:varphi_p_def}. 
Note that $\bm{\varphi}_{p+1}^{(t_n, m^{\ell} \delta t)}$ approximates the coefficient vector appearing in the leading-order term of the difference between the level $\ell$ coarse-grid semi-Lagrangian operator, and the ideal coarse-grid operator defined by stepping across the same interval $m$ times with the associated level $\ell-1$ semi-Lagrangian operators,
$
{\cal S}_{p, \infty}^{(t_n, m^{\ell} \delta t)} - \prod_{k = 0}^{m-1} {\cal S}_{p, \infty}^{(t_n + k m^{\ell-1} \delta t, m^{\ell-1} \delta t)}
$.

To develop a multilevel operator based on the backward Euler operator \eqref{eq:Psi_Bp_def2}, it is first instructive to consider a three-level algorithm. 
Given the backward Euler operators $\Phi^{(t_n + k m \delta t, m \delta t)} = {\cal B}_{p+1}^{(t_n+ k m \delta t, m \delta t)} {\cal S}_{p,\infty}^{(t_n + k m \delta t, m \delta t)}$ on level ${\ell = 1}$ for ${k \in \{0, \ldots, m-1\}}$, consider the associated ideal operator on level $\ell = 2$ and the following sequence of approximations to it,
\begingroup
\allowdisplaybreaks
\begin{align}
\Phi_{\rm ideal}^{(t_n, m^2 \delta t)} 
&=
\prod \limits_{k = 0}^{m-1} 
\Phi^{(t_n + k m \delta t, m \delta t)}
=
\prod \limits_{k = 0}^{m-1} 
\Big(
{\cal B}_{p+1}^{(t_n+ k m \delta t, m \delta t)} {\cal S}_{p,r_*}^{(t_n + k m \delta t, m \delta t)}
\Big), \\
\label{eq:BE_multi_approx1}
&\approx
\Bigg(
\prod \limits_{k = 0}^{m-1} 
{\cal B}_{p+1}^{(t_n+ k m \delta t, m \delta t)} 
\Bigg)
\Bigg(
\prod \limits_{k = 0}^{m-1} 
{\cal S}_{p,r_*}^{(t_n + k m \delta t, m \delta t)}
\Bigg), \\
\label{eq:BE_multi_approx2}
&\approx 
\Bigg(
\prod \limits_{k = 0}^{m-1} 
{\cal B}_{p+1}^{(t_n+ k m \delta t, m \delta t)}
\Bigg)
\Big(
{\cal B}_{p+1}^{(t_n, m^2 \delta t)} 
{\cal S}_{p,r_*}^{(t_n, m^2 \delta t)}\Big), \\
&=
\Bigg(
\prod \limits_{k = 0}^{m-1} 
\Bigg[
I - 
\diag
\Big( 
\bm{\varphi}_{p+1}^{(t_n + k m \delta t, m \delta t)} 
\Big)
{\cal D}_{p+1}
\Bigg]^{-1}
\Bigg)
\times
\\
\notag
&
\quad \quad \quad 
\Big[
I - 
\diag
\Big( 
\bm{\varphi}_{p+1}^{(t_n, m^2 \delta t)} 
\Big)
{\cal D}_{p+1}
\Big]^{-1} 
{\cal S}_{p,r_*}^{(t_n, m^2 \delta t)}, 
\\
\label{eq:BE_multi_approx3}
&\approx
\Bigg[
I - 
\diag
\Bigg( 
\sum \limits_{k = 0}^{m-1} \bm{\varphi}_{p+1}^{(t_n + k m \delta t, m \delta t)} 
\Bigg)
{\cal D}_{p+1}
\Bigg]^{-1}
\times
\\
\notag
&
\quad \quad \quad 
\Big[
I - 
\diag
\Big( 
\bm{\varphi}_{p+1}^{(t_n, m^2 \delta t)} {\cal D}_{p+1}
\Big)
\Big]^{-1} 
{\cal S}_{p,r_*}^{(t_n, m^2 \delta t)}, \\
\label{eq:BE_multi_approx4}
&\approx
\Bigg[
I - 
\diag
\Bigg( 
\sum \limits_{k = 0}^{m-1} \bm{\varphi}_{p+1}^{(t_n+k m \delta t, m \delta t)} 
+
\bm{\varphi}_{p+1}^{(t_n, m^2 \delta t)}
\Bigg)
{\cal D}_{p+1}
\Bigg]^{-1} 
{\cal S}_{p,r_*}^{(t_n, m^2 \delta t)}.
\end{align}
\endgroup
The approximation in \eqref{eq:BE_multi_approx1} is that the backward Euler and semi-Lagrangian operators commute. 
Only in the case of spatially independent wave-speed functions do these two operators commute (since the diagonal matrices built from the various $\bm{\varphi}_{p+1}$ vectors are constant, and ${\cal D}_{p+1}$ and the semi-Lagrangian operators are circulant). 
In any event, the forthcoming numerical results show this approximation is accurate enough to obtain fast MGRIT convergence when the wave-speed does depend on space.

In \eqref{eq:BE_multi_approx2}, $m$ successive $m \delta t$-sized semi-Lagrangian steps have been approximated using our existing two-level approximation. That is, the $m$ steps are replaced by a single $m^2 \delta t$-sized semi-Lagrangian step followed by a backward Euler step that approximately corrects for the lowest-order difference between their truncation errors. 

The approximation in \eqref{eq:BE_multi_approx3} is pulling the $m$ backward Euler steps under the inverse, and keeping only the lowest-order terms in their product, recalling ${\cal D}_{p+1} \bm{v} = {\cal O}(h^{p+1})$.
This approximation can be understood as a Taylor series interpretation of the standard rediscretization approach typically employed in MGRIT for backward Euler discretizations, in which $m$ backward Euler steps are approximated on the coarse level with a single backward Euler step using a time-step size that is $m$ times larger. 

Finally, \eqref{eq:BE_multi_approx4} arises from placing the two backward Euler matrices under a single inverse, taking their product, and then truncating the highest-order term, which is proportional to ${\cal D}_{p+1} {\cal D}_{p+1}$.
Notice that \eqref{eq:BE_multi_approx4} has the same structure as the operator \eqref{eq:Psi_Bp_def2} on level $\ell = 1$ proposed for the two-level algorithm, since it is a semi-Lagrangian step followed by a backward Euler correction.
Based on this, we propose the following time-stepping operators on level $\ell > 0$ for evolving solutions from $t_n \to t_n + m^{\ell} \delta t$,
\begin{align} \label{eq:Phi_multilevel}
\Phi^{(t_n, m^{\ell} \delta t)}
=
\left[ I - \diag \Big(\bm{\sigma}_{p+1}^{(t_n, m^{\ell} \delta t)} \Big) {\cal D}_{p+1} \right]^{-1}
\,
{\cal S}_{p, r_*}^{(t_n, m^{\ell} \delta t)},
\quad
\ell \in \mathbb{N},
\end{align}
in which the coefficient vector is defined recursively by
\begin{align} \label{eq:nu_p_def}
\bm{\sigma}_{p+1}^{(t_n, m^{\ell} \delta t)} = 
\begin{cases}
\displaystyle{\bm{\varphi}_{p+1}^{(t_n,  m^{\ell} \delta t)}}, 
\quad &\ell = 1, 
\\[1ex]
\displaystyle{\sum \limits_{k = 0}^{m-1} \bm{\sigma}_{p+1}^{(t_n+km^{\ell-1}\delta t,m^{\ell-1}\delta t)} + \bm{\varphi}_{p+1}^{(t_n,m^{\ell} \delta t)}}, 
\quad &\ell > 1.
\end{cases}
\end{align}

%
%
\renewcommand{\arraystretch}{1.1}
\begin{table}[b!]
  \centering
  \begin{tabular}{| c | c | ccc | ccc | ccc |}  
  \cline{3-11}
     \multicolumn{2}{ c|}{} & \multicolumn{3}{ c|}{$\alpha = \eqref{eq:alpha1}$} & \multicolumn{3}{ c|}{$\alpha = \eqref{eq:alpha2}$} & \multicolumn{3}{ c|}{$\alpha = \eqref{eq:alpha4}$}\\\cline{3-11}
   \multicolumn{2}{ c|}{} & \multicolumn{3}{ c|}{$m$} & \multicolumn{3}{ c|}{$m$} & \multicolumn{3}{ c|}{$m$} \\\hline
$p,r$ & $n_x \times n_t$ & 4 & 8 & 16 & 4 & 8 & 16 & 4 & 8 & 16 \\\hline
\Xhline{2\arrayrulewidth}
{\multirow{3}{*}{$1,1$}}
 & $2^{8} \times 2^{10}$  & 14 & 12 & 11       & 12 & 11 & 11     & 11 & 11 & 12  \\\cline{2-11}
 & $2^{10} \times 2^{12}$ & 14 & 12 & 11       & 12 & 11 & 11     & 12 & 12 & 12  \\\cline{2-11}
 & $2^{12} \times 2^{14}$ & 14 & 13 & 11       & 14 & 11 & 11     & 14 & 13 & 13 \\\hline
\Xhline{2\arrayrulewidth}
{\multirow{3}{*}{$3,3$}}
 & $2^{8} \times 2^{10}$   & 23 & 17 & 15       & 15 & 13 & 12     & 15 & 15 & 15  \\\cline{2-11}
 & $2^{10} \times 2^{12}$ & 23 & 17 & 15       & 16 & 14 & 13     & 16 & 16 & 16  \\\cline{2-11}
 & $2^{12} \times 2^{14}$ & 23 & 18 & 16       & 21 & 15 & 13     & 20 & 19 & 19 \\\hline
\Xhline{2\arrayrulewidth}
{\multirow{3}{*}{$5,5$}}
 & $2^{8} \times 2^{10}$   & 32 & 22 & 18        & 19 & 16 & 14     & 17 & 19 & 17   \\\cline{2-11}
 & $2^{10} \times 2^{12}$ & 34 & 24 & 20       & 19 & 16 & 15     & 19 & 19 & 20  \\\cline{2-11}
 & $2^{12} \times 2^{14}$ & 34 & 24 & 21        & 27 & 18 & 15     & 23 & 24 & 25 \\\hline
 \end{tabular}
  \caption{
	Number of MGRIT V-cycles to converge.
	The PDE is \eqref{eq:ad} with wave-speed $\alpha(x,t)$ indicated in the top row of the table.
	The fine-grid operator is ${\cal S}_{p,r}^{(t_n, \delta t)}$, and the coarse-grid operator is the dissipatively corrected operator \eqref{eq:Phi_multilevel}.
	Departure points on coarse levels are located by recursively applying the strategy from \Cref{sec:backtracking}.
	A coarsening factor of $m$ is used on all levels.
  \label{tab:multilevel_iters}
}
\end{table}

We now present results of our numerical tests using the coarse-grid operator \eqref{eq:Phi_multilevel}. 
In these tests, we use a constant coarsening factor of $m$ on all levels, and continue to coarsen until doing so would result in fewer than two points in time.
To locate departure points on coarse levels, the linear interpolation and backtracking strategy of \Cref{sec:backtracking} is employed recursively.
Furthermore, we now slightly change our strategy for approximately inverting the backward Euler matrix on coarse levels. Specifically, for each linear system, we iterate GMRES until the norm of the relative residual decreases below $10^{-2}$ or the number of iterations reaches 10.

The MGRIT V-cycle iteration counts for these tests are given in \Cref{tab:multilevel_iters}.
Notice that many of the iteration counts are almost identical to those for the two-level tests given in \Cref{tab:two_level_iters}.
We therefore conclude that the multilevel coarse-grid operator \eqref{eq:Phi_multilevel} performs as well as one could anticipate given the performance of the two-level operator that it generalizes.
We note that this is novel because it is not uncommon to see multigrid iterations strongly increase for hyperbolic problems when transitioning from two to many levels \cite{Howse_etal_2019,Yavneh_1998,Yavneh_etal_1998,Hessenthaler_etal_2018}.
These promising results also indicate that our linear interpolation and backtracking strategy of \Cref{sec:backtracking} for approximating coarse-grid departure points does so with a degree of accuracy that does not hamper multilevel MGRIT convergence, even on much coarser levels. 
%

\section{Two spatial dimensions}
\label{sec:2D}

We now extend the coarse-grid operator from the previous section to advection problems in two spatial dimensions.
\Cref{sec:2D_SL} discusses the semi-Lagrangian discretization, the coarse-grid operator is presented in \Cref{sec:2D_coarse-grid-operator}, and \Cref{sec:2D_num_results} presents numerical results.

\subsection{Semi-Lagrangian discretization}
\label{sec:2D_SL}

We now consider semi-Lagrangian discretizations of two-dimensional advection problems of the form
\begin{align} \label{eq:ad_2D}
\frac{\partial u}{\partial t} + \alpha(x,y,t) \frac{\partial u}{\partial x} + \beta(x,y,t) \frac{\partial u}{\partial y} = 0, \quad (x, y, t) \in \Omega \times (0,T],
\end{align}
with initial condition $u(x,y,0) = u_0(x,y)$, spatial domain $\Omega \subset \mathbb{R}^2$, and solution $u$ subject to periodic boundary conditions on $\partial \Omega$.
Specifically, our numerical tests for this two-dimensional problem will use the initial condition $u(x,y,0) = \sin^2 \big[ \tfrac{\pi}{2}(x-1) \big] \sin^2 \big[\tfrac{\pi}{2} (y-1) \big]$, and the spatial domain $\Omega = (-1,1)^2$.
The semi-Lagrangian discretizations we consider of \eqref{eq:ad_2D} are a straightforward generalization of those described in \Cref{sec:discretization} for the one-dimensional problem.

We define a discrete mesh on $\Omega$ as the tensor product of one-dimensional meshes in the $x$- and $y$-directions, respectively, both of which we assume are composed of $n_x$ points equispaced by a distance of $h$.
Let $(x, y, t) = (\xi(t), \eta(t), t)$ denote a characteristic of \eqref{eq:ad_2D}, then the Lagrangian formulation of \eqref{eq:ad_2D} reads 
\begin{align} 
\label{eq:Lag_2d}
\frac{\d }{\d t} \xi(t) = \alpha( \xi(t), \eta(t), t ), 
\quad
\frac{\d }{\d t} \eta(t) = \beta( \xi(t), \eta(t), t ),  
\quad
\frac{\d }{\d t} u(\xi(t), \eta(t), t) = 0.
\end{align}
Define $\big(\xi^{(t_n,\delta t)}_{ij}(t), \eta^{(t_n,\delta t)}_{ij}(t) \big)$ as the local characteristic that passes through the \textit{arrival point} ${(x,y,t) = \big(\xi^{(t_{n},\delta t)}_{ij}(t_{n+1}), \eta^{(t_{n},\delta t)}_{ij}(t_{n+1}), t_{n+1} \big)}$. 
Then, the associated  \textit{departure point} ${(x,y,t) = \big(\xi^{(t_n,\delta t)}_{ij}(t_n), \eta^{(t_n,\delta t)}_{ij}(t_n), t_n \big)}$ is given by the solution at time $t = t_n$ of the following final-value problem that holds over $t \in [t_n, t_{n+1})$
\begin{align} \label{eq:SL2D_char1}
\frac{\d }{\d t} \xi^{(t_n,\delta t)}_{ij}(t)
= 
\alpha \big( 
\xi^{(t_n,\delta t)}_{ij}(t), \eta^{(t_n,\delta t)}_{ij}(t), t 
\big), 
\quad \xi^{(t_n,\delta t)}_{ij}(t_{n+1}) = x_i, \\
\label{eq:SL2D_char2}
\frac{\d }{\d t} \eta^{(t_n,\delta t)}_{ij}(t)
= 
\beta \big( 
\xi^{(t_n,\delta t)}_{ij}(t), \eta^{(t_n,\delta t)}_{ij}(t), t 
\big), 
\quad \eta^{(t_n,\delta t)}_{ij}(t_{n+1}) = y_j.
\end{align}
Upon (approximately) locating the departure point of the local characteristic, the solution is estimated at it via two-dimensional polynomial interpolation through its nearest neighboring mesh points.
Generalizing what we did in the one-dimensional case, let $(x,y) = \big(x^{(t_n, \delta t)}_{ij}, y^{(t_n, \delta t)}_{ij} \big)$ be the mesh point immediately to the north-east of the departure point $(x,y) = \big(\xi^{(t_n,\delta t)}_{ij}(t_n), \eta^{(t_n,\delta t)}_{ij}(t_n)\big)$.
Then, decompose the $x$-coordinate of the departure point as $\xi^{(t_n,\delta t)}_{ij}(t_n) \equiv x^{(t_n, \delta t)}_{ij} - h \varepsilon^{(t_n, \delta t)}_{ij}, \, \varepsilon^{(t_n, \delta t)}_{ij} \in [0, 1)$, and the $y$-coordinate as $\eta^{(t_n,\delta t)}_{ij}(t_n) \equiv y^{(t_n, \delta t)}_{ij} - h \nu^{(t_n, \delta t)}_{ij}, \, \nu^{(t_n, \delta t)}_{ij} \in [0, 1)$.
The two-dimensional interpolating polynomial is then constructed through a tensor product of a one-dimensional interpolation in the $x$-direction and a one-dimensional interpolation in the $y$-direction. See \cite[pp. 61--62]{Falcone_Ferretti_2014} for further details.
%

\subsection{Coarse-grid operator}
\label{sec:2D_coarse-grid-operator}

We now generalize the one-dimensional considerations of \Cref{sec:var} to develop a coarse-grid operator for the two-dimensional problem \eqref{eq:ad_2D}.
In the following, the matrix ${\cal D}_{p+1} \in \mathbb{R}^{n_x \times n_x}$ is defined as in \eqref{eq:Dp+1_def}. 
Supposing spatial degrees-of-freedom are ordered row-wise lexicographically, applying $h^{-(p+1)}\left( I_{n_x} \otimes {\cal D}_{p+1} \right)$ or $h^{-(p+1)} \left( {\cal D}_{p+1} \otimes I_{n_x} \right)$ to a periodic grid vector gives an approximation to its $p+1$st partial derivative with respect to $x$ or $y$, respectively. 

\begin{lemma}[Semi-Lagrangian truncation error for $r = \infty$]
Let ${\cal D}_{p+1}$ be as in \eqref{eq:Dp+1_def}. 
Suppose that the solution $u(x,y,t)$ of \eqref{eq:ad_2D} is at least $p+1+s$ times continuously differentiable with respect to $x$ and $y$, and that the wave-speed in \eqref{eq:ad_2D} is independent of space, $(\alpha(x,y,t), \beta(x,y,t)) \equiv (\alpha(t), \beta(t))$.
Define $\bm{u}(t) \in \mathbb{R}^{n_x^2}$ as the vector composed of the PDE solution sampled in space at the mesh points and at time $t$, and let ${\cal S}_{p,\infty}^{(t_n, \delta t)}$ be the semi-Lagrangian discretization of \eqref{eq:ad_2D} that exactly locates departure points at time $t_n$.
Then, the local truncation error of this discretization can be expressed as
\begin{align} \label{eq:SL_var_trunc_2D}
\begin{split}
&\bm{u}(t_{n+1}) - {\cal S}_{p, \infty}^{(t_{n}, \delta t)} \bm{u}(t_{n})
= (-h)^{p+1} 
\bigg[
\diag \Big( f_{p+1} \big(\bm{\varepsilon}^{(t_{n}, \delta t)} \big) \Big) 
\frac{\left( I_{n_x} \otimes {\cal D}_{p+1} \right)}{h^{p+1}}
\\ 
&\quad
+
\diag \Big( f_{p+1} \big(\bm{\nu}^{(t_{n}, \delta t)} \big) \Big) 
\frac{\left( {\cal D}_{p+1} \otimes I_{n_x} \right)}{h^{p+1}}
\bigg]
\bm{u}(t_{n+1})
+ {\cal O}(h^{p+2}),
\end{split}
\end{align}
in which the entries of the vectors $\bm{\varepsilon}^{(t_{n}, \delta t)}, \bm{\nu}^{(t_{n}, \delta t)} \in \mathbb{R}^{n_x^2}$ associated with the mesh point $(x,y) = (x_i, y_j)$ are $\varepsilon_{ij}^{(t_{n}, \delta t)}$, and $\nu_{ij}^{(t_{n}, \delta t)}$, respectively. The polynomial $f_{p+1}$ is given in \eqref{eq:fpoly_def}.

Furthermore, the ideal coarse-grid operator defined by time-stepping across $t \in [t_n, t_n + m \delta t]$ with the $m$ fine-grid operators ${\cal S}_{p, \infty}^{(t_n+k \delta t, \delta t)}$, ${k = 0, \ldots, m-1}$, has a local truncation error given by
\begin{align} \label{eq:SL_var_ideal_trunc_2D} 
\begin{split}
&\bm{u}(t_{n+m})- \Bigg[ \prod_{k = 0}^{m-1} {\cal S}_{p, \infty}^{(t_n+k \delta t, \delta t)} \Bigg]
\bm{u}(t_n)  
=
(-h)^{p+1}
\Bigg[  
\diag \Bigg( \sum \limits_{k = 0}^{m-1} f_{p+1} \big(\bm{\varepsilon}^{(t_n+k \delta  t, \delta t)} \big) \Bigg) \times
\\
&\quad
\frac{\left( I_{n_x} \otimes {\cal D}_{p+1} \right)}{h^{p+1}}
+
\diag \Bigg( \sum \limits_{k = 0}^{m-1} f_{p+1} \big(\bm{\nu}^{(t_n+k \delta  t, \delta t)} \big) \Bigg)
\frac{\left( {\cal D}_{p+1} \otimes I_{n_x} \right)}{h^{p+1}}
\Bigg]
\bm{u}(t_{n+m})  
+ {\cal O}(h^{p+2}).
\end{split}
\end{align}
\end{lemma}
\begin{proof}
We omit details of these proofs since they follow analogously to those for the one-dimensional cases given in \Cref{lem:SL_trunc} and \Cref{cor:SL_ideal_trunc}. 
The one caveat is that an error estimate for two-dimensional polynomial interpolation needs to be applied (the required estimate is given as \cite[Lem. 4.11]{KrzysikThesis2021}).
\end{proof}

Generalizing the one-dimensional coarse-grid operator from \Cref{sec:exact}, we propose the following coarse-grid operator for semi-Lagrangian discretizations ${\cal S}_{p,r}^{(t_n, \delta t)}$ of the two-dimensional advection problem \eqref{eq:ad_2D}
\begin{align} 
\label{eq:BEOper_2D}
\begin{aligned}
&\Phi^{(t_n, \delta t)}_{p, r_*}
=
\left[
I_{n_x^2}
-
\diag 
\Bigg( 
f_{p+1} \big(\bm{\varepsilon}^{(t_n , m \delta t)} \big)
-
\sum \limits_{k = 0}^{m-1} f_{p+1} \big(\bm{\varepsilon}^{(t_n+k \delta  t, \delta t)} \big)
 \Bigg)
\left( I_{n_x} \otimes {\cal D}_{p+1} \right)
\right.
\\
&\quad-
\left.
\diag \Bigg( 
f_{p+1} \big(\bm{\nu}^{(t_n , m \delta t)} \big)
-
\sum \limits_{k = 0}^{m-1} f_{p+1} \big(\bm{\nu}^{(t_n+k \delta  t, \delta t)} \big) \Bigg)
\left( {\cal D}_{p+1} \otimes I_{n_x} \right)
\right]^{-1} {\cal S}_{p,r_*}^{(t_n, m \delta t)}.
\end{aligned}
\end{align}

\subsection{Numerical results}
\label{sec:2D_num_results}

\renewcommand{\arraystretch}{1.1}
\begin{table}[b!]
  \centering
  \begin{tabular}{| c | c | lll | lll |}  
  \cline{3-8}
     \multicolumn{2}{ c|}{} & \multicolumn{3}{ c|}{$(\alpha, \beta) = \eqref{eq:alpha-const_2D}$} & \multicolumn{3}{ c|}{$(\alpha, \beta) = \eqref{eq:alpha-var_2D}$}\\\cline{3-8}
   \multicolumn{2}{ c|}{} & \multicolumn{3}{ c|}{$m$} & \multicolumn{3}{ c|}{$m$} \\\hline
$p,r$ & $n_x^2 \times n_t$ & \multicolumn{1}{c}{4} & \multicolumn{1}{c}{8} & \multicolumn{1}{c|}{16} & \multicolumn{1}{c}{4} & \multicolumn{1}{c}{8} & \multicolumn{1}{c|}{16} \\\hline
\Xhline{2\arrayrulewidth}
{\multirow{4}{*}{$1,1$}}
 & $(2^6)^2 \times 2^{10}$      & 14 (50) & 12 (45) & 11 (27)      & 12 (68) & 11 (44) & 12 (27)  \\\cline{2-8}
 & $(2^7)^2 \times 2^{11}$       & 14 (86) & 12 (81) & 11 (46)      & 12 (157) & 12 (80) & 11 (46)   \\\cline{2-8}
 & $(2^8)^2 \times 2^{12}$      & 14      & 12      & 11           & 12      & 12      & 12       \\\cline{2-8}
 & $(2^{9})^2 \times 2^{13}$   & 14      & 12      & 11           & 14      & 13     & 13       \\\hline
\Xhline{2\arrayrulewidth}
{\multirow{4}{*}{$3,3$}}
 & $(2^6)^2 \times 2^{10}$      & 22 (57) & 17 (47) & 15 (27)       & 16 (70) & 15 (44) & 15 (27)  \\\cline{2-8}
 & $(2^7)^2 \times 2^{11}$       & 23 (98) & 17 (84) & 15 (47)       & 15 (170) & 16 (81) & 16 (46)   \\\cline{2-8}
 & $(2^8)^2 \times 2^{12}$      & 23      & 17      & 15            & 16      & 16      & 17       \\\cline{2-8}
 & $(2^{9})^2 \times 2^{13}$   & 23      & 17      & 16            & 19      & 18      & 17       \\\hline
\Xhline{2\arrayrulewidth}
{\multirow{4}{*}{$5,5$}}
 & $(2^6)^2 \times 2^{10}$      & 32 (63) & 22 (48) & 18 (27)       & 19 (69) & 18 (44) & 18 (27)  \\\cline{2-8}
 & $(2^7)^2 \times 2^{11}$       & 33 (110) & 23 (88) & 20 (48)      & 19 (169) & 19 (80) & 20 (46)   \\\cline{2-8}
 & $(2^8)^2 \times 2^{12}$      & 34      & 24     & 20           & 19      & 21      & 22       \\\cline{2-8}
 & $(2^{9})^2 \times 2^{13}$   & 34          & 24      & 21            & 24      & 23      & 23       \\\hline
 \end{tabular}
  \caption{
	Number of MGRIT V-cycles to reach convergence on the two-dimensional advection problem \eqref{eq:ad_2D}, with wave-speed indicated in the top row of the table.
	The fine-grid discretization is ${\cal S}_{p,r}^{(t_n, \delta t)}$, and the coarse-grid operator is the dissipatively corrected semi-Lagrangian operator \eqref{eq:BEOper_2D}. 
	For the two smallest mesh resolutions $(n_x^2, n_t) = \big((2^6)^2 \times 2^{10}, (2^6)^2 \times 2^{11}\big)$, iterations reported in parenthesis corresponds to using, instead, the rediscretized semi-Lagrangian operator ${\cal S}_{p,r}^{(t_n, m^{\ell} \delta t)}$ on coarse levels.
	\label{tab:mutilevel_iters_2D}
}
\end{table}

We now present numerical results for solving the two-dimensional advection equation \eqref{eq:ad_2D} with the constant wave-speed
\begin{align} 
\label{eq:alpha-const_2D}
\big(\alpha(x, y, t), \beta(x, y, t) \big) = (1,1),
\end{align}
as well as the variable wave-speed
\begin{align}
\label{eq:alpha-var_2D}
\big(\alpha(x, y, t), \beta(x, y, t) \big) 
= 
\left( \sin^2( \pi y ) \cos \left( \frac{2 \pi t}{3.4} \right ), -\cos^2( \pi x ) \cos \left( \frac{2 \pi t}{3.4} \right ) \right).
\end{align}
Plots of the velocity field and the solution of PDE \eqref{eq:ad_2D} associated with the wave-speed \eqref{eq:alpha-var_2D} are given in \Cref{SM:fig:velocity_field_2D,SM:fig:sol_evolution_2D}, respectively. Note that the PDE solution associated with \eqref{eq:alpha-var_2D} is periodic in time, with period $3.4$. 

The numerical tests use a fine-grid time-step size of $\delta t = 0.85 h$.
%
In \Cref{tab:mutilevel_iters_2D}, we report the number of MGRIT V-cycles to converge using the dissipatively corrected operator \eqref{eq:BEOper_2D}.
On coarse levels $\ell > 1$, we use a straightforward generalization of the level $\ell = 1$ operator \eqref{eq:BEOper_2D} by following the development of the multilevel operator in \Cref{sec:multilevel} for the one-dimensional case.
Solutions of coarse-grid linear systems are approximated by iterating GMRES until either the relative residual norm is $10^{-2}$ or smaller, or the number of iterations reaches 10.
The coarse-grid semi-Lagrangian operators ${\cal S}_{p,r_*}^{(t_n, m^{\ell} \delta t)}$ used in \eqref{eq:BEOper_2D} estimate departure points using a generalization of the one-dimensional backtracking and linear interpolation strategy from \Cref{sec:backtracking}; see \Cref{SM:sec:depart_point_estimation} for details.
Finally, for the two smallest mesh resolutions, \Cref{tab:mutilevel_iters_2D} also includes iteration counts for the naive choice of simply rediscretizing the semi-Lagrangian discretization; in these rediscretization experiments, departure points on a coarse-level $\ell > 0$ are estimated accurately using $m^{\ell}$ ERK steps of size $\delta t$.

Iteration counts in \Cref{tab:mutilevel_iters_2D} for the rediscretized coarse-grid operator are large, and grow strongly as the mesh is refined. Therefore, as for the one-dimensional case, we conclude that simply rediscretizing the semi-Lagrangian coarse-grid operator does not lead to a robust MGRIT solver for our two-dimensional model problem.
In contrast, the proposed modified semi-Lagrangian coarse-grid operator \eqref{eq:BEOper_2D} yields much smaller iteration counts, that seem nearly constant as the mesh is refined. 
These sequential numerical results highlight the potential our proposed coarse-grid operator \eqref{eq:BEOper_2D} has for parallel-in-time simulations of advection problems.

\begin{remark}[Potential non-robustness for spatially varying wave-speeds]
In a small number of test problems with spatially varying wave-speeds, we have observed less favorable MGRIT convergence when using the proposed coarse-grid operator.
While we have sometimes also encountered these issues in one dimension, they seem most pronounced in two dimensions, and they appear related to the time-step size on some coarse level not being sufficiently small.
Recall that we conjectured in \Cref{rem:space_vary_estimates} that in one dimension, the error estimates for spatially independent wave-speed functions carry over to the spatially variable case up to terms of size ${\cal O}(h^{p+1} \delta t)$.
We suspect that the robustness issues we describe above are a consequence of such additional ${\cal O}(h^{p+1} \delta t)$ terms that are not captured by the proposed coarse-grid operator.
In particular, if one coarsens down to $\delta t = {\cal O}(1)$ on some coarse level, then such ${\cal O} (h^{p+1} \delta t)$ terms may no longer be small relative to the ${\cal O}(h^{p+1})$ terms that are included in the proposed coarse-grid operator.
It remains on-going work to further investigate and address this issue, potentially by deriving estimates that explicitly include the ${\cal O} (h^{p+1} \delta t)$ terms and then incorporating them into the coarse-grid operator.
Regardless, when $\delta t$ is sufficiently small on all levels, as is the case for all the reported results, these issues do not occur.
\end{remark}

\section{Parallel results}
\label{sec:parallel}

In this section we present parallel strong-scaling results for MGRIT V-cycles using the coarse-grid operators developed in earlier sections.
Since we want to demonstrate the efficacy of our proposed coarse-grid operator, we consider tests that use parallelism in time only, and we leave a space-time parallel implementation to future work.
The results were generated on Ruby, a Linux cluster at Lawrence Livermore National Laboratory consisting of 1,480 compute nodes, with 56 Intel Xeon CLX-8276L cores per node.
The results used the following node configurations: $(\textrm{number of MPI tasks}, \textrm{number of nodes}) = ((1, 1), (4,2), (16, 2), (64, 3), (256, 10),\\ (1024, 37))$.
Tests using four and 16 MPI tasks were run across two nodes since we observed this was faster than on a single node. The number of nodes used in the larger tests was chosen to ensure the number of MPI tasks per node did not exceed 28 (the number of physical cores per node).

We consider strong-scaling studies for both the one- and two-dimensional advection problems \eqref{eq:ad} and \eqref{eq:ad_2D}, respectively. 
The tests use the largest problem sizes considered previously of $n_x \times n_t = 2^{12} \times 2^{14}$, and $n_x^2 \times n_t = (2^9)^2 \times 2^{13}$ for the one- and two-dimensional problems, respectively (see \Cref{tab:multilevel_iters,tab:mutilevel_iters_2D}).
For the one-dimensional problem, we have done scaling studies using coarsening factors $m = 4, 8, 16$, as well as an aggressive coarsening strategy using $m = 16$ on the first level and $m = 4$ on all other levels. 
For this problem, the aggressive coarsening strategy yielded the largest speed-ups, so we report results for this coarsening strategy only. Based on this finding we have adopted this aggressive coarsening strategy for our two-dimensional tests.
%

%
%
%
%
%
\begin{figure}[t!]
\centerline{
\includegraphics[scale=0.345]{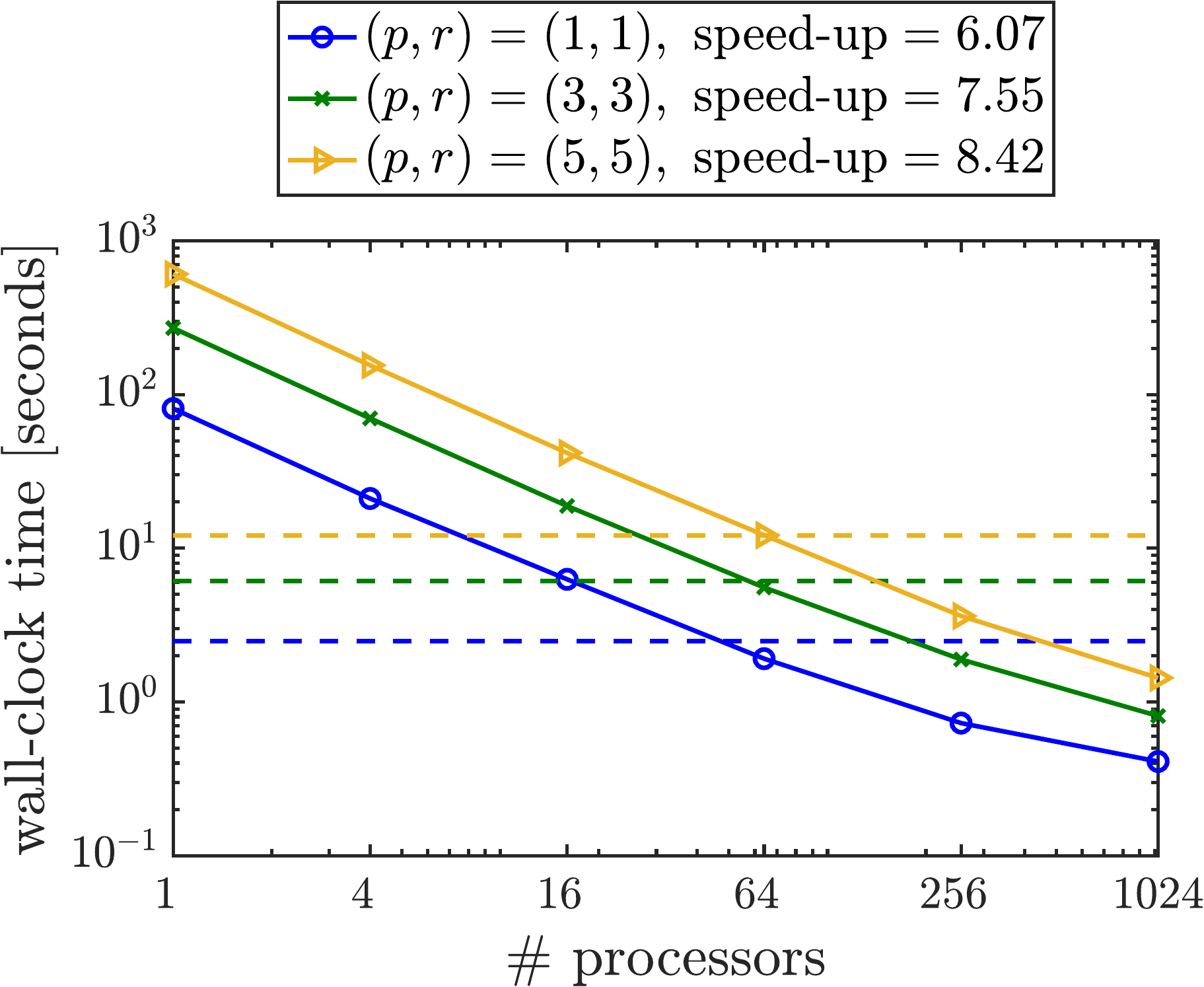}
\quad
\includegraphics[scale=0.335]{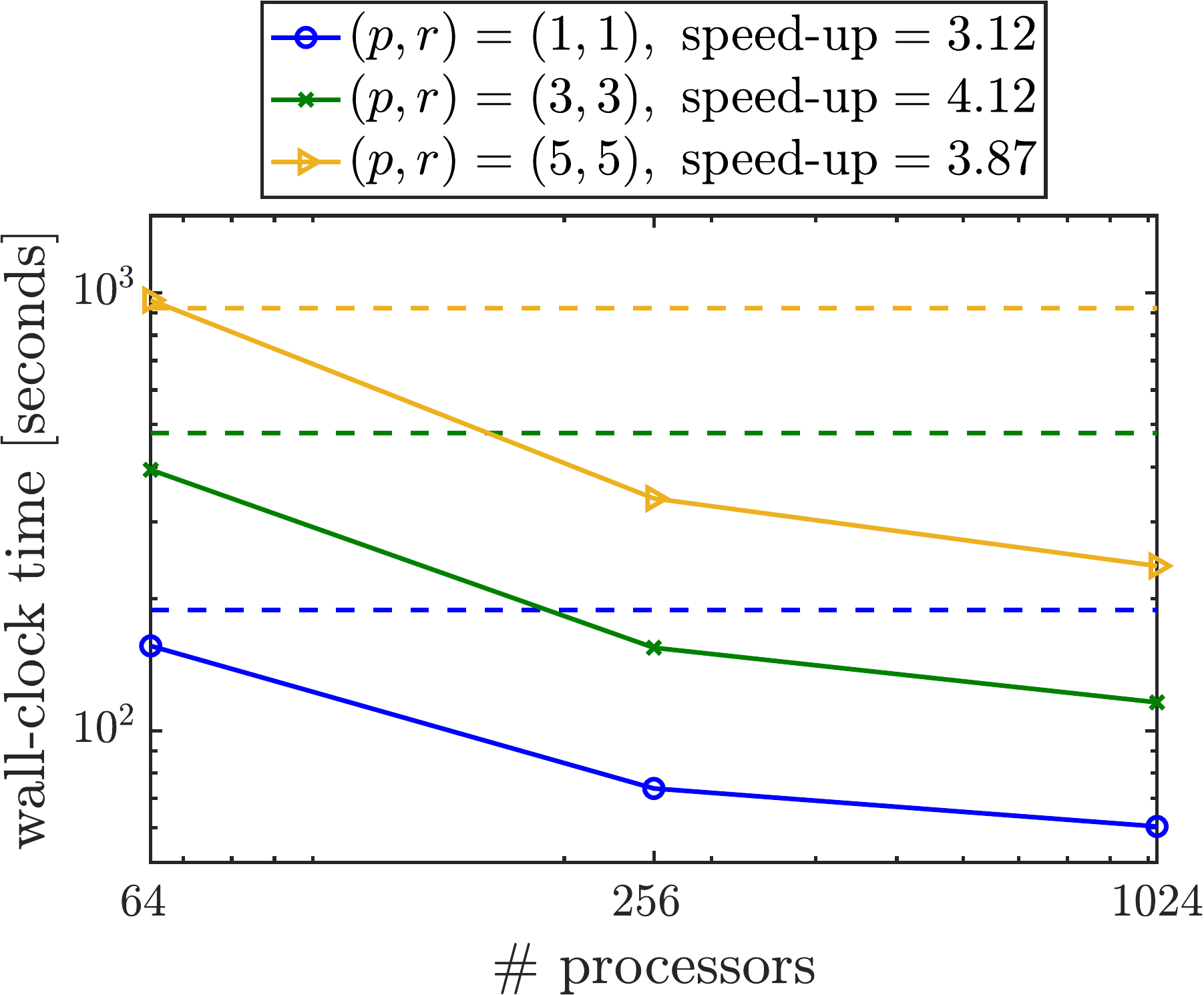}
}
\caption{Strong parallel scaling using time-only parallelism: Runtimes of MGRIT V-cycles with an aggressive coarsening factor of $m= 16$ on the first level followed by $m = 4$ on all coarser levels.
The semi-Lagrangian discretizations are of orders $(p,r) = ((1,1), (3,3), (5,5))$.
Left: The one-dimensional PDE \eqref{eq:ad} with wave-speed \eqref{eq:alpha4}, on space-time mesh having $n_x \times n_t = 2^{12} \times 2^{14}$ points.
Right: The two-dimensional PDE \eqref{eq:ad_2D} with wave-speed \eqref{eq:alpha-var_2D}, on a space-time mesh having $n_x^2 \times n_t = (2^9)^2 \times 2^{13}$ points.
Dashed lines represent runtimes of time-stepping on one processor.
Speed-ups obtained using 1024 processors are listed in the legends.
\label{fig:strong_scaling}
}
\end{figure}

Runtimes as a function of processor count are shown in \Cref{fig:strong_scaling}.
For the two-dimensional problem, a minimum of 64 processors is used since runtimes on fewer processors would have been too large.
For both one- and two-dimensional problems, the cross-over point at which MGRIT is faster than time-stepping is around 64 processors.
Notice that the speed-ups we achieve on 1024 processors for the two-dimensional problem are roughly a factor of two smaller than for the one-dimensional problem; this is perhaps unsurprising since half as many points in time are used, thus cutting the potential for speed-ups roughly in half.
Additional one-dimensional tests solving to discretization error rather than reducing the $\ell^2$-norm of the space-time residual by 10 orders of magnitude are given in \Cref{SM:sec:strong-scaling}.

Making fair comparisons between speed-ups obtained on different problems using different setups is difficult; however, it is fair to say that when using a similar number of processors, the speed-ups we obtain here are smaller than those achieved for MGRIT applied to diffusion-dominated problems, such as the heat equation \cite{Falgout_etal_2014,Falgout_etal_2017_nonlin}.
This is likely due to the increased complexity of our coarse-grid operator relative to rediscretization (as is typically used for diffusion-dominated problems), and also the somewhat higher iteration counts we require to reach convergence.
Nonetheless, the speed-ups we report are quite strong relative to speed-ups reported elsewhere in the literature for hyperbolic problems, especially since we also consider high-order discretizations \cite{Nielsen_etal_2018, Howse_etal_2019}. 
Furthermore, the speed-ups we obtain for the one-dimensional problem are somewhat comparable to those in \cite{DeSterck_etal_2021} that used optimized coarse-grid operators for constant-wave-speed advection problems; the methods from \cite{DeSterck_etal_2021}, however, are impractical for real computations and do not extend to variable wave-speeds.

\section{Conclusions}
\label{sec:conclusions}

Robust and efficient parallel-in-time integration of advection-dominated PDEs using the iterative multigrid-in-time method MGRIT, and the closely related Parareal method, is notoriously difficult. 
\textit{Rediscretizing} the fine-grid problem on coarse grids often results in MGRIT diverging on advection-dominated PDEs, despite the same technique typically leading to excellent convergence for diffusion-dominated problems.
To date, no practical coarse-grid operators have been proposed for these algorithms that are capable of providing speed-up over sequential time-stepping, even for the simplest case of constant-wave-speed linear advection.

We have considered a specific class of semi-Lagrangian discretizations for linear advection problems, including those in two spatial dimensions, and with variable wave-speeds.
For these problems, we have proposed a novel modified semi-Lagrangian coarse-grid operator that, for the first time, leads to fast and robust MGRIT convergence.
Parallel results show the coarse-grid operator can provide speed-up over sequential time-stepping, including for high-order discretizations.
The proposed coarse-grid operator applies a semi-Lagrangian discretization followed by a correction, designed so that the truncation error of this composition approximately matches that of the so-called ideal coarse-grid operator.
For the problems considered here, the correction adds dissipation to the coarse-grid semi-Lagrangian discretization via solving a linear system resembling a backward Euler time discretization of a heat-like equation.

The ideas presented here open many directions for future research, including adapting our approach to dispersive semi-Lagrangian schemes, since our current approach is only effective for dissipative semi-Lagrangian schemes.
A key component of our on-going research is the extension of the ideas presented here to enable the robust MGRIT solution of more classical method-of-lines discretizations of hyperbolic PDEs; that is, coupling a fine-grid method-of-lines discretization with a coarse-grid operator that is semi-Lagrangian in nature with a truncation error perturbation.
For constant-wave-speed problems, we have promising initial results for discretizations considered previously in \cite{DeSterck_etal_2021}, which use explicit and implicit Runge-Kutta in time, and finite differences in space.
Another direction of future work is to extend these ideas to nonlinear conservation laws.



\bibliographystyle{siamplain}
\bibliography{pit_advection_bib}


\pagebreak

\setcounter{section}{0}
\setcounter{equation}{0}
\setcounter{figure}{0}
\setcounter{table}{0}
\setcounter{page}{1}
\makeatletter
\renewcommand{\thesection}{SM\arabic{section}}
\renewcommand{\theequation}{SM\arabic{equation}}
\renewcommand{\thefigure}{SM\arabic{figure}}
\renewcommand{\thetable}{SM\arabic{table}}
\renewcommand{\thepage}{SM\arabic{page}}

\thispagestyle{plain} 

\headers{SUPPLEMENTARY MATERIALS: SL MGRIT coarse-grid operators}{H. De Sterck, R. D. Falgout, O. A. Krzysik}

\begin{center}
    \textbf{\normalsize\MakeUppercase{
    Supplementary Materials:
    Fast multigrid reduction-in-time for advection via modified semi-Lagrangian coarse-grid operators}} 
    \vspace{6ex}
\end{center}



These supplementary materials are organized as follows.
\Cref{SM:sec:SLtrunc_num_evidence} presents supporting numerical evidence for \Cref{lem:SL_ideal_pert} and the claims made in \Cref{rem:space_vary_estimates}.
\Cref{SM:sec:GMRES_FAS} describes an implementation detail relating to the use of GMRES to inexactly solve coarse-grid linear systems.
\Cref{SM:sec:one-dim-num-tests} describes some further details about the variable-wave-speed test problems we consider in one dimension.
\Cref{SM:sec:strong-scaling} provides some additional parallel strong-scaling results for the one-dimensional problem.
\Cref{SM:sec:two-dim-num-tests} describes some further details about the variable-wave-speed test problem we consider in two dimensions.
Finally, \Cref{SM:sec:depart_point_estimation} provides a strategy for estimating coarse-grid departure points in two spatial dimensions.

\section{Supporting numerical evidence for \Cref{lem:SL_ideal_pert} and claims made in \Cref{rem:space_vary_estimates}}
\label{SM:sec:SLtrunc_num_evidence}

In this section, we provide supporting numerical evidence for \Cref{lem:SL_ideal_pert} and the claims made in \Cref{rem:space_vary_estimates}. 
To do so, we test numerically to what extent (if any) the relationships \eqref{eq:SL_ideal_pert1} and \eqref{eq:SL_ideal_pert3} hold when the wave-speed function varies in space.
In particular, we consider the following two wave-speed functions
\begin{align}
\label{SM:eq:alpha2}
\alpha(x,t) &= \cos(2 \pi t), \\
\label{SM:eq:alpha4}
\alpha(x,t) &= \cos(2 \pi t) \cos(2 \pi x),
\end{align}
and we integrate from time $t_n = 0$ to time $t_n + m \delta t$, taking $m = 4$. To mimic \textit{exactly locating} departure points---as the idealized semi-Lagrangian schemes ${\cal S}_{p, \infty}^{(t_n + k \delta t, \delta t)}$ and ${\cal S}_{p, \infty}^{(t_n, m \delta t)}$ in \eqref{eq:SL_ideal_pert1} and \eqref{eq:SL_ideal_pert3} do---we use MATLAB's \texttt{ode45} with very tight tolerances to integrate backwards along local characteristics with high accuracy.

Based on the results in \Cref{lem:SL_ideal_pert}, numerically we measure the following quantity under mesh refinement in $h$:
\begin{align} \label{SM:eq:var_conj_plot_quantity}
\frac{
\left\Vert 
\prod \limits_{k = 0}^{m-1} {\cal S}_{p, \infty}^{(t_n + k \delta t, \delta t)} \bm{u}(t_n) 
- 
C^{(t_n, m \delta t)} {\cal S}_{p, \infty}^{(t_n, m \delta t)} \bm{u}(t_n)  
\right\Vert_{\infty}}{
\left\Vert 
\bm{\varphi}_{p+1}^{(t_n, m \delta t)}
\right\Vert_{\infty}
}.
\end{align}
Here, $C^{(t_n, m \delta t)}$ is either the identity matrix, or \\$C^{(t_n, m \delta t)} = {\cal B}_{p+1}^{(t_n, m \delta t)} := \left[I - \diag \Big( \bm{\varphi}_{p+1}^{(t_n, m \delta t)} \Big)
{\cal D}_{p+1}
\right]^{-1}$, which can be thought of as a \textit{correction} matrix that may  map the rediscretized coarse-grid operator ${\cal S}_{p, \infty}^{(t_n, m \delta t)} $ closer to the ideal coarse-grid operator $\prod \limits_{k = 0}^{m-1} {\cal S}_{p, \infty}^{(t_n + k \delta t, \delta t)}$.

First, in \Cref{SM:fig:lem_verification} we provide numerical verification of \Cref{lem:SL_ideal_pert}. That is, we plot the quantity \eqref{SM:eq:var_conj_plot_quantity} for the spatially independent wave-speed function \eqref{SM:eq:alpha2}.
If $C^{(t_n, m \delta t)} = I$, then from \eqref{eq:SL_ideal_pert1} we expect \eqref{SM:eq:var_conj_plot_quantity} to decay as ${\cal O}(h^{p+1})$, which is indeed what we observe, whether $\delta t = {\cal O}(h)$, or $\delta t = {\cal O}(1)$.
Furthermore, if $C^{(t_n, m \delta t)} = {\cal B}_{p+1}^{(t_n, m \delta t)}$, then from \eqref{eq:SL_ideal_pert3} we expect \eqref{SM:eq:var_conj_plot_quantity} to decay as ${\cal O}(h^{p+2})$, which is indeed what we observe, whether $\delta t = {\cal O}(h)$, or $\delta t = {\cal O}(1)$.
%

\begin{figure}[t!]
\centerline{
\includegraphics[scale = 0.375]{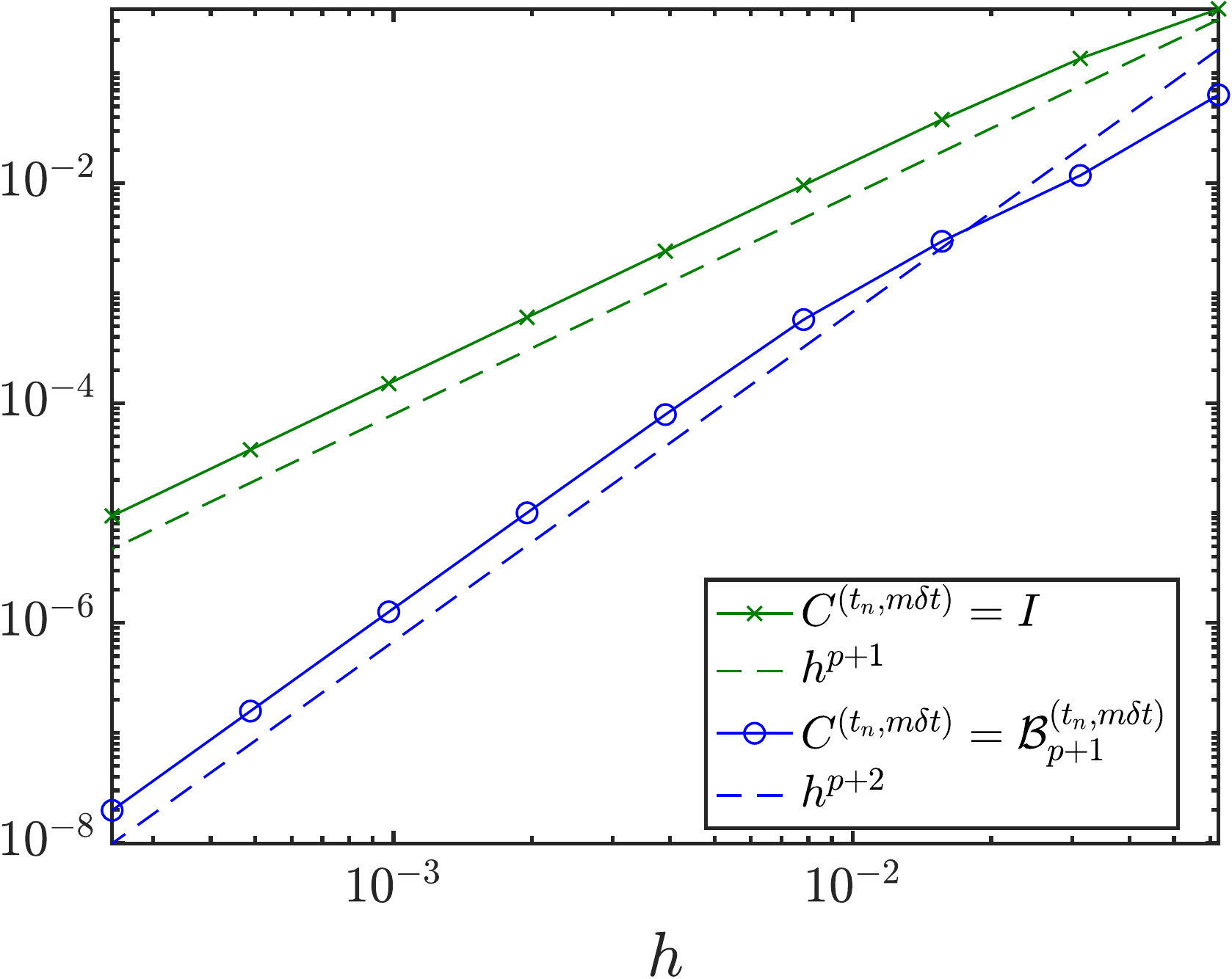}
\quad
\includegraphics[scale = 0.375]{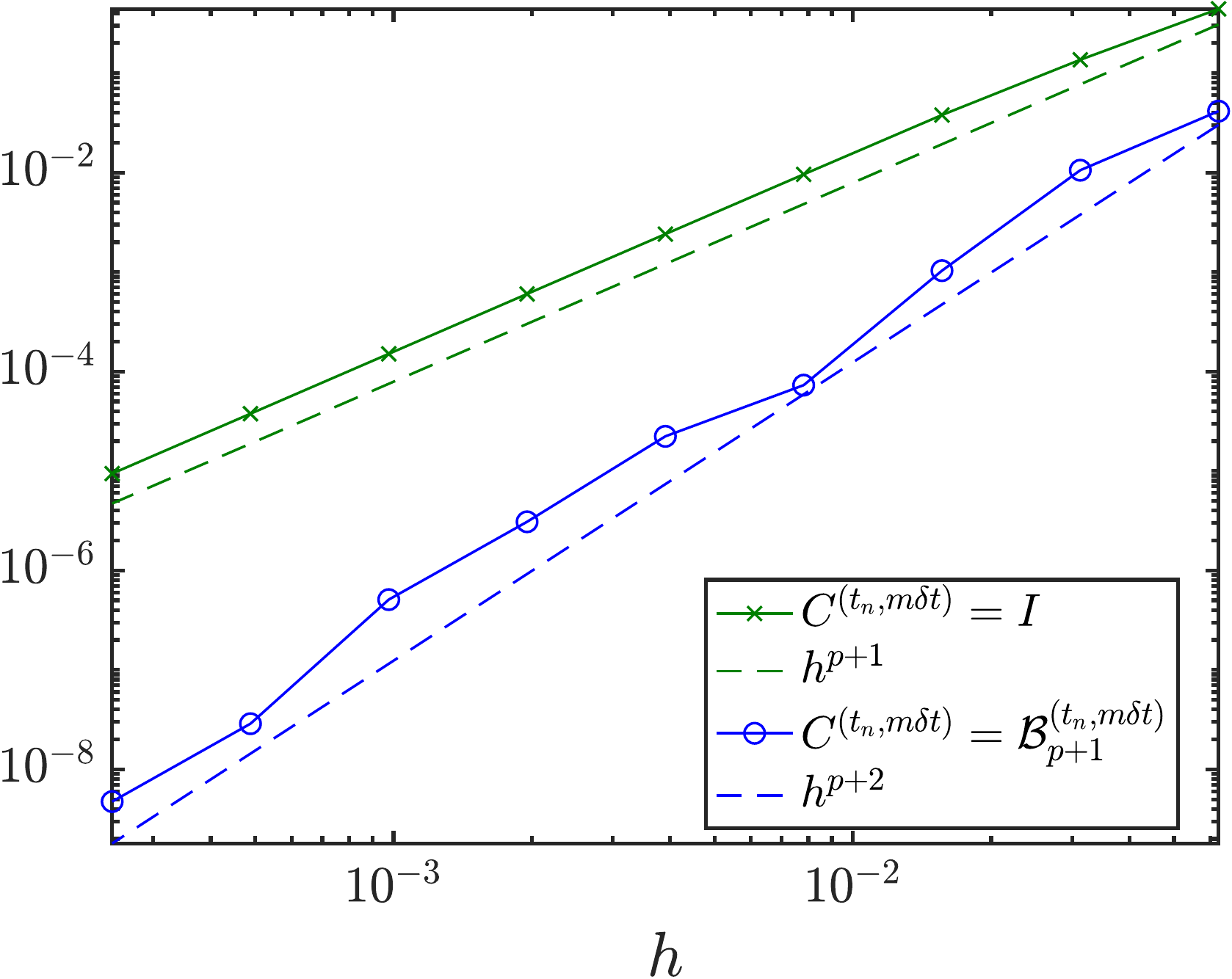}
}
\vspace{2ex}
\centerline{
\includegraphics[scale = 0.375]{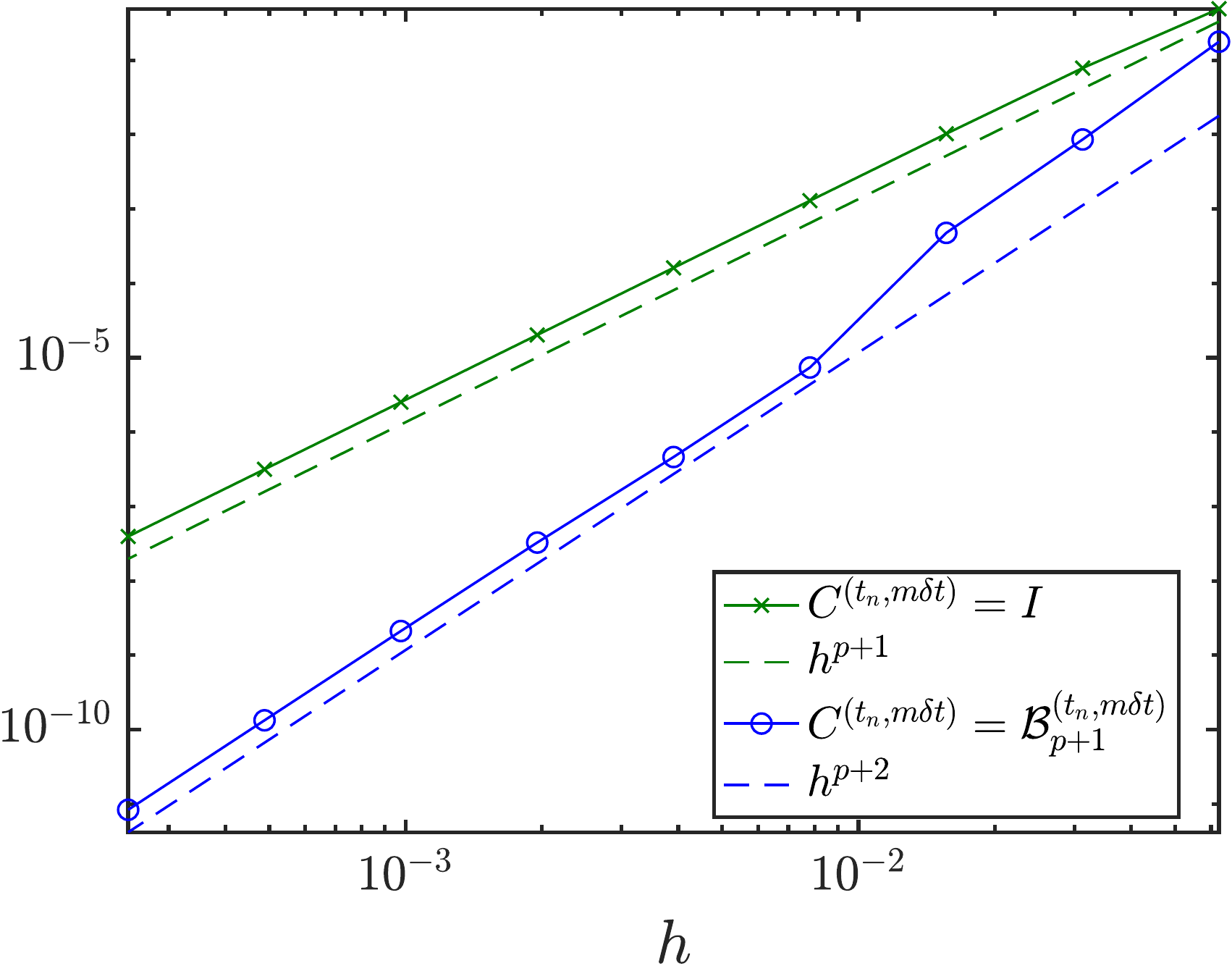}
\quad
\includegraphics[scale = 0.375]{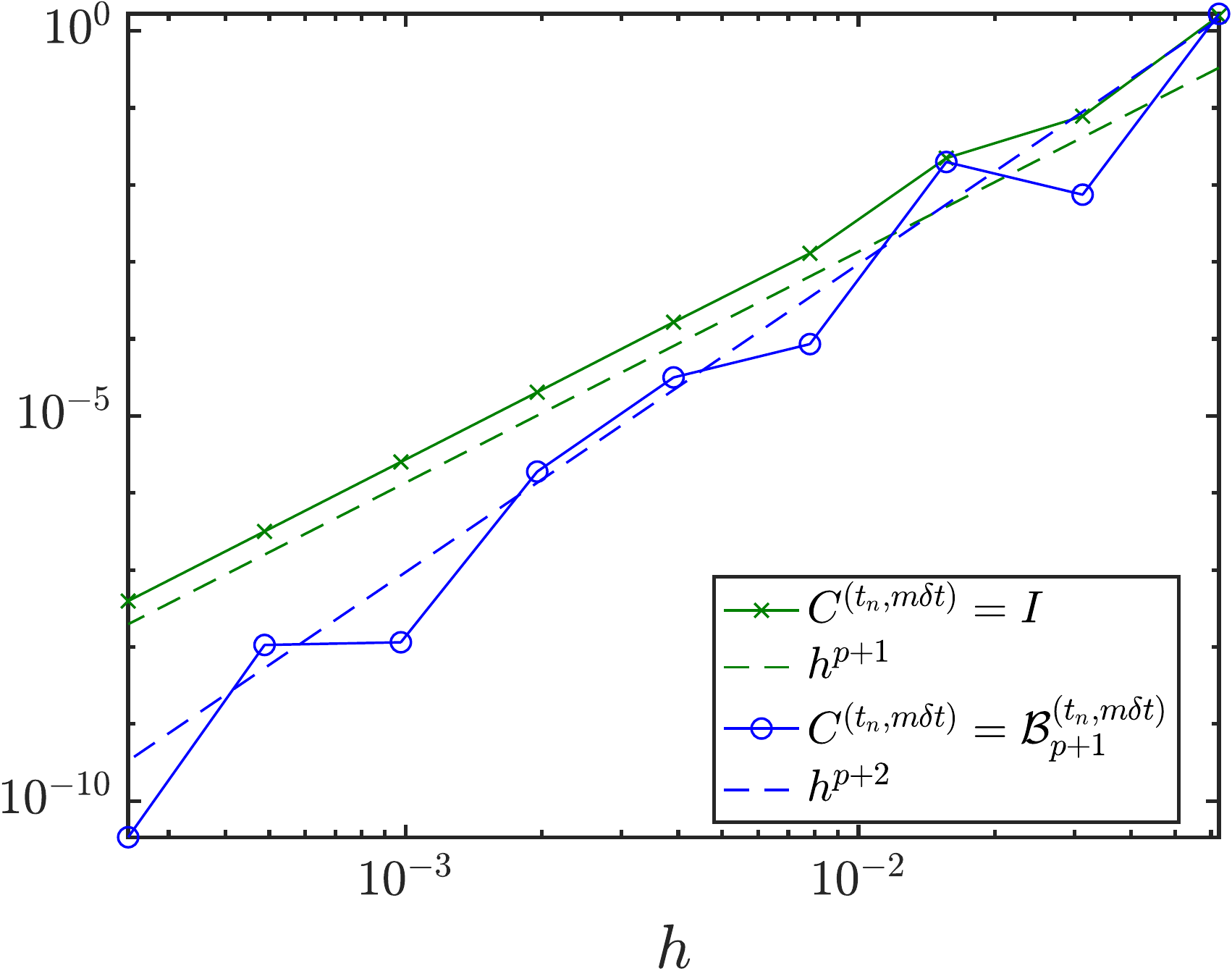}
}
\caption{
Numerical verification of \Cref{lem:SL_ideal_pert}. The quantity \eqref{SM:eq:var_conj_plot_quantity} with spatially independent wave-speed \eqref{SM:eq:alpha2} is plotted under mesh refinement in $h$.
Top: $p = 1$. Bottom $p = 2$. Left: $\delta t = 0.85 h$. Right: $\delta t = 0.85$.
\label{SM:fig:lem_verification}
}
\end{figure}

Now, in \Cref{SM:fig:var_BE_conject} we provide numerical evidence for the claims made in \Cref{rem:space_vary_estimates} regarding whether the spatially independent wave-speed estimates extend to spatially variable wave-speeds.
That is, we test whether the relationships \eqref{eq:SL_ideal_pert1} and \eqref{eq:SL_ideal_pert3} hold when the wave-speed function is given by  \eqref{SM:eq:alpha4}.
Consider when $C^{(t_n, m \delta t)} = {\cal B}_{p+1}^{(t_n, m \delta t)}$, so as verify whether the estimate \eqref{eq:SL_ideal_pert3} holds.
For $p = 1$ (top row of \Cref{SM:fig:var_BE_conject}), we see that \eqref{SM:eq:var_conj_plot_quantity} decays as ${\cal O}(h^{p+2})$ if $\delta t = {\cal O}(h)$ (top left panel), but decays as ${\cal O}(h^{p+1})$ if $\delta t = {\cal O}(1)$ (top right panel). Both of these results are consistent with estimate \eqref{eq:SL_ideal_pert3} holding up to terms of size ${\cal O}(h^{p+1} \delta t)$, as conjectured in \Cref{rem:space_vary_estimates} when $p$ is odd.
However, for the $p = 2$ case (bottom row of \Cref{SM:fig:var_BE_conject}), we see that when $C^{(t_n, m \delta t)} = {\cal B}_{p+1}^{(t_n, m \delta t)}$, the quantity \eqref{SM:eq:var_conj_plot_quantity} only decays as ${\cal O}(h^{p+1})$, even when $\delta t = {\cal O}(h)$. This is why, as stated in \Cref{rem:space_vary_estimates}, we do not believe that all of the estimates generalize to the spatially variable case when $p$ is even.

%
\begin{figure}[t!]
\centerline{
\includegraphics[scale = 0.375]{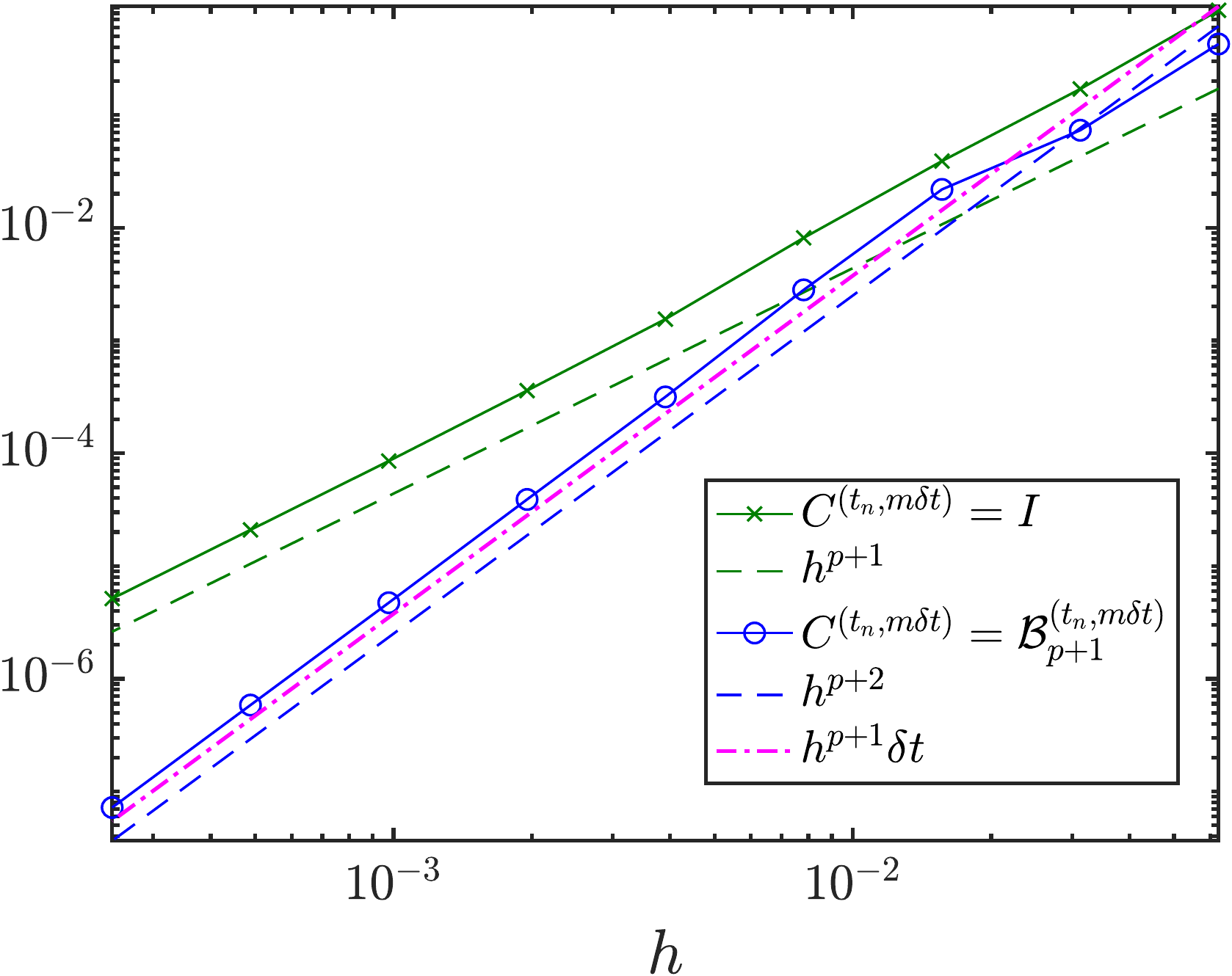}
\quad
\includegraphics[scale = 0.375]{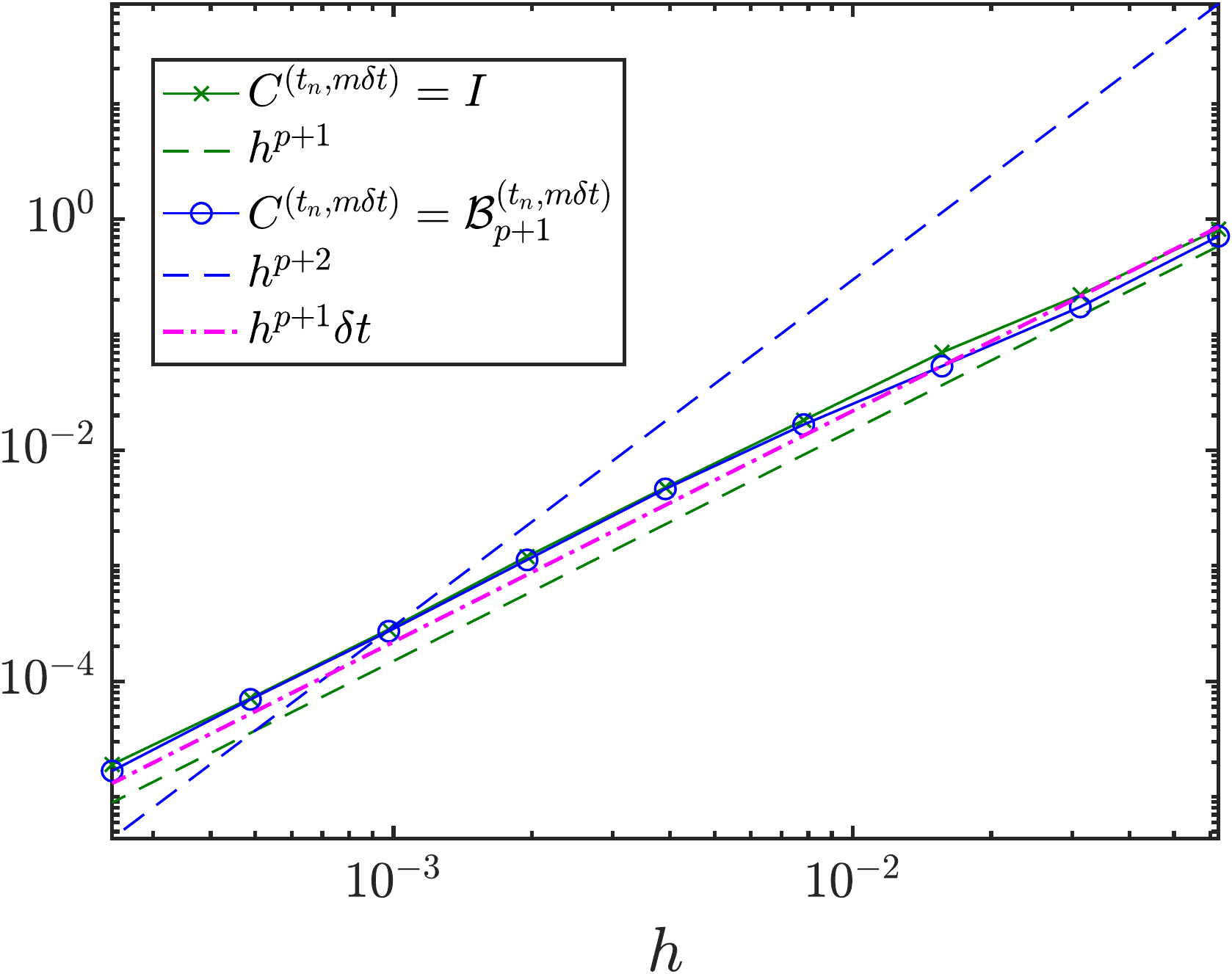}
}
\vspace{2ex}
\centerline{
\includegraphics[scale = 0.375]{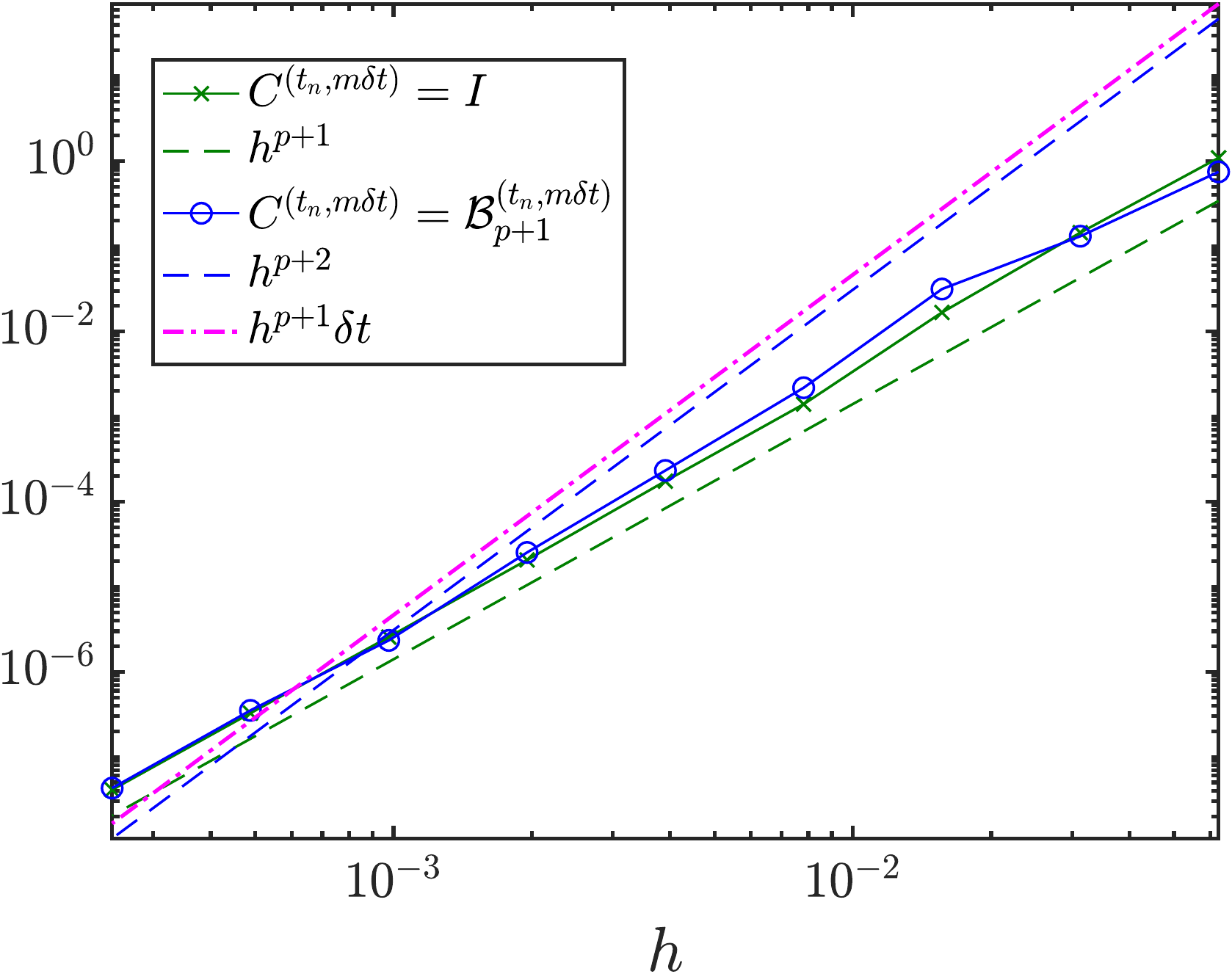}
\quad
\includegraphics[scale = 0.375]{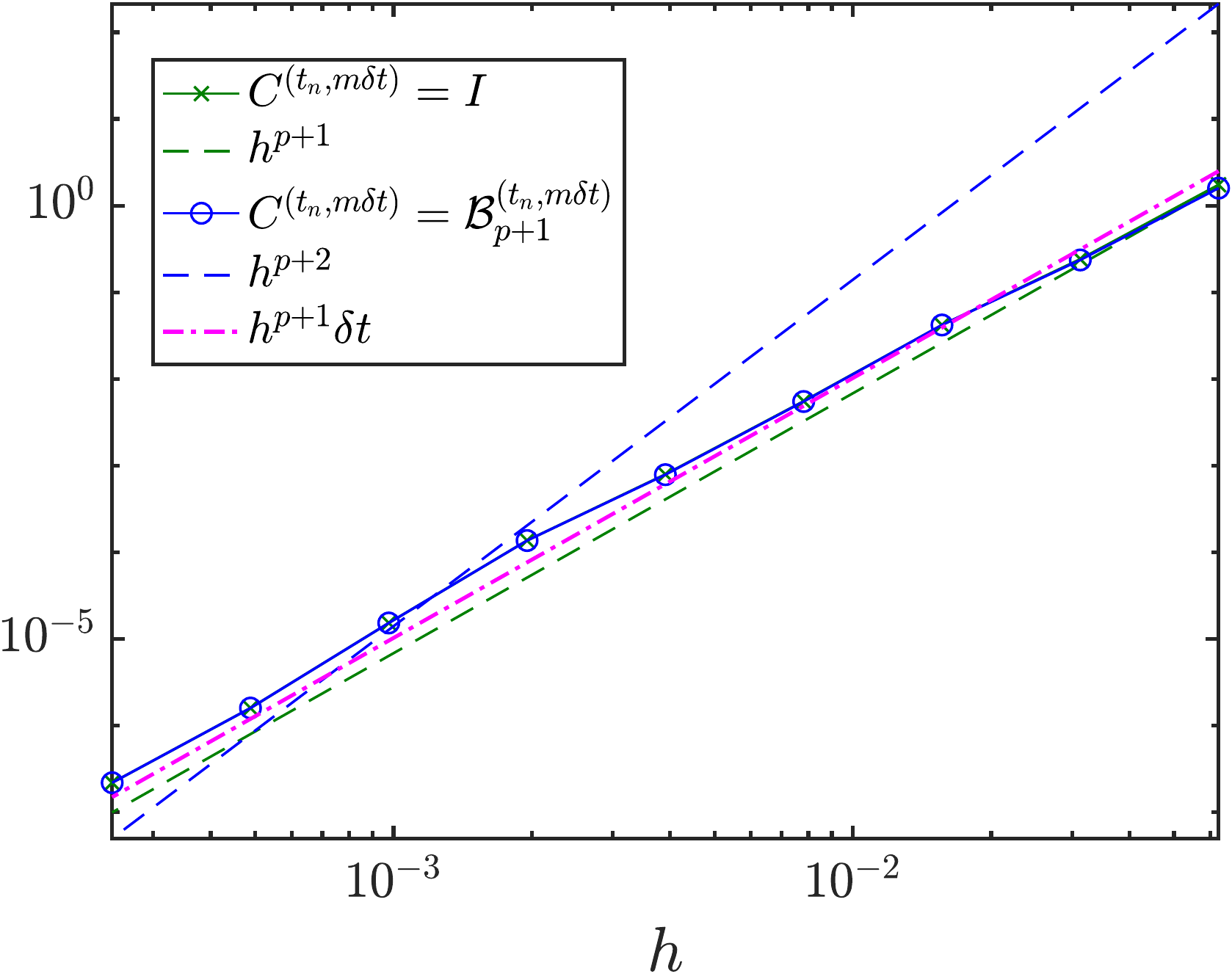}
}
\caption{
Supporting numerical evidence for \Cref{rem:space_vary_estimates}. The quantity \eqref{SM:eq:var_conj_plot_quantity} with spatially variable wave-speed \eqref{SM:eq:alpha4} is plotted under mesh refinement in $h$.
Top: $p = 1$. Bottom $p = 2$. Left: $\delta t = 0.85 h$. Right: $\delta t = 0.85$.
\label{SM:fig:var_BE_conject}
}
\end{figure}

\section{Nonlinearity introduced by GMRES}
\label{SM:sec:GMRES_FAS}

Using GMRES to inexactly solve the coarse-grid linear systems introduces nonlinearity to the problem, since the approximations generated by GMRES depend nonlinearly on the right-hand sides of the linear systems.
Applying GMRES inexactly within the full approximation scheme (FAS) framework is not straightforward (though possible), because it is important to ensure that the same GMRES polynomial is used for both the coarse-grid solve and the construction of the tau-correction term. 
To avoid this complication, we use a linear version of MGRIT, which we have implemented in XBraid (by default, XBraid uses the FAS framework).
This has the added benefit of reducing memory and computation overhead.
Note that an FAS implementation of our algorithm would not produce identical results to the linear implementation used here, since the initial GMRES residuals (and hence the GMRES polynomials) would be different.
The impact of this difference on convergence, if any, is a topic of future work.
%

\section{One-dimensional numerical test problems}
\label{SM:sec:one-dim-num-tests}

In this Section, plots of the wave-speed functions used in the numerical tests for the one-dimensional advection problem are shown. In addition, plots of the associated PDE solutions are shown.
\Cref{SM:fig:wave_and_sol_1D_alpha2} presents the case in which the wave-speed function depends on time only.
\Cref{SM:fig:wave_and_sol_1D_alpha4} presents the case in which the wave-speed function depends on both space and time.

\begin{figure}[h!]
\centerline{
\begin{tikzpicture}
  		\node at (0,0) {\includegraphics[scale = 0.35]{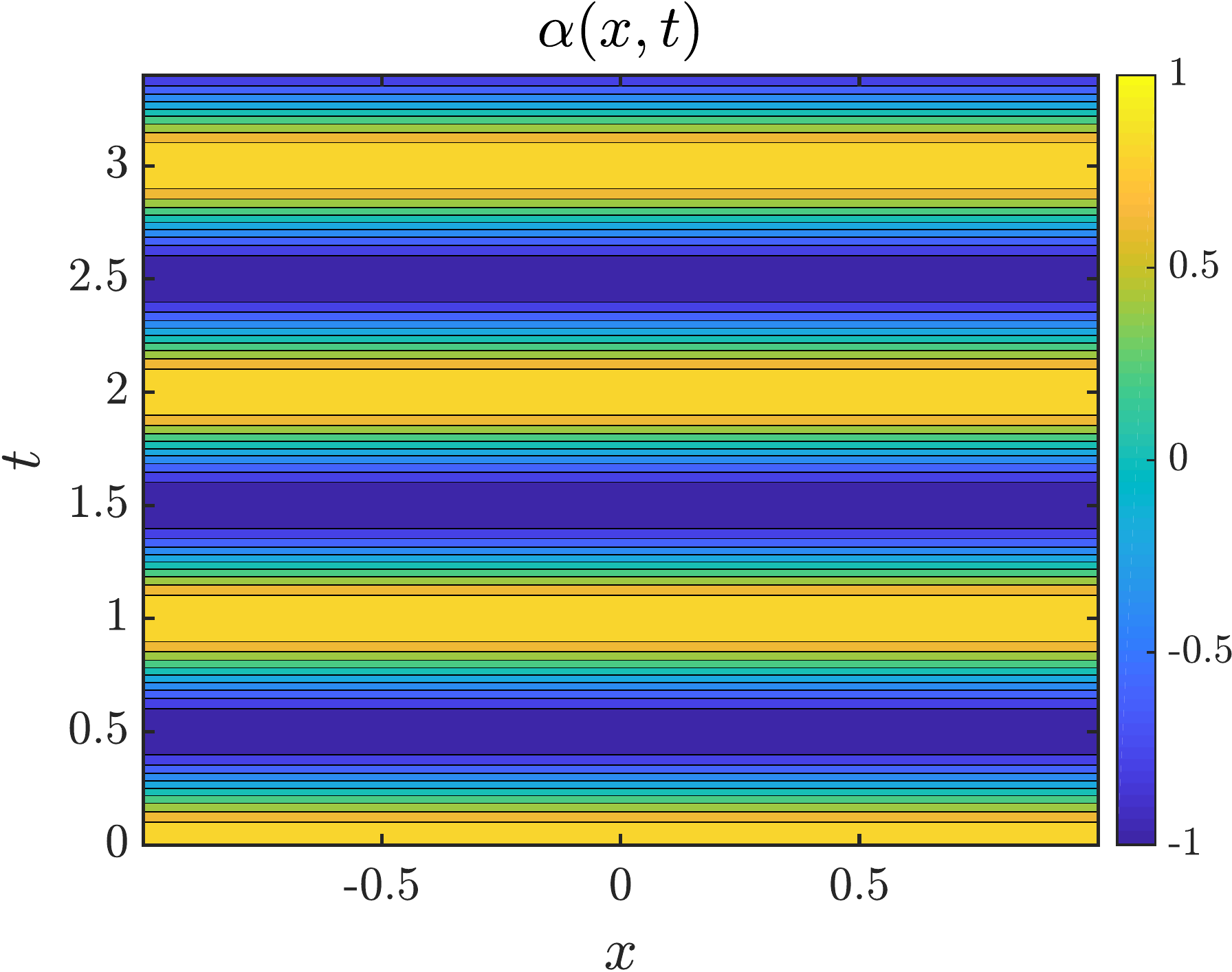}};
\end{tikzpicture}
\begin{tikzpicture}
  		\node at (0,0) {\includegraphics[scale = 0.35]{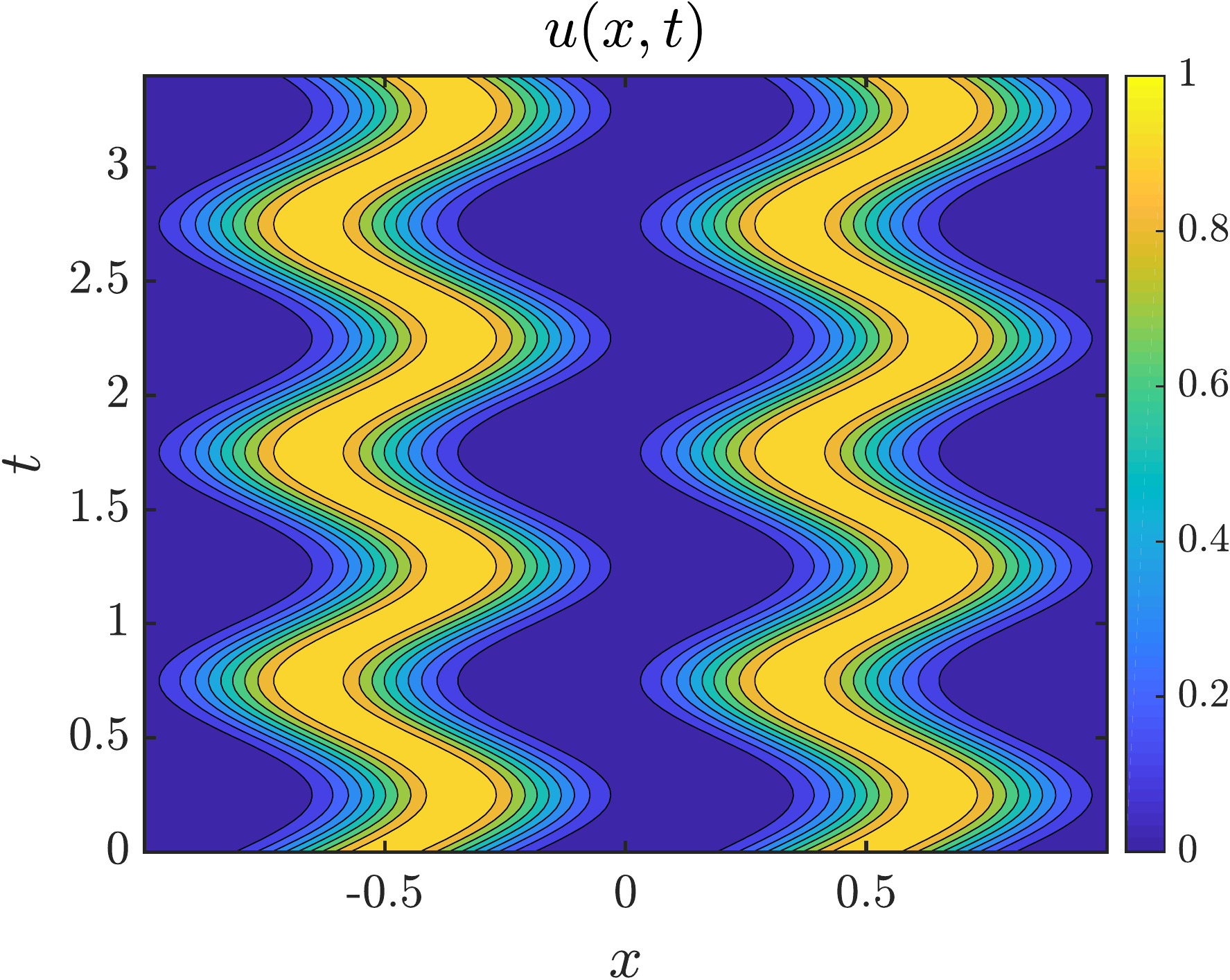}};
\end{tikzpicture}
}
\caption{
Left: Wave-speed function $\alpha(x,t) = \cos(2 \pi t)$, as given in \eqref{eq:alpha2}.
Right: The corresponding space-time solution of the one-dimensional advection problem \eqref{eq:ad}.
Note that the wave-speed function and PDE solution are shown on the truncated time domain $t \in [0, 3.4]$ rather than the full domain $t \in [0, 13.6]$ used in the numerical tests since they are easier to visualize on this shorter domain.
\label{SM:fig:wave_and_sol_1D_alpha2}
}
\end{figure}

\begin{figure}[H]
\centerline{
\begin{tikzpicture}
  		\node at (0,0) {\includegraphics[scale = 0.35]{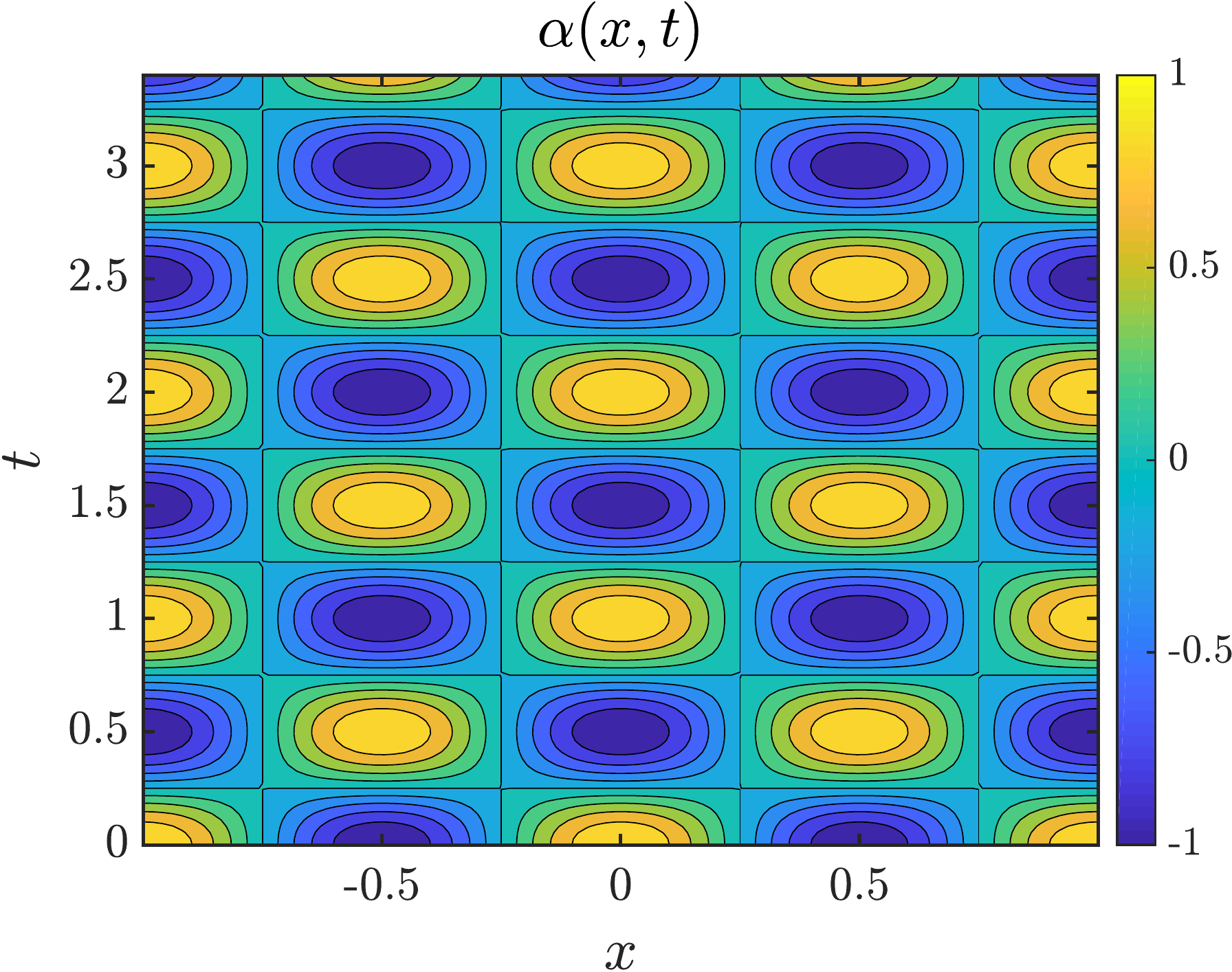}};
\end{tikzpicture}
\begin{tikzpicture}
  		\node at (0,0) {\includegraphics[scale = 0.35]{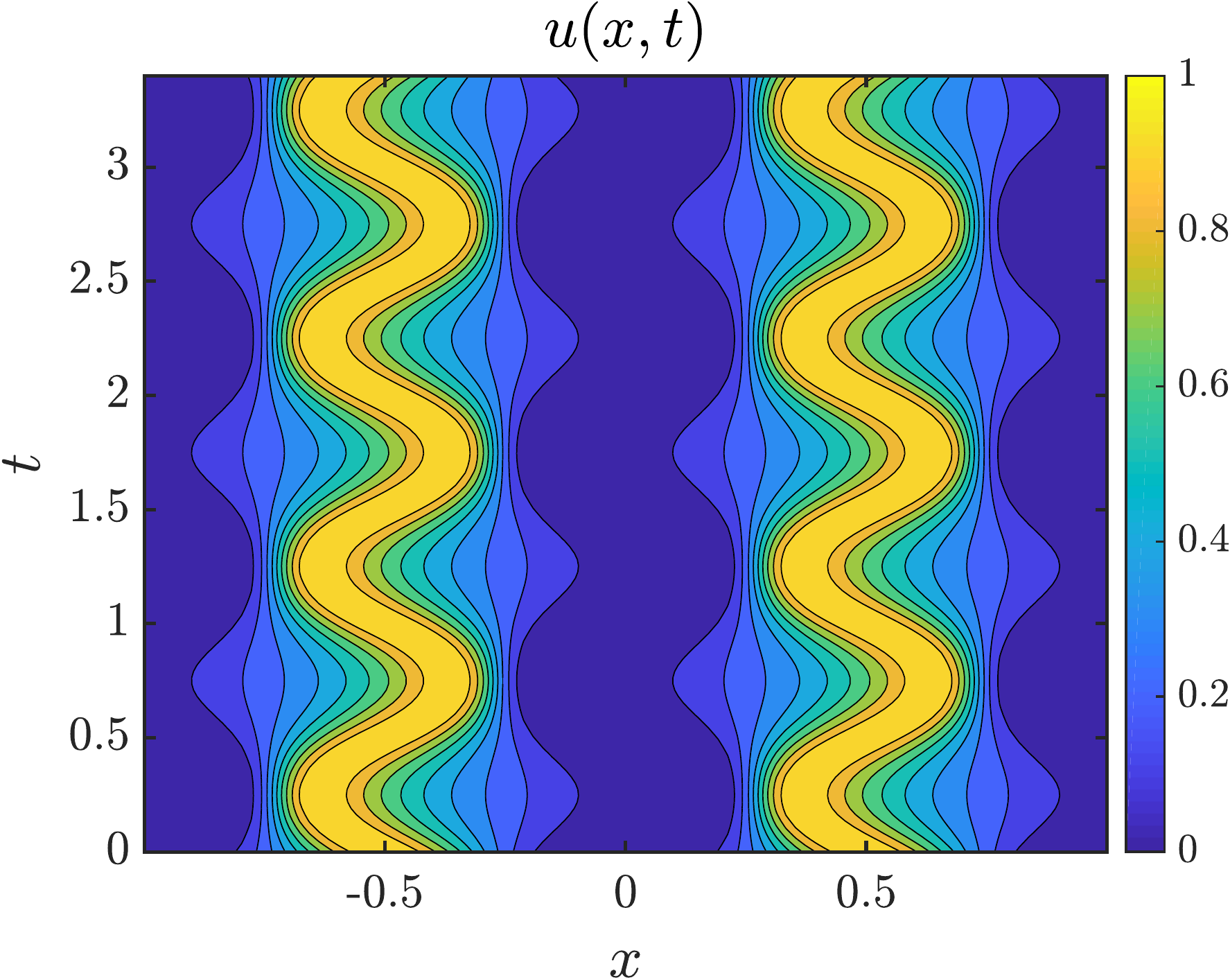}};
\end{tikzpicture}
}
\caption{
Left: Wave-speed function $\alpha(x,t) = \cos(2 \pi t)  \cos (2 \pi x)$, as given in \eqref{eq:alpha4}.
Right: The corresponding space-time solution of the one-dimensional advection problem \eqref{eq:ad}.
Note that the wave-speed function and PDE solution are shown on the truncated time domain $t \in [0, 3.4]$ rather than the full domain $t \in [0, 13.6]$ used in the numerical tests since they are easier to visualize on this shorter domain.
\label{SM:fig:wave_and_sol_1D_alpha4}
}
\end{figure}

\clearpage
\section{One-dimensional strong-scaling tests solving to discretization accuracy}
\label{SM:sec:strong-scaling}

In the left panel of \Cref{SM:fig:strong_scaling} we reproduce the strong-scaling plot from the left panel of \Cref{fig:strong_scaling} for the one-dimensional advection problem. As throughout the rest of the paper, these tests iterate MGRIT until the $\ell^2$-norm of the space-time residual is reduced by at least 10 orders of magnitude from its initial value.
This required 13, 19, and 24 MGRIT iterations for $(p,r) = ((1,1), (3,3), (5,5))$, respectively.
In the right panel of \Cref{SM:fig:strong_scaling} we show strong-scaling results corresponding to iterating MGRIT until the solution at the final time point reaches discretization error accuracy in the discrete $\ell^2$-norm. 
The exact PDE solution needed to measure the discretization error is given in \cite[App. B.4]{KrzysikThesis2021}.
For $(p,r) = ((1,1), (3,3), (5,5))$, the discretization error at the final time is approximately $0.75, 1.0\times 10^{-5}$, and $8.6 \times 10^{-10}$, respectively, and requires 5, 12, and 25 MGRIT iterations to reach this.
Since solving to discretization error accuracy requires fewer iterations than reducing the residual norm by 10 orders of magnitude for the 1st- and 3rd-order discretizations (i.e., reducing the residual norm by 10 order of magnitude \textit{over solves} these problems), the speed-ups are slightly better.

%
%
%
\begin{figure}[h!]
\centerline{
\includegraphics[scale=0.345]{figures/strong_scaling_tol-1e-10}
\quad
\includegraphics[scale=0.345]{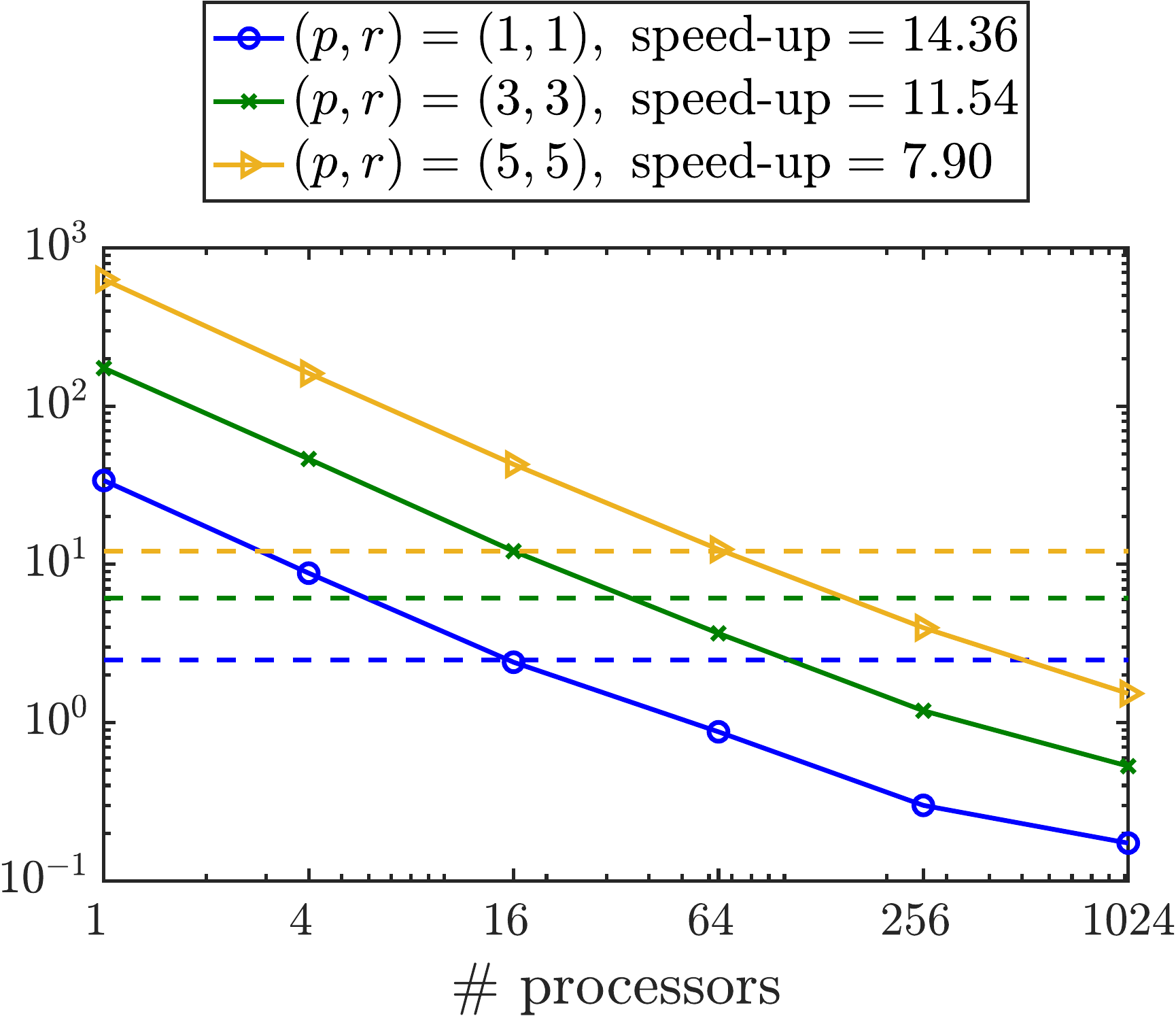}
}
\caption{Strong parallel scaling using time-only parallelism: Runtimes of MGRIT V-cycles with an aggressive coarsening factor of $m= 16$ on the first level followed by $m = 4$ on all coarser levels.
The one-dimensional PDE \eqref{eq:ad} uses wave-speed \eqref{eq:alpha4}, and the semi-Lagrangian discretizations of orders $(p,r) = ((1,1), (3,3), (5,5))$ use a space-time mesh having $n_x \times n_t = 2^{12} \times 2^{14}$ points.
Left: Fixed space-time residual stopping tolerance of $10^{-10}$. Right: Residual stopping tolerance based on reaching discretization error at the final time point.
Dashed lines represent runtimes of time-stepping on one processor.
Speed-ups obtained using 1024 processors are listed in the legends.
\label{SM:fig:strong_scaling}
}
\end{figure}


\clearpage
\section{Two-dimensional numerical test problem}
\label{SM:sec:two-dim-num-tests}

Plots of the velocity field associated with the variable-wave-speed function \eqref{eq:alpha-var_2D} are show in \Cref{SM:fig:velocity_field_2D} for several different times.
In addition, snapshots of the associated PDE solution are shown for several different times in \Cref{SM:fig:sol_evolution_2D}.

\begin{figure}[h!]
\centerline{
\includegraphics[scale = 0.35]{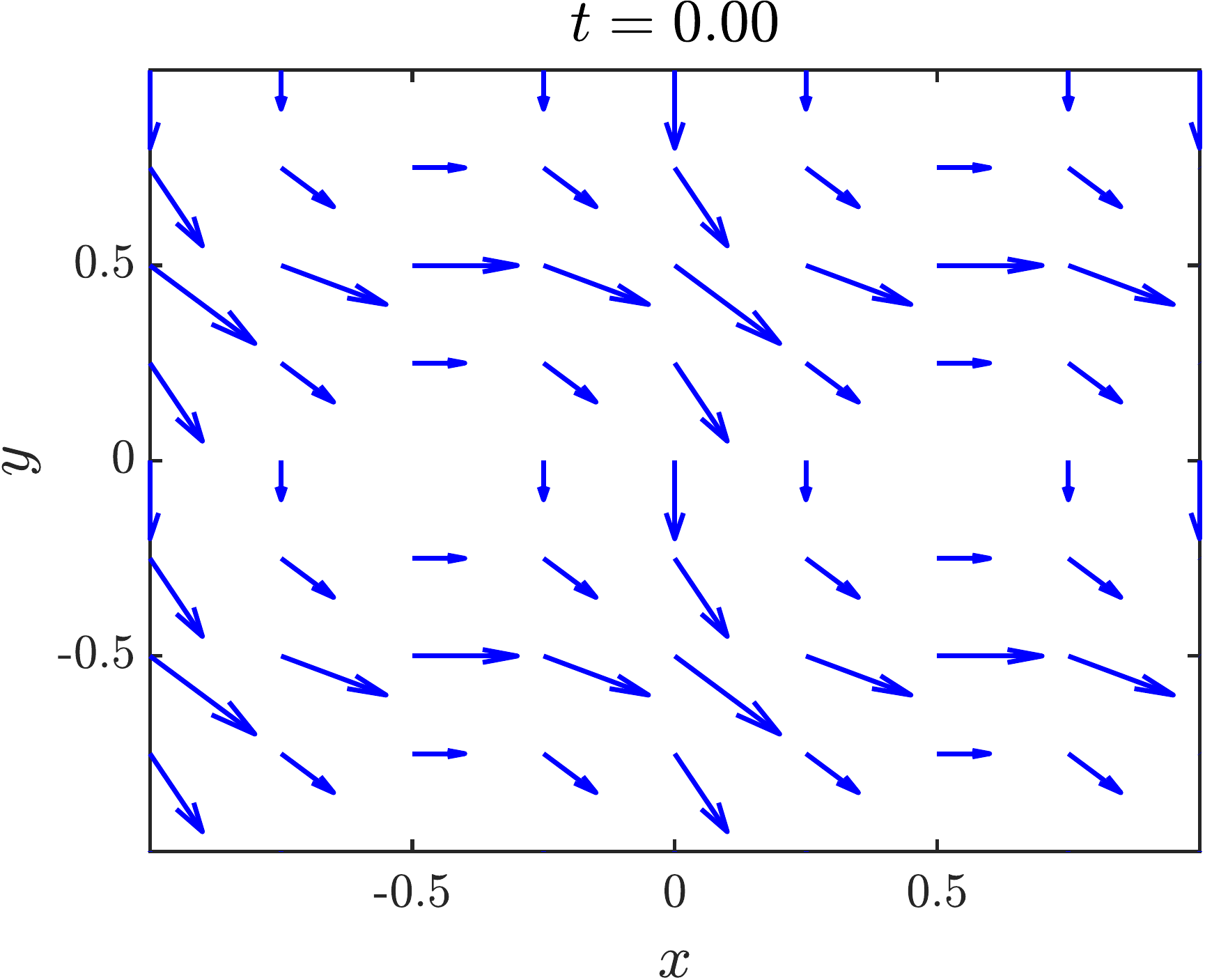}
\;
\includegraphics[scale = 0.35]{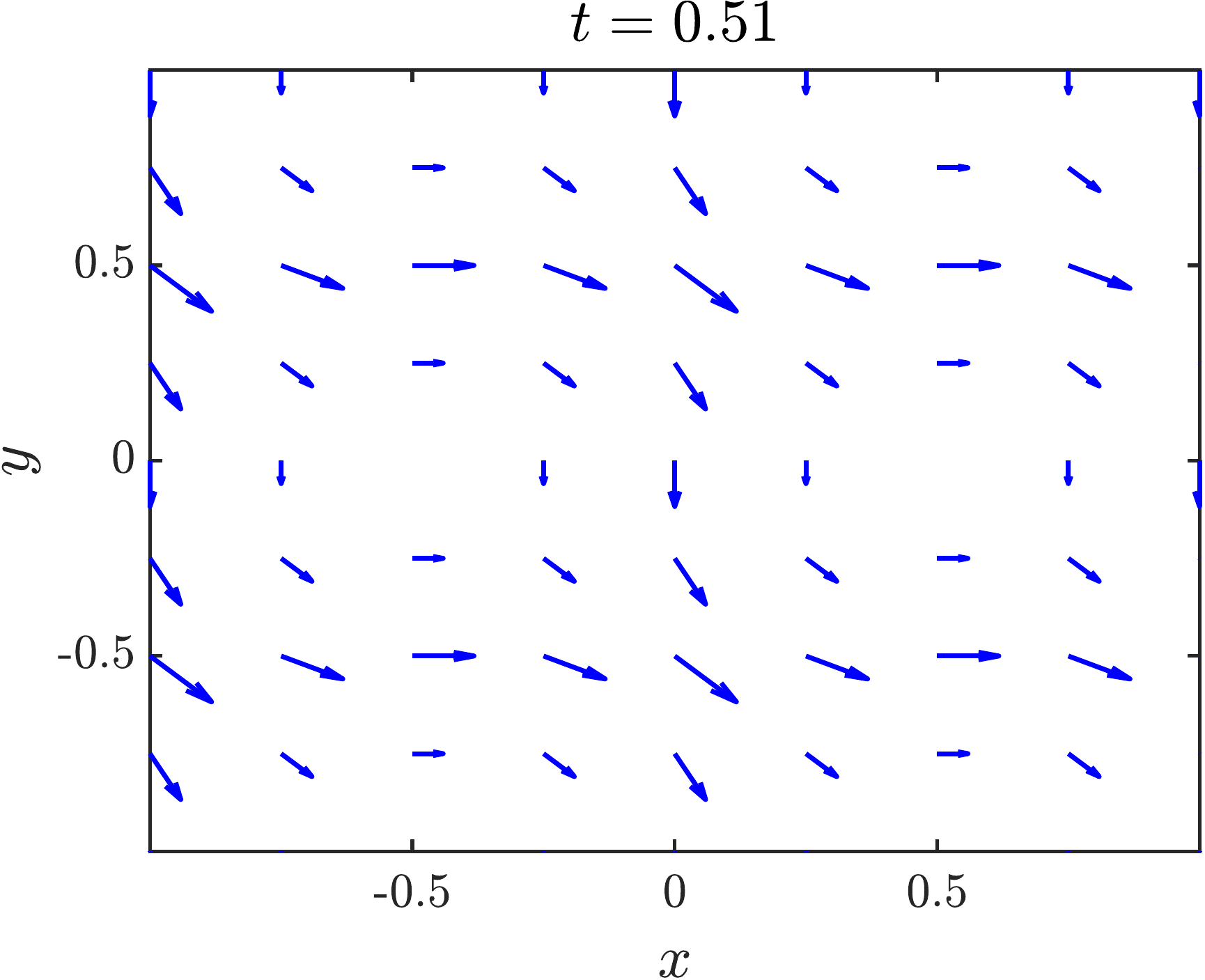}
}
\vspace{2ex}
\centerline{
\includegraphics[scale = 0.35]{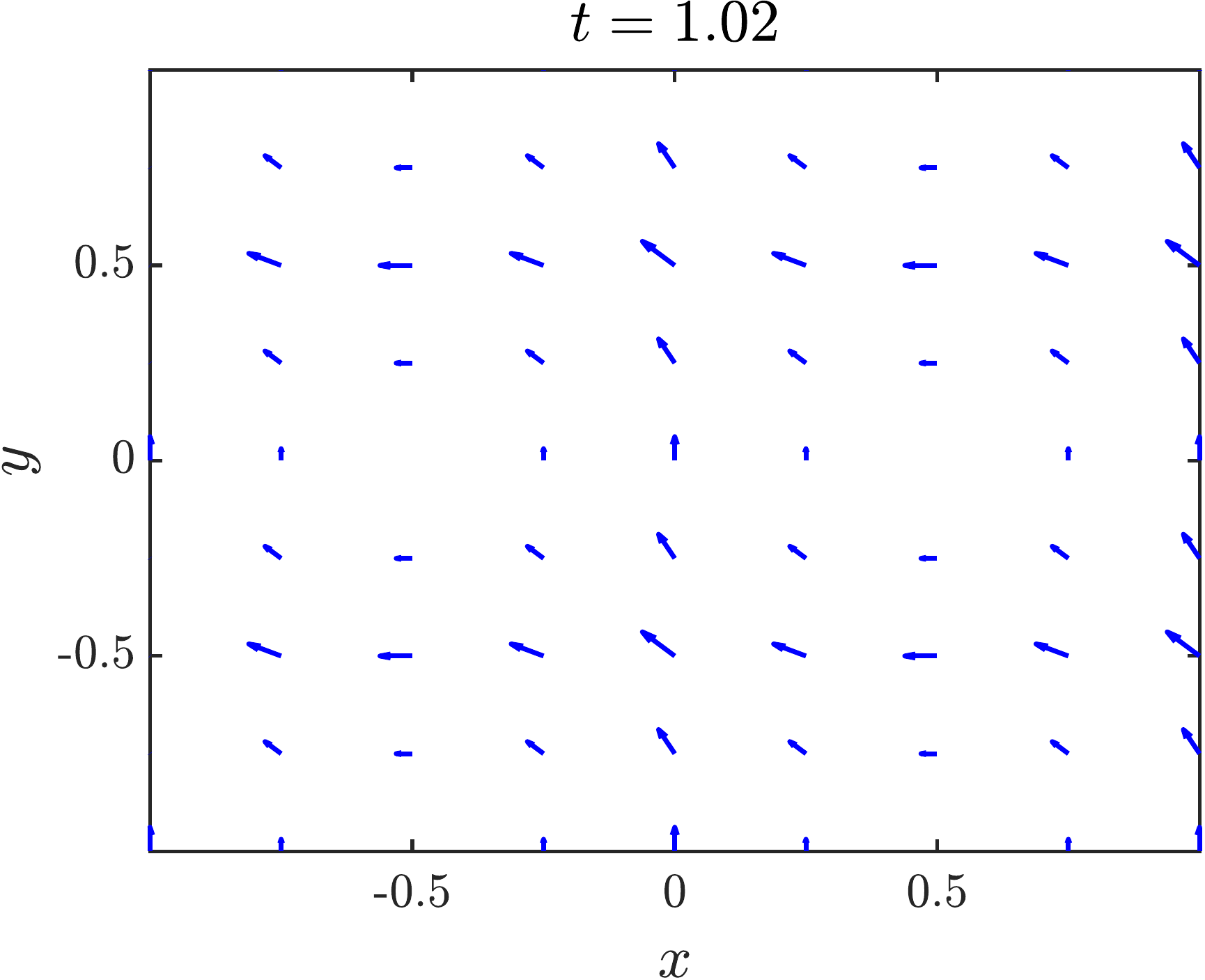}
\;
\includegraphics[scale = 0.35]{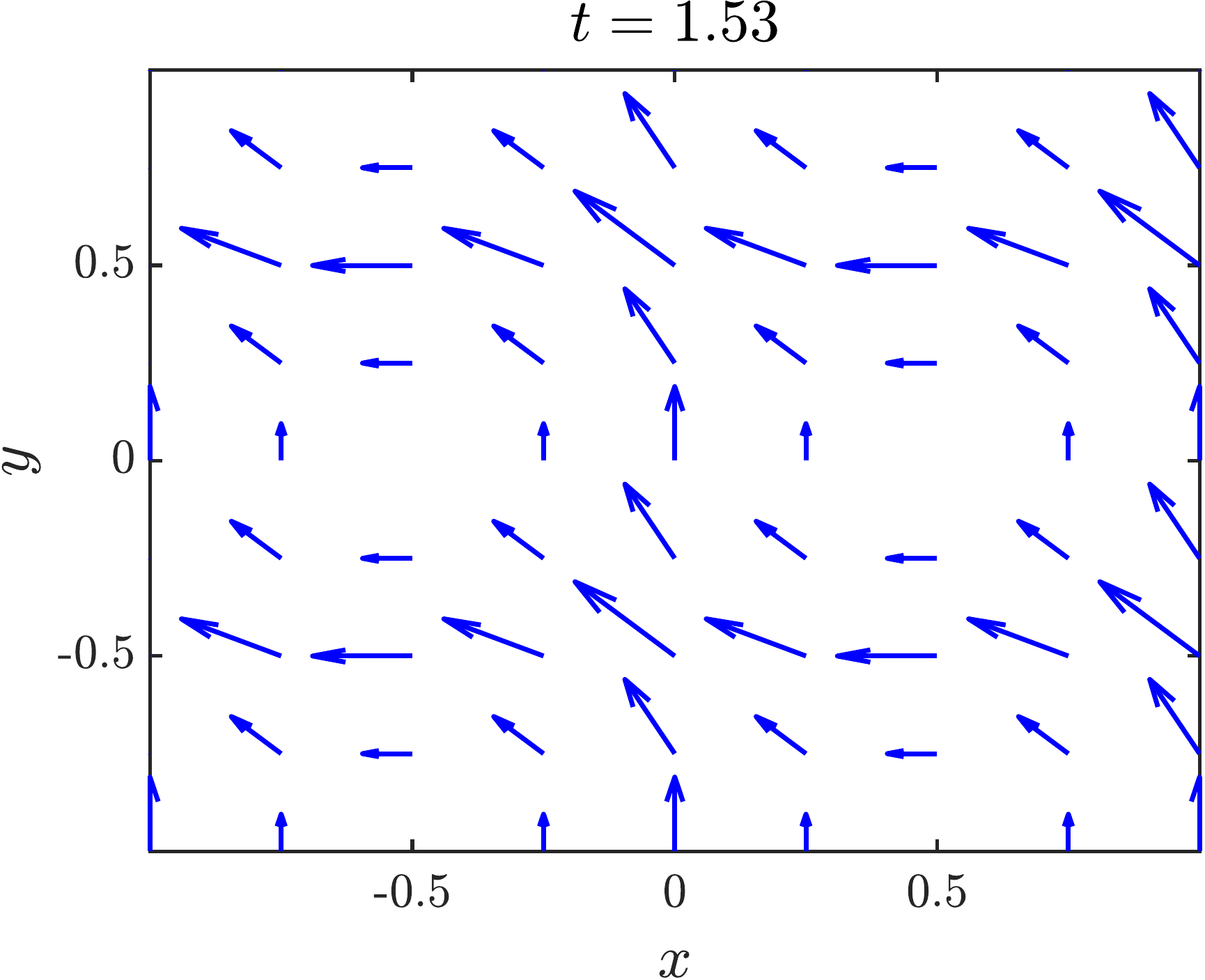}
}
\caption{
Plotted at four different times $t = (0, 0.51, 1.02, 1.53)$ is the velocity field $\big(\alpha(x,y,t),  \beta(x,y,t)\big)$ of the two-dimensional advection problem \eqref{eq:ad_2D} for wave-speed given by \eqref{eq:alpha-var_2D}.
The four times $t = (0, 0.51, 1.02, 1.53)$ represent 0\%, 15\%, 30\%, and 45\% of one period of the time-periodic velocity field, which has a period of 3.4.
\label{SM:fig:velocity_field_2D}
}
\end{figure}

\begin{figure}[t!]
\centerline{
\includegraphics[scale = 0.4]{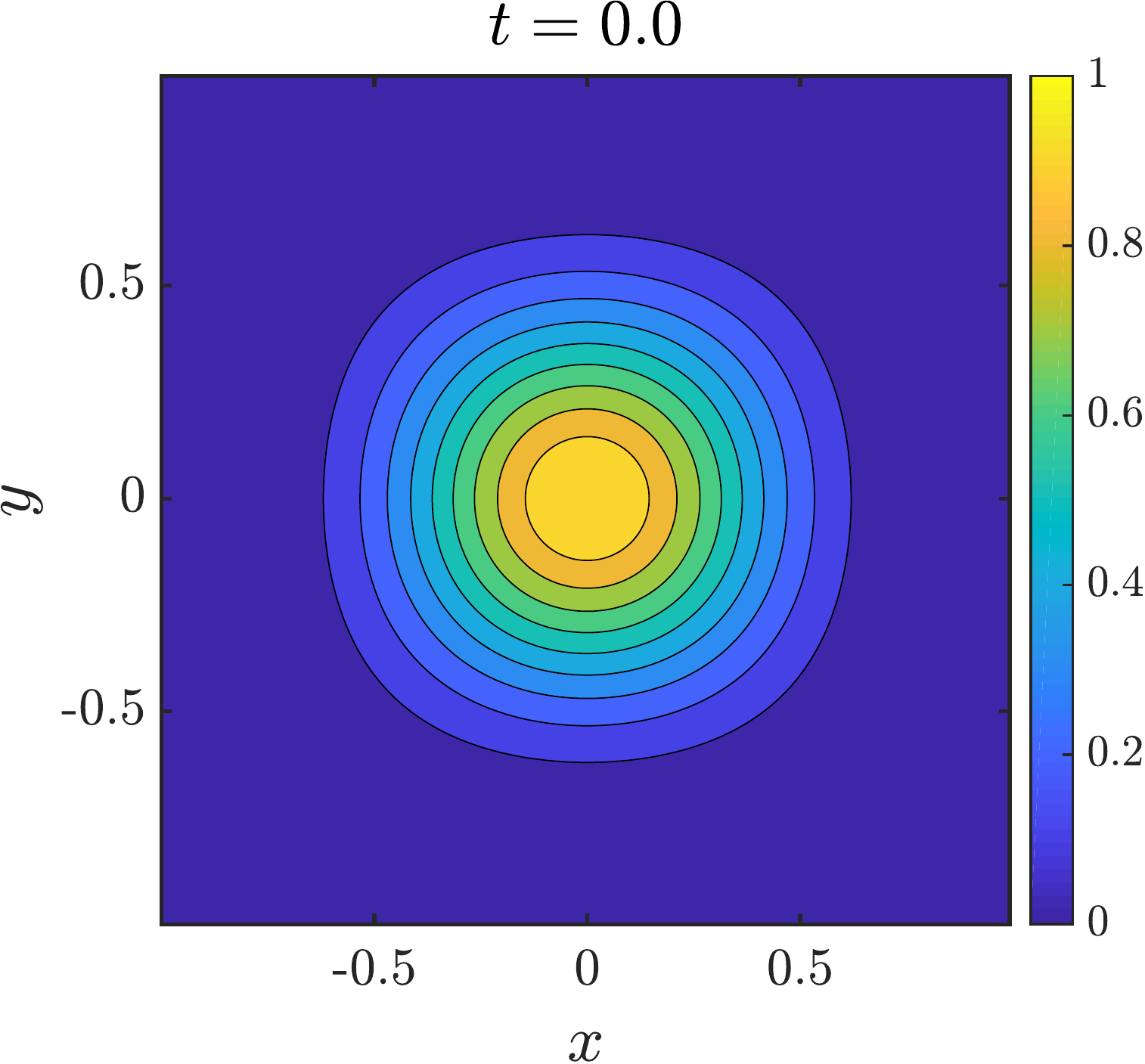}
\quad
\includegraphics[scale = 0.4]{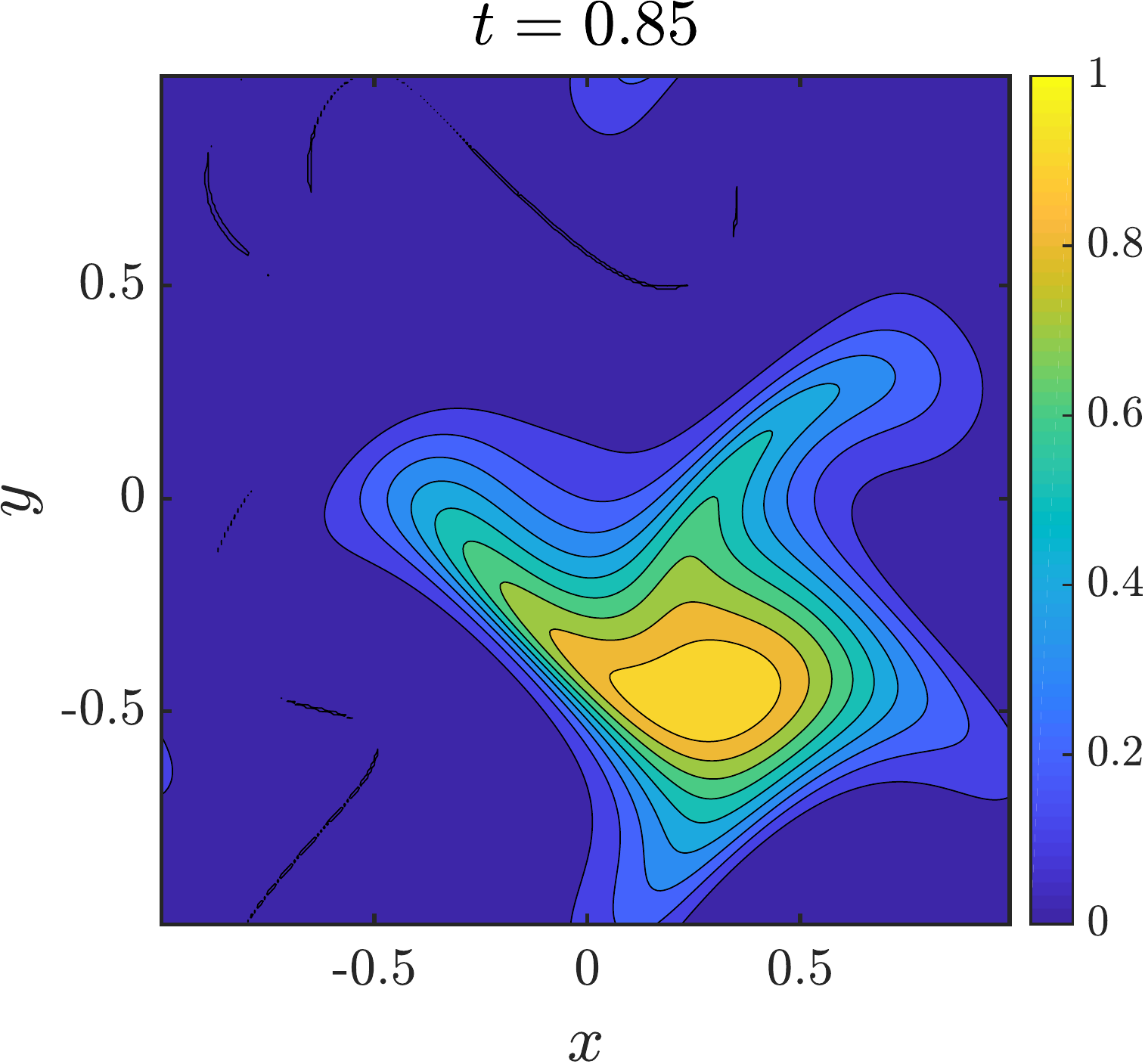}
}
\vspace{2ex}
\centerline{
\includegraphics[scale = 0.4]{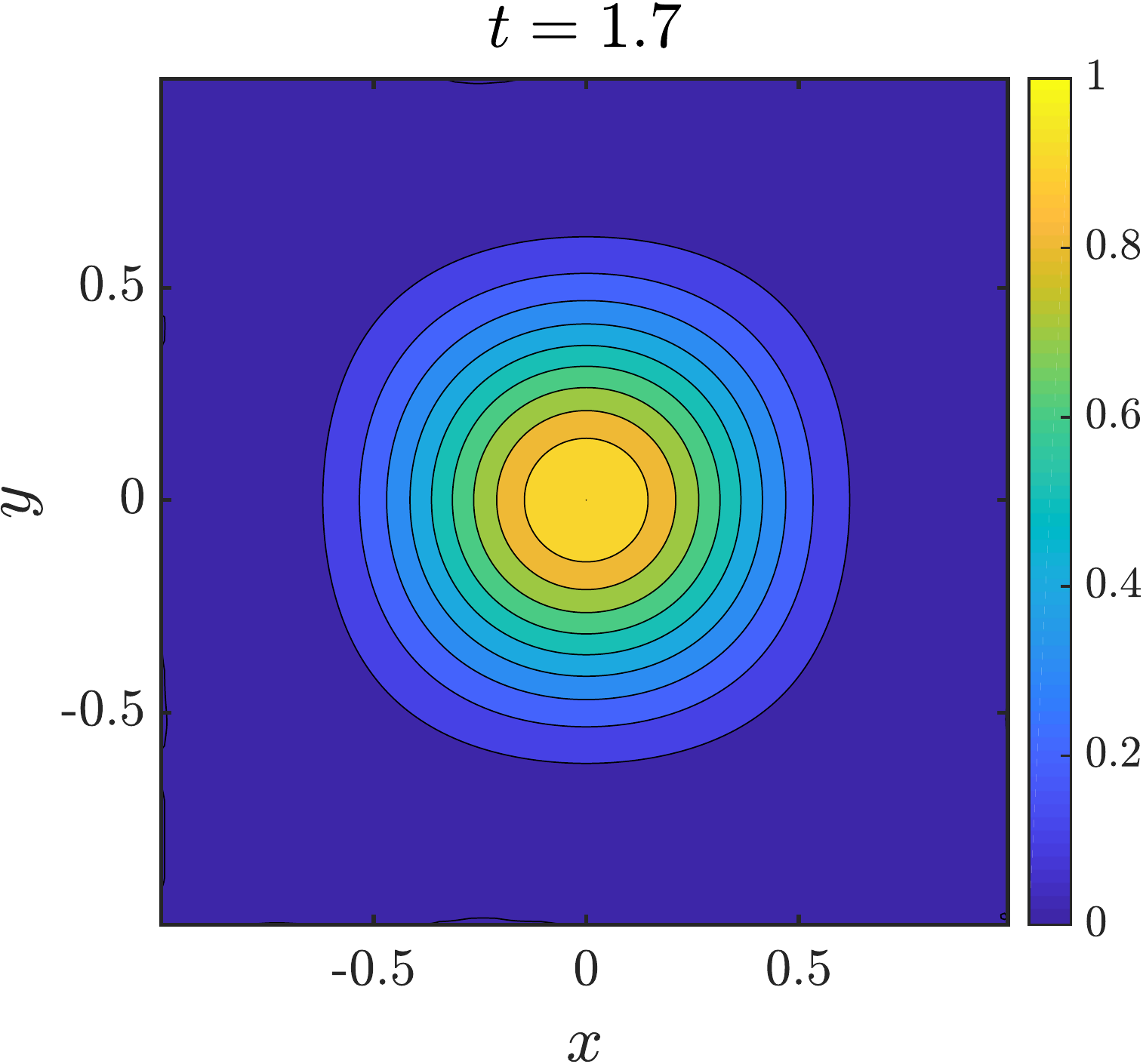}
\quad
\includegraphics[scale = 0.4]{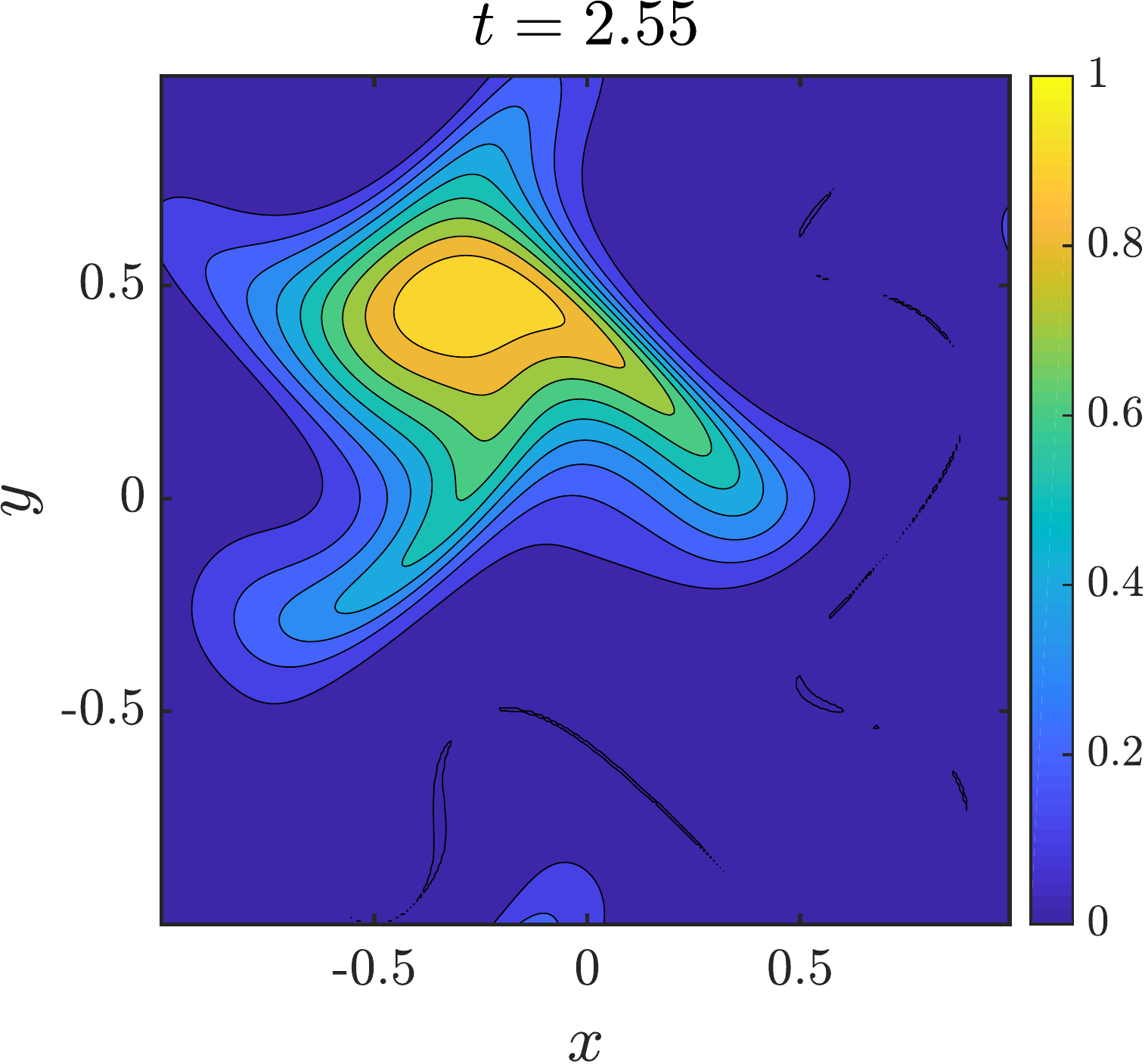}
}
\caption{
One period of the time-periodic solution of the two-dimensional advection problem \eqref{eq:ad_2D} with wave-speed given by \eqref{eq:alpha-var_2D}. 
Moving left to right, top to bottom, the solution is shown at the times $t = (0, 0.85, 1.70, 2.55, 3.4)$, representing 0\%, 25\%, 50\%, and 75\% of one period, respectively.
Note: These snapshots are taken from a simulation that used a 5th-order discretization with ${n_x^2 = 256^2}$ points in space.
\label{SM:fig:sol_evolution_2D}
}
\end{figure}

\newpage
\section{Estimating coarse-grid departure points in two dimensions via backtracking and linear interpolation}
\label{SM:sec:depart_point_estimation}

In this section, we propose a strategy for approximating coarse-grid departure points in two dimensions by a backtracking and linear interpolation procedure, generalizing the one-dimensional procedure proposed in \Cref{sec:backtracking}.
In \Cref{SM:sec:backtrack_poly} we first present the bilinear interpolating polynomial that is used in this procedure, then \Cref{SM:sec:backtrack_derivation} presents the details of procedure.

\subsection{A two-dimensional interpolating polynomial}
\label{SM:sec:backtrack_poly}

In one spatial dimension, define the linear interpolating polynomial $p^{\rm 1d}(x; q_{\rm E}, q_{\rm W})$ as that which interpolates the two values $q_{\rm E}$ and $q_{\rm W}$ associated with the two mesh points $x =x_{\rm W}$, and $x= x_{\rm E}$, respectively, with $x_{\rm E} = x_{\rm W} + h$. 
For example, $q_{\rm E}$ and $q_{\rm W}$ could represent the evaluation of some function $q(x)$ at $x= x_{\rm E}$ and $x = x_{\rm W}$.
This interpolating polynomial can be written as
\begin{align} \label{SM:eq:poly_interp_1d}
p^{\rm 1d}(x; q_{\rm E}, q_{\rm W}) = \frac{q_{\rm E} - q_{\rm W}}{h}(x - x_{\rm E}) + q_{\rm E}, \quad x \in [x_{\rm W}, x_{\rm E}].
\end{align}

Using the one-dimensional polynomial \eqref{SM:eq:poly_interp_1d} we now construct a two-dimensional interpolating polynomial. 
Let $p^{\rm 2d}(x,y; r_{\rm NE}, r_{\rm NW}, r_{\rm SE}, r_{\rm SW})$ denote the bilinear function that interpolates the four values $r_{\rm NE}, r_{\rm NW}, r_{\rm SE}, r_{\rm SW}$ that are associated with the four mesh points
$(x,y) = \big( (x_{\rm E}, y_{\rm N}), (x_{\rm W}, y_{\rm N}), (x_{\rm E}, y_{\rm S}), (x_{\rm W}, y_{\rm S}) \big)$, respectively, with $y_{\rm N} = x_{\rm S} + h$.
In terms of \eqref{SM:eq:poly_interp_1d}, this function can be written as
\begin{align} \label{SM:eq:poly_interp_2d}
\begin{split}
&p^{\rm 2d}(x,y; r_{\rm NE}, r_{\rm NW}, r_{\rm SE}, r_{\rm SW})
\\
&\quad=\frac{
p^{\rm 1d}(x; r_{\rm NE}, r_{\rm NW})
-
p^{\rm 1d}(x; r_{\rm SE}, r_{\rm SW})
}{
h
}(y - y_{\rm N}) 
+ 
p^{\rm 1d}(x; r_{\rm NE}, r_{\rm NW}),
\end{split}
\end{align}
with $(x, y) \in [x_{\rm W} , x_{\rm E}] \times [y_{\rm S}, y_{\rm N}]$.
%

\subsection{Coarse-grid departure point estimation}
\label{SM:sec:backtrack_derivation}

To begin, it is helpful to recall the notation used in \Cref{sec:2D} to describe a local characteristic. 
Recall that the characteristic 
\begin{align}
(x(t), y(t)) = \big(\xi_{ij}^{(t_n, \delta t)}(t), \eta_{ij}^{(t_n, \delta t)}(t) \big),
\end{align}
is defined over the interval $t \in [t_n, t_n \delta t]$. The \textit{arrival} point of this characteristic is $\big(\xi_{ij}^{(t_n, \delta t)}(t_n + \delta t), \eta_{ij}^{(t_n, \delta t)}(t_n + \delta t) \big) = (x_i, y_j)$, and the departure point is its location at time $t = t_n$.

As in \Cref{sec:backtracking}, for notational simplicity, suppose that we are working on the first coarse-grid time interval, $t \in [0, m \delta t]$, but note that the following strategy can be applied on any coarse-grid time interval $t \in [t_n, t_n + m \delta t]$.
Consider the local coarse-grid characteristic
\begin{align} \label{SM:eq:coarse_char_concern_2D}
(x(t), y(t)) = \big(\xi_{ij}^{(0, m \delta t)}(t), \eta_{ij}^{(0, m \delta t)}(t) \big)
\end{align}
that arrives at the mesh point $(x,y) = (x_i, y_j)$ at time $t = m \delta t$.
Our objective is to estimate the departure point of this characteristic: $\big(\xi_{ij}^{(0, m \delta t)}(0), \eta_{ij}^{(0, m \delta t)}(0) \big)$.

Generalizing our one-dimensional strategy from \Cref{sec:backtracking}, we propose to estimate this departure point by approximately tracking the path of the characteristic \eqref{SM:eq:coarse_char_concern_2D} backwards from $t = m \delta t \to t= 0$. 
More specifically, we do so by using the local fine-grid characteristic directions that have already been computed on fine-grid subintervals $t \in [k \delta t, (k+1) \delta t], k = m-1, \ldots, 0$ to guide its path in an interpolating manner.
Let our approximation to the coarse-grid characteristic at time $t = k \delta t$ be denoted by
\begin{align}
\big(
c^{(k)}_{ij}, d^{(k)}_{ij}
\big) 
\approx 
\big(\xi_{ij}^{(0, m \delta t)}(k \delta t), \eta_{ij}^{(0, m \delta t)}(k \delta t) \big).
\end{align}
The approximate departure point is therefore denoted by $\big( c^{(0)}_{ij}, d^{(0)}_{ij} \big)$.

Since \eqref{SM:eq:coarse_char_concern_2D} arrives at $(x, y) = (x_i, y_j)$ at time $t = m \delta t$, over the last fine-grid subinterval ${t \in [(m-1) \delta t, m \delta t]}$ it is the same as the local fine-grid characteristic that also arrives at $(x, y) = (x_i, y_j)$ at time $t = m \delta t$. 
As such, we use the departure point of this fine-grid characteristic as a final-time condition to initialize our approximation:
\begin{align}
\big(
c^{(m-1)}_{ij}, d^{(m-1)}_{ij}
\big) 
=
\big(
\xi_{ij}^{((m-1) \delta t, \delta t)}((m-1) \delta t), \eta_{ij}^{((m-1) \delta t, \delta t)}((m-1) \delta t) 
\big).
\end{align}

\begin{figure}[b!]
\centerline{
\includegraphics[scale=1]{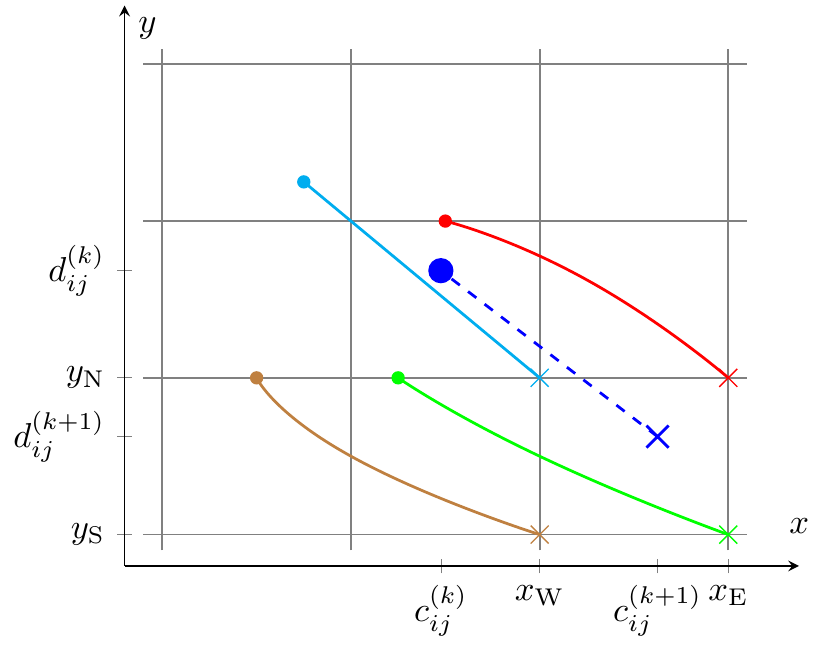}
}
\caption{
A bird's eye view of the four local fine-grid characteristics \eqref{SM:eq:char_2d_NE}, \eqref{SM:eq:char_2d_NW}, \eqref{SM:eq:char_2d_SE}, \eqref{SM:eq:char_2d_SW} (the solid red, cyan, green, and brown lines, respectively). 
These characteristics arrive at time $t = (k+1) \delta t$ at the mesh points (small crosses) that are the closest neighbors of the point $\big( c^{(k+1)}_{ij}, d^{(k+1)}_{ij} \big)$.
The departure points of these local fine-grid characteristics at time $t = k \delta t$ are marked as small circles. 
The point $\big( c^{(k+1)}_{ij}, d^{(k+1)}_{ij} \big)$ (large blue cross) is the approximate location of the coarse-grid characteristic \eqref{SM:eq:coarse_char_concern_2D} at time $t = (k+1) \delta t$.
As described in \Cref{SM:sec:backtrack_derivation}, the unknown point $\big( c^{(k)}_{ij}, d^{(k)}_{ij} \big)$ (large blue circle) is the approximate location of the coarse-grid characteristic \eqref{SM:eq:coarse_char_concern_2D} at time $t = k \delta t$, and is obtained by fitting a bilinear function to the paths of these neighboring fine-grid characteristics.
\label{SM:fig:char_track_2D}
}
\end{figure}

Given $\big( c^{(k+1)}_{ij}, d^{(k+1)}_{ij} \big)$, for a single $k \in \{ m-2, \ldots, 0 \}$, we now describe how to compute $\big( c^{(k)}_{ij}, d^{(k)}_{ij} \big)$; that is, how to propagate the coarse-grid characteristic from $t = (k+1) \delta t \to t = k \delta t$.
Define $(x,y) =\big( (x_{\rm E}, y_{\rm N}), (x_{\rm W}, y_{\rm N}), (x_{\rm E}, y_{\rm S}), (x_{\rm W}, y_{\rm S}) \big)$ as the four mesh points that are immediately to the north-east, north-west, south-east, and south-west, respectively, of $(x,y) = \big(c^{(k+1)}_{ij}, d^{(k+1)}_{ij} \big)$.
See the schematic in \Cref{SM:fig:char_track_2D}.
Furthermore, denote the four local fine-grid characteristics that \textit{arrive} at these four mesh points by
\begin{align} 
\label{SM:eq:char_2d_NE}
(x(t), y(t)) &= \big(\xi_{\rm NE}^{(k \delta t,  \delta t)}(t), \eta_{\rm NE}^{(k \delta t, \delta t)}(t) \big),
\\
\label{SM:eq:char_2d_NW}
(x(t), y(t)) &= \big(\xi_{\rm NW}^{(k \delta t,  \delta t)}(t), \eta_{\rm NW}^{(k \delta t, \delta t)}(t) \big),
\\
\label{SM:eq:char_2d_SE}
(x(t), y(t)) &= \big(\xi_{\rm SE}^{(k \delta t,  \delta t)}(t), \eta_{\rm SE}^{(k \delta t, \delta t)}(t) \big), 
\\
\label{SM:eq:char_2d_SW}
(x(t), y(t)) &= \big(\xi_{\rm SW}^{(k \delta t,  \delta t)}(t), \eta_{\rm SW}^{(k \delta t, \delta t)}(t) \big),
\end{align} 
respectively. 
Then, to approximate the $x$-component of the coarse-grid characteristic at time $t = k \delta t$, we simply fit a bilinear function to how the fine-grid characteristics \eqref{SM:eq:char_2d_NE}--\eqref{SM:eq:char_2d_SW} map the $x$-component of their arrival point into the $x$-component of their departure point:
\begin{align}
\begin{split}
&c^{(k)}_{ij}
=
p^{\rm 2d}\big( c^{(k+1)}_{ij}, d^{(k+1)}_{ij}; 
\\
&\quad
\xi_{\rm NE}^{(k \delta t,  \delta t)}(k \delta t),
\xi_{\rm NW}^{(k \delta t,  \delta t)}(k \delta t),
\xi_{\rm SE}^{(k \delta t,  \delta t)}(k \delta t),
\xi_{\rm SW}^{(k \delta t,  \delta t)}(k \delta t) 
\big).
\end{split}
\end{align}
To estimate the $y$-component of the coarse-grid characteristic at time $t = k \delta t$, we carry out the analogous procedure of fitting a bilinear function to how the fine-grid characteristics \eqref{SM:eq:char_2d_NE}--\eqref{SM:eq:char_2d_SW} map the $y$-component of their arrival point into the $y$-component of their departure point:
\begin{align}
\begin{split}
&d^{(k)}_{ij}
=
p^{\rm 2d}\big( c^{(k+1)}_{ij}, d^{(k+1)}_{ij}; 
\\
&\quad
\eta_{\rm NE}^{(k \delta t,  \delta t)}(k \delta t),
\eta_{\rm NW}^{(k \delta t,  \delta t)}(k \delta t),
\eta_{\rm SE}^{(k \delta t,  \delta t)}(k \delta t),
\eta_{\rm SW}^{(k \delta t,  \delta t)}(k \delta t) 
\big).
\end{split}
\end{align}


\end{document}